\documentclass[oneside,11pt]{amsart}
\usepackage{amsmath}
\usepackage{amssymb}
\usepackage{graphicx}
\usepackage{caption}
\usepackage{subcaption}
\usepackage{pinlabel}
\usepackage{tikz}
\usepackage[colorlinks=true]{hyperref}
\usepackage[margin=1in]{geometry}
\usepackage{soul}
\usepackage{url}

\newtheorem{thm}{Theorem}[section]
\newtheorem{prop}[thm]{Proposition}
\newtheorem{lem}[thm]{Lemma}
\newtheorem{cor}[thm]{Corollary}

\newtheorem*{claim}{Claim}
\theoremstyle{definition}
\newtheorem{defn}[thm]{Definition}
\newtheorem{rem}[thm]{Remark}
\newtheorem{conv}[thm]{Convention}
\newtheorem{exmp}[thm]{Example}

\renewcommand{\bar}[1]{\overline{#1}}

\renewcommand{\emptyset}{\varnothing}

\renewcommand{\setminus}{-}

\newcommand{\field}[1]{\mathbb{#1}}
\newcommand{\Z}{\field{Z}}

\newcommand{\R}{\field{R}}

\newcommand{\PP}{\field{P}}

\DeclareMathOperator{\Isom}{Isom}

\DeclareMathOperator{\Lk}{Lk}

\DeclareMathOperator{\diam}{diam}
\DeclareMathOperator{\Cay}{Cay}

\DeclareMathOperator{\ver}{ver}
\DeclareMathOperator{\hor}{hor}

\usepackage{ifthen}

\newcommand{\showcomments}{yes}

\newsavebox{\commentbox}
{\ifthenelse{\equal{\showcomments}{yes}}%
{\footnotemark
    \begin{lrbox}{\commentbox}
    \begin{minipage}[t]{1.25in}\raggedright\sffamily\upshape\tiny
    \footnotemark[\arabic{footnote}]}%
{\begin{lrbox}{\commentbox}}}%
{\ifthenelse{\equal{\showcomments}{yes}}%
{\end{minipage}\end{lrbox}\marginpar{\usebox{\commentbox}}}%
{\end{lrbox}}}

\title{Property (QT) for 3-manifold groups}

\author{Suzhen Han}
\address{Beijing International Center for Mathematical Research\\
Peking University\\
 Beijing 100871, China
P.R.}
\email{suzhenhan@pku.edu.cn}

\author{Hoang Thanh Nguyen}
\address{Department of Mathematics\\
The University of Danang - University of Science and
Education\\
 459 Ton Duc Thang, Da Nang, Vietnam \& International Centre for Research and Postgraduate Training in Mathematics (ICRTM), Institute of Mathematics, VAST, 18 Hoang Quoc Viet Road,
Cau Giay District, Hanoi, Vietnam}
\email{nthoang.math@gmail.com}

\author{Wenyuan Yang}
\address{Beijing International Center for Mathematical Research\\
Peking University\\
 Beijing 100871, China
P.R.}
\email{wyang@math.pku.edu.cn}

\date{\today}
\begin{document}

\begin{abstract}
According to Bestvina-Bromberg-Fujiwara, a finitely generated group is said to have property (QT) if it acts isometrically on a finite product of quasi-trees so that orbital maps are quasi-isometric embeddings. We prove that the fundamental group $\pi_1(M)$ of a compact, connected, orientable 3-manifold $M$ has property (QT) if and only if no summand in  the sphere-disc decomposition of $M$ supports either Sol or Nil geometry. In particular, all compact, orientable, irreducible 3-manifold groups with nontrivial torus decomposition and not supporting Sol geometry have  property (QT).
In the course of our study, we establish property (QT) for the class of Croke-Kleiner admissible groups and of  relatively hyperbolic groups under natural assumptions
has property (QT).
\end{abstract}
\maketitle


\section{Introduction}

\subsection{Background and Motivation}
The study of group actions on quasi-trees has   recently received a great deal of interest. A \emph{quasi-tree} means here a possibly locally infinite {connected} graph that is quasi-isometric to a simplicial tree. Groups acting on (simplicial) trees have been well-understood thanks to the Bass-Serre theory. On the one hand, quasi-trees have the obvious advantage of being more flexible; hence,  many groups can act unboundedly on quasi-trees but act on any trees with global fixed points. Many hyperbolic groups with Kazhdan's property (T), mapping class groups among many other examples belong to this category (see \cite{M05, M06} for other examples). In effect, these are sample applications of a powerful axiomatic construction of quasi-trees proposed in the work of Bestvina, Bromberg and Fujiwara \cite{BBF}. This construction will be fundamental in this paper.


We say that a finitely generated group G has \emph{property (QT)} if it acts isometrically
on a finite product $X=T_1\times T_2\times \cdots \times T_n$  of quasi-trees {with $L^2$-metric} so that for any basepoint $o\in X$, the induced orbit map
$$
g\in G\longmapsto go\in X
$$
is a quasi-isometric embedding of $G$ equipped with some (or any) word metric $d_G$ to $X$. 
Informally speaking, property (QT) gives an undistorted picture of the group under consideration in a   reasonably good space. Here, the direct product structure usually comes from the independence of several negatively curved layers endowed on the  group. Such a hierarchy structure has emerged from the study of mapping class groups since Masur-Minsky \cite{MM00}. In addition, property (QT)  is a commensurability invariant in \cite{BBF2, Button}  and could be thought of as a stronger property than the finiteness of asymptotic dimension.

Extending their earlier results of \cite{BBF}, Bestvina, Bromberg and Fujiwara \cite{BBF2} recently showed that residually finite hyperbolic groups and mapping class groups have property (QT).  It is known that Coxeter groups have property (QT) (see \cite{DJ99}), and thus every right-angled Artin group has property (QT) (see  \cite[Induction~2.2]{BBF2}).


In 3-manifold theory, the study of the fundamental groups of 3-manifolds is one
of the most central topics. Determining property (QT) of  finitely generated 3-manifold groups is the main task of the present paper.

\subsection{Property (QT) of 3-manifold groups}
Let   $M$ be a 3-manifold with finitely generated   fundamental group. Since property (QT) is a commensurability invariant, we can assume that $M$ is compact and orientable by considering the Scott core of $M$ and a double cover of $M$ (if $M$ is non-orientable).

In recent years, the theory of special cube complexes \cite{HW08} has  led to a deep understanding of 3-manifold groups \cite{Wise20} \cite{Agol}.  By definition,  the fundamental group of a compact special cube complex is undistorted in a right-angled Artin group, and then has property (QT) by \cite{DJ99}. However, 3--manifolds without non-positively curved Riemannian metrics cannot be cubulated by \cite{PW18} and certain {{cubulated}} $3$--manifold groups are not virtually cocompact special (see \cite{HP15}, \cite{Tid18}). Thus it was left open to   determine the property (QT) for all 3-manifold groups. 

By  the  sphere-disc decomposition, a compact oriented 3-manifold $M$ is a connected sum of prime summands $M_i$ $(1\le i\le n)$ with incompressible boundary. It is an easy observation that if a group has property (QT) then every non-trivial element is undistorted (see Lemma \ref{QTUndistortedLem}), and hence if $M_i$ supports $Sol$ or $Nil$ from the eight Thurston geometries, then $\pi_1(M_i)$ fails to have property (QT).
Our first main result is the following   characterization of property (QT) for all 3-manifolds. 

\begin{thm}
\label{thm:QT3MFD}
Let $M$ be a   connected,  compact, orientable 3-manifold. Then $\pi_1(M)$ has property (QT) if and only if no summand in its sphere-disk decomposition supports either $Sol$ or $Nil$ geometry.
\end{thm}

By  standard arguments, we are reduced to the case where $M$ is a compact, connected, orientable, irreducible 3-manifold with empty or tori boundary. Theorem~\ref{thm:QT3MFD} actually follows from the following theorem.
\begin{thm}
\label{thm:3-mfd}
Let $M$ be a compact orientable irreducible 3-manifold with
empty or tori boundary, with nontrivial torus decomposition and that does not support the Sol geometry. Then $\pi_1(M)$ has property (QT).
\end{thm}

A 3-manifold $M$ as in Theorem~\ref{thm:3-mfd} is called
a \emph{graph manifold} if all the pieces in its torus decomposition are Seifert fibered spaces; otherwise $M$ is called a
\emph{mixed manifold}. It is well-known that the fundamental group of a mixed 3-manifold is hyperbolic relative to a collection of abelian groups and graph manifolds groups.
To prove Theorem~\ref{thm:3-mfd}, we actually determine the property (QT) of \emph{Croke-Kleiner admissible groups}, and of relatively hyperbolic groups that will be discussed in detail in the following subsections. These results include but are much more general than the fundamental groups of graph manifolds and mixed manifolds.



\subsection{Property (QT) of Croke-Kleiner admissible groups.}



We first address property (QT) of graph manifolds. Our approach is based on a study of a particular class of graph of groups introduced by Croke and Kleiner \cite{CK02}    which they called {\it admissible groups}  and generalized the fundamental groups of  graph manifolds. We say that an admissible group $G$ is a {\it Croke-Kleiner admissible group} or a {\it CKA group} if it acts properly discontinuous, cocompactly and by isometries on a complete proper CAT(0) space $X$. Such action $G \curvearrowright X$ is called a {\it CKA action} and the space $X$ is called a {\it CKA space}.
The CKA action   is modeled on the JSJ structure of graph manifolds where the Seifert fibration is replaced by the following central extension of a general hyperbolic group $H_v$:
\begin{equation}\label{centralExtEQ}
1\to Z(G_v)\to G_v\to H_v\to 1
\end{equation}
where $Z(G_v)=\mathbb Z.$
It is worth pointing out that CKA groups   encompass a much more general class of groups and can be used to produce interesting groups by a ``flip" trick from any finite number of hyperbolic groups (see Example~\ref{exmp:CKA}).


The notion of  an \textit{omnipotent} group  was introduced by Wise in \cite{Wise00} and has found many applications in subgroup separability. We refer the reader to Definition \ref{defn:omnipotent} for its definition and note here that free groups \cite{Wise00}, surface groups \cite{B07}, and the more general class of virtually special hyperbolic groups \cite{Wise20} are omnipotent.
In \cite{NY20}, Nguyen-Yang proved  property (QT) for a special class of CKA actions under flip conditions (see Definition~\ref{defn:flip}). One of the main contributions of this paper is to remove this assumption and prove the following result in   full generality.
\begin{thm}
\label{QTCKAThm}
Let $G \curvearrowright X$ be a   CKA action where for every vertex group the central extension  (\ref{centralExtEQ}) has omnipotent hyperbolic quotient group. Then $G$ has property (QT).
\end{thm}

\begin{rem}
It is a long-standing problem whether every hyperbolic group is residually finite. Wise noted that if every hyperbolic group is residually finite, then any hyperbolic
group is omnipotent (see Remark~3.4 in \cite{Wise00}).
\end{rem}

Let us comment on the relation of  this work with the previous  \cite{NY20}. As  in \cite{NY20},   the ultimate goal is to utilize Bestvina-Bromberg-Fujiwara's projection complex machinery to obtain   actions on quasi-trees. The common starting point is the class of special paths  developed in \cite{NY20} that record the distances of $X$. However, the flip assumption (see Definition~\ref{defn:flip}) on CKA actions was crucially used  there: the  fiber lines coincide with boundary lines in adjacent vertex pieces when crossing the boundary plane, roughly speaking.
 Hence, a straightforward gluing construction works in that case but fails in our general setting. In this paper, we use a completely different projection system to achieve the same purpose with a more delicate analysis.  
 
 It is worth mentioning the following fact  frequently invoked by many authors: any graph manifolds are quasi-isometric to some flip ones (see \cite{KL98}). This simplification, however, is useless to address property (QT), as such a quasi-isometry   does not respect the group actions.  Conversely, our direct treatment of any graph manifolds (closed or with nonempty boundary) is new, and we believe it will potentially allow for further applications.


We now explain  how we apply Theorem~\ref{QTCKAThm} to graph manifolds. If $M$ is a graph manifold with nonempty boundary then it always admit a Riemannian metric of nonpositive curvature (see \cite{B95}). In particular, $\pi(M) \curvearrowright \tilde{M}$ is a CKA action, and thus property (QT) of $\pi_1(M)$ follows immediately from Theorem~\ref{QTCKAThm}. However, closed graph manifolds may not support any Riemannian metric of nonpositive curvature (see \cite{B95}), so property (QT) in this case does not follow immediately from Theorem \ref{QTCKAThm}. We have to make certain modifications on some steps to run the proof of Theorem~\ref{QTCKAThm} for closed graph manifolds (see Section \ref{sec:qtgraphmld} for details).




\subsection{Property (QT) of relatively hyperbolic groups.}
When $M$ is a mixed 3–manifold, then $\pi_1(M)$ is  hyperbolic relative to the finite collection $\mathcal P$ of fundamental groups of maximal graph manifold
components, isolated Seifert components, and isolated JSJ tori (see \cite{Bigdely-Wise13}, \cite{Dah}). Therefore, we need to study property (QT) for relatively hyperbolic groups.

Our main result in this direction is a characterization of property (QT) for residually finite relatively hyperbolic groups,  which generalizes the corresponding results of \cite{BBF2} on  Gromov-hyperbolic groups.

\begin{thm}\label{QTRHGThm}
Suppose that a finitely generated group $H$ is hyperbolic relative to a finite set of subgroups $\mathbb P$. Assume that each $P\in\mathbb P$  acts by isometry on finitely many quasi-trees $T_i$ $(1\le i\le n_P)$ so that the induced diagonal action on $\prod_{i=1}^{n_P} T_i$ has property (QT). If $H$ is residually finite, then $H$ has property (QT).
\end{thm}
\begin{rem}
Since maximal parabolic subgroups are undistorted, each $P\in\mathcal P$ obviously has property (QT) if $G$ has property (QT). A non-equivariant version of this result was proven by Mackay-Sisto \cite{MS13}.
\end{rem}


\begin{rem}
It is well-known that mixed 3-manifold groups $G=\pi_1(M)$ are hyperbolic relative to a collection $\mathbb P$ of abelian groups and graph manifold groups $P=\pi_1(M_i)$. However, it is still insufficient to derive directly via Theorem \ref{QTRHGThm} the property (QT) of $G$  from that of  graph manifold groups    $P$ asserted in Theorem \ref{QTCKAThm}, since  $P$  may not preserve factors in the finite product of quasi-trees. Of course, passing to an appropriate finite index subgroup $P'<P$ preserves the factors, but it is not clear at all whether $P'$ are peripheral subgroups of a finite index subgroup $G'$ of $G$. 
In order to find such a $G'$,  a stronger assumption must be satisfied so that every finite index subgroup of each $P$ is separable in $G$. This requires the notion of a \textit{full profinite topology} induced on subgroups (see the precise definition before Theorem \ref{thm:sepaQT} and a relevant discussion in \cite{R18}).  See Theorem \ref{thm:sepaQT}     for the precise statement.
In the setting of a mixed 3-manifold,      Lemma \ref{lem:separable} verifies that each peripheral subgroup $P\in\mathbb P$ of $\pi_1(M)$  satisfies  this assumption. Therefore, all mixed 3-manifolds are proven to have property (QT).
\end{rem}

We now explain a few algebraic and geometric consequences for groups with property (QT).

Similar to trees, any isometry on quasi-trees must be either elliptic or loxodromic (\cite{M05}).  Hence, if a finitely generated group acts properly (in a metric sense) on finite products of quasi-trees, then every non-trivial element is undistorted (Lemma \ref{QTUndistortedLem}).  Moreover, property (QT) allows  to characterize virtually abelian groups among sub-exponential growth groups and solvable groups.

\begin{thm}
\label{thm:subexponential}
Let $G$ be a finitely generated group. Then the following statements hold.
\begin{enumerate}
    \item Assume that $G$ has sub-exponential growth. Then $G$ has property (QT) if and only if $G$ is virtually abelian.
    \item Suppose that $G$ is solvable with finite virtual cohomological dimension. Then $G$ has property (QT) if and only if it is virtually abelian.
\end{enumerate}
\end{thm}

By Theorem \ref{QTRHGThm}, this yields as a consequence that non-uniform lattices in $SU(n,1)$ and $Sp(n, 1)$ for $n\ge 2$ fail to act properly on  finite products of quasi-trees.  
\begin{cor}
A non-uniform lattice in $SU(n,1)$ for $n\ge 2$ or $Sp(n,1)$ for $n\ge 1$ does not have property (QT), while any lattice of $SO(n,1)$ has property (QT) for $n\ge 2$.
\end{cor}



\subsection*{Overview}
The paper is structured as follows. In Section \ref{section:preliminary}, we  recall the preliminary materials about Croke-Kleiner admissible groups, axiomatic constructions of quasi-trees,  and collect a few preliminary observations to disprove property (QT) for Sol and Nil geometries and to prove Theorem~\ref{thm:subexponential}. Section \ref{QTRHGSec} contains a proof of Theorem \ref{QTRHGThm} and its variant  Theorem \ref{thm:sepaQT}.  The next four sections aim to prove Theorem \ref{QTCKAThm}: Section \ref{ConeoffSpaceSec} first recalls a cone-off construction of CKA actions from \cite{NY20}  and then outlines the steps executed in Sections \ref{FiberLineSystemSec}, \ref{CKADistFormulaSec}, and \ref{ProofQTCKASec} to prove property (QT) for CKA actions. Sections \ref{FiberLineSystemSec} and  \ref{CKADistFormulaSec}  explain in detail the construction of projection systems of fiber lines and then prove the corresponding distance formula. We  finish the proof of Theorem \ref{QTCKAThm} in Section~\ref{sec:proofCKA}. In Section \ref{sec:application3mld}, we present the applications of the previous results for 3-manifold groups and prove Theorem~\ref{thm:3-mfd} and  Theorem  \ref{thm:QT3MFD}.

\subsection*{Acknowledgments}
H.T.Nguyen is partially supported by Project ICRTM04\_2021.07 of the International Centre for Research and Postgraduate Training in Mathematics, Vietnam. W. Y. is supported by National Key R \& D Program of China (SQ2020YFA070059) and National Natural Science
Foundation of China (No. 12131009). We are also grateful to the anonymous
referee for many very helpful comments.

\section{Preliminary}
\label{section:preliminary}
This section reviews concepts property (QT),  Croke-Kleiner admissible actions, and the construction of quasi-trees. Several observations are made to determine property (QT) of 3-manifolds with Sol and Nil geometry. This  includes the fact that  every elements are undistorted in groups with property (QT)  and  some attempts to characterize by property (QT) the class of virtually abelian groups in  solvable/sub-exponential growth groups.

In the sequel, we use the notion $a \preceq_{K} b$ if the exists $C = C(K)>0$ such that $a \le Cb + C$, and $a \sim_{K} b$ if $a \preceq_{K} b$ and $b \preceq_{K} a$. Also, when we write $a \asymp_{K} b$ we mean that $a/C \le b \le Ca$. If the constant $C$ is universal from context, the sub-index $\preceq_{K}$ shall be omitted.

\subsection{Property (QT)}\label{SSPropertyQT}
\begin{defn}
We say that a finitely generated group G has \emph{property (QT)} if it acts isometrically
on finite products $X=T_1\times T_2\times \cdots \times T_n$  of quasi-trees {with $L^2$-metric} so that for any basepoint $o\in X$, the induced orbit map
$$
g\in G\mapsto go\in X
$$
is a quasi-isometric embedding of $G$ equipped with some (or any) word metric $d_G$ to $X$ with the product metric $d$.
\end{defn}
\begin{rem}
A group with property (QT) acts properly on finite products of quasi-trees in a metric sense: $d(o, go)\to\infty$ as $d_G(1,g)\to\infty$. We would emphasize that all consequences of the property (QT) in this paper use merely the existence of a metric proper action.
\end{rem}

By definition, a quasi-tree is assumed to be a graph quasi-isometric to a simplicial tree. This does not loss generality as any geodesic metric space (with an isometric action) is quasi-isometric to a graph (with an equivariant isometric action) by taking the \textit{1-skeleton of its Rips complex}: the vertex set consists of all points and two points with distance less than 1 are connected by an edge.

The first part of the following lemma allows one to pass to finite index subgroups in the study Property (QT) of groups, as explained in Section~2.2 of \cite{BBF2}. The second part of Lemma~\ref{lem:passfiniteindex} is an immediate consequence of the definition of property (QT).

\begin{lem}
\label{lem:passfiniteindex}
\begin{enumerate}
    \item Let $H \le G$ be a finite index subgroup of $G$. Then $G$ has property (QT) if and only if $H$ has property (QT).
    
    \item Let $H \le G$ be an undistorted subgroup of $G$. Suppose that $G$ has property (QT) then $H$ has property (QT).
\end{enumerate}
\end{lem}



Below is a corollary of the de Rham decomposition theorem (see \cite[Theorem 1.1]{FL08})
that will be used for the next discussions.
\begin{cor}\label{deRhamDecompQTCOr}
A finite product $X=T_1\times T_2\times \cdots\times T_n$  of quasi-trees  must have de Rham decomposition $$ X=\mathbb R^k\times T_{k+1}\times\cdots\times T_n$$  if the first $k$ quasi-trees ($k\ge 0$) are all real lines among $\{T_i: 1\le i\le n\}$.
\end{cor}

A finite product $\prod_{i=1}^{n} T_i$ of quasi-trees has no \textit{$\mathbb R$-factor} if no $T_i$ is isometric to $\mathbb R$ or a point. In this case, the Euclidean factor $\mathbb R^k$ will disappear. In what follows, we give some general results about groups with property (QT). 

\begin{lem}\label{QTUndistortedLem}
Assume that  $G$ has property (QT). Then  the subgroup generated by an element $g\in G$ is undistorted in $G$.
\end{lem}

\begin{proof}
Let $X=\mathbb R^k\times T_{k+1}\times\cdots\times T_n$ be the de Rham decomposition of a finite product of quasi-trees.  By \cite[Corollary 1.3]{FL08}, up to passage to finite index subgroups, $G$ acts by isometry on each factor $\mathbb R^k,$ and $T_i$ for $k+1\le i\le n$. Let $g\in G$ be an infinite order element. If the image of $g$ is an isometry on the Euclidean space $\mathbb R^k$, then  it either fixes a point or preserves an axe. If  the image of  $g$ is an isometry on a quasi-tree $T_i$  then by  \cite[Corollary 3.2]{M06},  it has either a bounded orbit or a quasi-isometrically embedded orbit.  

Fix a basepoint $o=(o_k,o_{k+1},\cdots, o_n)\in X$.  If the action of $G$ on $X$ is proper, by the first paragraph, there must exist a unbounded action of $\langle g\rangle$ on some factor $Y=\mathbb R^k$ or $Y=T_i$, so we have $m \le \lambda |o_k- g^mo_k|_{Y} +c$ for some $\lambda, c>0$. Since any isometric orbital map is Lipschitz, we have $ |o-g^mo|_X \le  C|1- g^m|_G$ for some $C>0$. Noting that $|o-g^mo|_Y\le |o-g^mo|_X$, we have  $m\mapsto g^m$ is a quasi-isometric embedding of $\langle g\rangle$ into $G$.
\end{proof}

Note that the Sol group embeds quasi-isometrically into a product of two hyperbolic planes ({for example, see \cite[Section~9]{C08}}). However, the Sol lattice contains exponentially distorted elements by \cite[Lemma 5.2]{NS20}.
\begin{cor}\label{QTFailforSOLCor}
The fundamental group of a 3-manifold with Sol geometry does not have property (QT).
\end{cor}

\begin{cor}\label{QTFailforBSCor}
The Baumslag-Solitar  group $BS(1,n)$ for $n> 1$ does not have property (QT). 
\end{cor}

\subsection{Sub-exponential growth  and solvable groups with property (QT)}
The fundamental group of a 3-manifold $M$ with Nil geometry also fails to have property (QT) since it contains  quadratically distorted elements (for example, see Proposition~1.2 in \cite{NS20}). Generalizing results about property (QT) of 3-manifolds with Sol or Nil geometry, in the rest of this subsection, we provide a characterization of sub-exponential growth groups/ solvable groups with property (QT) and give the proof of Theorem~\ref{thm:subexponential}.

In next results, we apply the general results in \cite{CCMT} about the isometric actions on hyperbolic spaces to quasi-trees.  By Gromov, unbounded isometric group actions   can be classified into the following four types:
\begin{enumerate}
    \item
    \textit{horocyclic} if it has no loxodromic element;
    \item
    \textit{lineal} if it has a loxodromic element and any two loxodromic elements have the same fixed points in the Gromov boundary;
    \item
    \textit{focal} if it has a loxodromic element, is not lineal and  any two loxodromic elements have one common fixed point;
   \item
    \textit{general type} if it has two loxodromic elements with no common fixed point.

\end{enumerate}

\begin{prop}\label{QIQTProp}
Assume that  $G$ has property (QT). Then there exist  a finite index subgroup $\dot G$ of $G$ which acts on a Euclidean space $\mathbb R^k$ with $k\ge 0$ and   finitely many quasi-trees $T_i$ for $1\le i\le n$  with  lineal or focal or general type action  so that the orbital map of $\dot G$ into $\mathbb R^k\times \prod_{i=1}^n T_i$ is a quasi-isometric embedding.

Moreover, the action on each $T_i$  can be chosen to be cobounded.
\end{prop}
\begin{proof}
By Corollary \ref{deRhamDecompQTCOr}, the finite product of quasi-trees given  by property (QT) has the above form of de Rham decomposition. By \cite[Corollary 1.3]{FL08},
$$1\to \Isom(\mathbb R^k) \times \prod_{i=k+1}^n \Isom(Y_{i})
\to \Isom(X) \to F\to 1 $$
where $F$ is a subgroup of the permutation group on the indices $\{k+1, \cdots, n\}$. Thus, there exists a finite index subgroup $\dot G$ of $G$ acting on each de Rham factor  such that $\dot G \subset \Isom(\mathbb R^k)\times \prod_{i=1}^n \Isom(Y_{i})$ for $k\ge 0$ and $i\ge k+1$.

First of all, we can assume that the actions of $\dot G$ on  $\mathbb R^k$ and  each $T_i$ is unbounded. Otherwise, we can remove $\mathbb R^k,$ and $ T_i$ with bounded actions from the product without affecting property (QT).

We now consider the action on $T_i$ for $k+1\le i\le n$.   We then need verify that  the   action of $\dot G$ on $T_i$ cannot be horocyclic. By way of contradiction, assume that the action of $\dot G$ on given $T_i$  is horocyclic.

Note that  the proof of \cite[Prop 3.1]{CCMT}  shows that the intersection of any orbit of $\dot G$ on $T_i$ with any quasi-geodesic is bounded.   By  \cite[Corollary 3.2]{M06}, any isometry   on a quasi-tree $T_i$   has either bounded orbits or  a quasi-geodesic orbit. Thus, we conclude that any orbit of $\langle h\rangle$ for every $h\in \dot G$ on $T_i$ is bounded. We are then going to prove that the action of $\dot G$ on $T_i$ has bounded orbits. This is a well-known fact and we present the proof for completeness.

By $\delta$-hyperbolicity of $T_i$, each $ h\in \dot G$ (with bounded orbits) has a  \textit{quasi-center} $c_h\in T_i$: there exists a constant $D>0$ depending only on $\delta$ such that $|c_h- h^ic_h|_{T_i}\le D$ for $i\in \mathbb Z$.  Moreover, for any $x\in c_h$ and any $y\in T_i$, the Gromov product $\langle y, hy\rangle_x$ is bounded by a constant $C$ depending only on $D$. As a consequence, the union $Z$ of quasi-centers $\{c_h: h\in \dot G\}$ has finite diameter. Indeed, note that    $\langle y, h_1y\rangle_x$ and $\langle x, h_2^{-1}x\rangle_y$ are bounded by $C$ for any $x\in c_{h_1}, y\in c_{h_2}$. If for two elements $h_1, h_2$,  the distance  $|c_{h_1}- c_{h_2}|_{T_i}$ is sufficiently large relative to $C$, the path connecting dots $(h_1h_2)^nx$ for $n\in \mathbb Z$ would be a sufficiently long local quasi-geodesic, so it is a global quasi-geodesic. By the previous paragraph, we obtain a contradiction so the $\dot G$-invariant set $Z$ is bounded.  Since    the action on $T_i$ is assumed to be unbounded, we thus proved that the action on $T_i$ cannot be horocyclic.

At last, it remains to prove the ``moreover" statement. By Manning's bottleneck criterion \cite{M06}, any geodesic is contained in a uniform neighborhood of every path with the same endpoints. Thus, any connected subgraph of a quasi-tree is uniform quasiconvex and so is a uniform quasi-tree.  Since $G$ is a finitely generated group, by taking the image of the Cayley graph, we can thus construct a connected subgraph on each quasi-tree $T_i$ so that the action on the subgraph (quasi-tree) is co-bounded. Thus, the proposition is proved.
\end{proof}

We are able to characterize sub-exponential groups with property (QT) as follows.

\begin{prop}\label{SubExpQTLem}
Let $G$ be a finitely generated group with sub-exponential growth. Then $G$ has property (QT) if and only if $G$ is virtually abelian.
\end{prop}
\begin{proof}
We first observe that $\mathbb R^k$ in Proposition \ref{QIQTProp} can be replaced by a finite product of real lines. Indeed, consider the action of $\dot G$ on Euclidean space $\mathbb R^k$. By assumption, $\dot G$ is of sub-exponential growth. It is well-known that the growth of any finitely generated group dominates that of quotients, so the image $\Gamma\subset \Isom(\mathbb R^k)$ of $\dot G$ acting on $\mathbb R^k$ has sub-exponential growth. Since finitely generated linear groups do not have intermediate growth, $\Gamma$ must be virtually nilpotent.   It is well-known that virtually nilpotent subgroups in $\Isom(\mathbb R^k)$ must be virtually abelian. Thus, $\Gamma$ contains  a finite index subgroup $\mathbb Z^l$ for $1\le l\le k$. By taking the preimage of  $\mathbb Z^l$ in $\dot G$, we can assume further that $\dot G$ acts on $\mathbb R^k$   through $\mathbb Z^l$. It is clear that $\mathbb Z^l$ acts on $l$ real lines $\mathbb R_1, \mathbb R_2\cdots,  \mathbb R_l$ so that the product action is geometric. We thus replace $\mathbb R^k$ by the product $\prod_{1\le i\le l}\mathbb R_i$ where $\dot G$ admits a lineal action on each $\mathbb R_i$ by translation.

By Proposition \ref{QIQTProp} the action of $\dot G$ on $T_i$ is either  lineal or focal or general type. In the latter two cases, $\dot G$ contains a free (semi-)group by   \cite[Lemma 3.3]{CCMT}, contradicting the sub-exponential growth of $\dot G$. Thus, the action of $\dot G$ on each $T_i$ is lineal. By Proposition \ref{QIQTProp}, we can assume that $T_i$ is a quasi-line. 

By \cite[Lemma 3.7]{M06}, a quasi-line $T$ admits a $(1, C)$-quasi-isometry $\phi$ (with a quasi-inverse $\psi$) to $\mathbb R$ for some $C>0$.   A lineal action of $G$ on $T$ is then conjugated to a quasi-action  of $G$ on $\mathbb R$ sending $g\in G$ to a $(1, C')$-quasi-isometry $\phi g\psi$ on $\mathbb R$ for some $C'=C'(C)>0$. By taking an index at most 2 subgroup, we can assume that every element in $G$ fixes pointwise the two ends of $T$. Note that a $(1, C')$-quasi-isometry $\phi g\psi $ on $\mathbb R$ fixing the two ends of $\mathbb R$ is uniformly bounded away from a translation on $\mathbb R$.   So, for any $x\in \mathbb R$, the orbital map $g\mapsto \phi g\psi(x)$   is a quasi-homomorphism $G \to \mathbb R$. It is well-known that for any amenable group, any quasi-homomorphism must be a homomorphism up to bounded error. We conclude that  any $[G, G]$-orbit on $T$ stays in a bounded set.

Therefore, any $[G, G]$-orbit on $(\prod_{1\le i\le l}\mathbb R_i) \times (\prod_{1\le i\le n} T_i)$  is bounded, so the proper action on $X$ implies that $[\dot G, \dot G]$ is a finite group.  It is well-known that if a group has finite commutator subgroup, then it is virtually abelian (\cite[Lemma II.7.9]{BH99}). The lemma is proved.
\end{proof}

It would be interesting to ask whether Proposition \ref{SubExpQTLem} holds within the class of solvable groups. In Proposition \ref{SOLQTLem} below, we are able to give a positive answer to the previous question when the solvable group has finite virtual cohomological dimension. To this end, we need the following fact.  


\begin{lem}\label{metaAbLinealLem}
Any unbounded isometric action of a meta-abelian group on a quasi-tree must be lineal.
\end{lem}
Recall that  a meta-abelian group is a group whose commutator subgroup is abelian.
\begin{proof}
Indeed,   the abelian group $\Gamma=[G,G]$ (of possibly infinite rank) cannot contain free semi-groups, so by \cite[Lemma 3.3]{CCMT}, the action of $\Gamma$ on a quasi-tree $T$ must be bounded or lineal.

Assume first that  $\Gamma$ has a bounded orbit $K$ in $T$. Since $G/\Gamma$ is abelian, we have that $g^mh^n K=h^ng^m K$ for any $n,m\in \mathbb Z$ and $g,h\in G$, and thus $g h^n K =h^n gK$ has finite Hausdorff distance to $h^nK$ for any $n\in \mathbb Z$.  Assume that $g, h$ are loxodromic. Then $\{h^nK, n\in \mathbb Z\}$ is quasi-isometric to a line. Hence, we obtain that the fixed points of $g, h$ at the Gromov boundary must be coincide. This means the action of $G$ on $T$ is lineal.

In the lineal case, $\Gamma$ preserves some bi-infinite quasi-geodesic $\gamma$ up to finite Hausdorff distance.  Since $\Gamma$ is a normal subgroup in $G$, we see that every loxodromic element in $G$ also preserves $\gamma$ up to a finite Hausdorff distance. Thus, the action of $G$ on $T$ is also lineal.
\end{proof}

By Lemma \ref{QTUndistortedLem}, a group with property (QT) is \textit{translation proper} in the sense of Conner \cite{C00}: the translation length of any non-torsion element is positive. If $G$ is solvable and has finite v.c.d., then Conner shows that $G$ is virtually meta-abelian.

\begin{prop}\label{SOLQTLem}
Suppose that a solvable group $G$ has finite virtual cohomological dimension. If $G$ has property (QT) then it is virtually abelian.
\end{prop}

\begin{proof}
 Passing to finite index subgroups, assume that $G$ is meta-abelian so  any quotient of $G$ is meta-abelian. By Lemma \ref{metaAbLinealLem}, the action   of $G$ on each $T_i$ is lineal.

After possibly passing to an index 2 subgroup,  a lineal action of any amenable group $G$ on a quasi-line $T$ can be quasi-conjugated to be an isometric action on $\mathbb R$. Indeed, by the proof of Lemma \ref{SubExpQTLem}, conjugating the original action by almost isometry gives a quasi-action  of $G$ on $\mathbb R$ so that any orbital map induces a quasi-homomorphism of $G$ to $\mathbb R$.   For  amenable groups, any quasi-homomorphism differs from a homomorphism by a uniform bounded constant. Thus, up to quasi-conjugacy, the lineal action of $G$ on $T$ can be promoted to be an isometric action on $\mathbb R$.

Consequently, we can quasi-conjugate  the action of a solvable group $G$ on a finite product of quasi-trees to a proper action on a Euclidean space.  Thus, $G$ must be virtually abelian.
\end{proof}

\begin{proof}[Proof of Theorem~\ref{thm:subexponential}]
 The proof is a combination of  Proposition~\ref{SubExpQTLem} and Proposition~\ref{SOLQTLem}.
\end{proof}

\subsection{CKA groups}
\label{CKA}
\emph{Admissible groups} firstly introduced in \cite{CK02} are a particular class of graph of groups that includes fundamental groups of $3$--dimensional graph manifolds. In this section, we review  admissible groups and their properties that will used throughout the paper.

Let $\mathcal G$ be a connected graph. We often consider oriented edges from $e_-$ to $e_+$ and denote $e=[e_-,e_+]$. Then $\bar e=[e_+,e_-]$ denotes the oriented edge with reversed orientation. Denote by $\mathcal G^0$ the set of vertices and by $\mathcal G^1$ the set of all oriented edges.

\begin{defn}
\label{defn:admissible}
A graph of groups $\mathcal{G}$ is \emph{admissible} if
\begin{enumerate}
    \item $\mathcal{G}$ is a finite graph with at least one edge.
    \item Each vertex group ${ G}_v$ has center $Z({ G}_v) \cong \Z$, ${ H}_v \colon = { G}_{v} / Z({ G}_v)$ is a non-elementary hyperbolic group, and every edge subgroup ${ G}_{e}$ is isomorphic to $\Z^2$.
    \item Let $e_1$ and $e_2$ be distinct directed edges entering a vertex $v$, and for $i = 1,2$, let $K_i \subset { G}_v$ be the image of the edge homomorphism ${G}_{e_i} \to {G}_v$. Then for every $g \in { G}_v$, $gK_{1}g^{-1}$ is not commensurable with $K_2$, and for every $g \in  G_v \setminus K_i$, $gK_ig^{-1}$ is not commensurable with $K_i$.
    \item For every edge group ${ G}_e$, if $\alpha_i \colon { G}_{e} \to { G}_{v_i}$ is the edge monomorphism, then the subgroup generated by $\alpha_{1}^{-1}(Z({ G}_{v_1}))$ and $\alpha_{2}^{-1}(Z({ G}_{v_1}))$ has finite index in ${ G}_e$.
\end{enumerate}
A group $G$ is \emph{admissible} if it is the fundamental group of an admissible graph of groups.
\end{defn}

\begin{defn}
\label{defn:admissibleaction}
We say that an admissible group $G$ is a {\it Croke-Kleiner admissible group} or {\it CKA group} if it acts properly discontinuous, cocompactly and by isometries on a complete proper CAT(0) space $X$. Such action $G \curvearrowright X$ is called a {\it CKA action} and the space $X$ is called a {\it CKA space}.
\end{defn}

\begin{exmp}
\label{exmp:CKA}
\begin{enumerate}
    \item Let $M$ be a nongeometric graph manifold that admits a nonpositively curved metric. Lift this metric to the universal cover $\tilde{M}$ of $M$, and we denote this metric by $d$. Then the action $\pi_1(M) \curvearrowright (\tilde{M}, d)$ is a CKA action.
    \item Let $T$ be the torus complexes constructed in \cite{CK00}. Then $\pi_1(T) \curvearrowright \tilde{T}$ is a CKA action.
    \item One may build Croke-Kleiner admissible groups algebraically from any finite number of  hyperbolic CAT(0) groups. The following example is for $n =2$ but the same principle works for any $n \ge 2$. Let $H_1$ and $H_2$ be two torsion-free hyperbolic groups that act geometrically on $CAT(0)$ spaces $X_1$ and $X_2$ respectively. Then $G_i = H_i \times \langle t_i \rangle$ (with $i =1,2$)  acts geometrically on the $CAT(0)$ space $Y_ i = X_{i} \times \mathbb R$. Any primitive hyperbolic element $h_i$ in $H_i$ gives rise to a totally geodesic torus $T_i$ in the quotient space $Y_{i}/G_i$ with basis $([h_i], [t_i])$. We re-scale $Y_ i$ so that the translation length of $h_i$ is equal to that of $t_i$ for each $i$.  Let $f \colon T_1 \to T_2$ be a \textit{flip} isometry respecting these lengths, that is, an orientation-reversing isometry mapping $[h_1]$ to $[t_2]$ and $[t_1]$ to $[h_2]$.   Let $M$ be the space obtained by gluing $Y_1$ to $Y_2$ by the isometry $f$. There is a metric on $M$ which makes $M$ into a locally $CAT(0)$ space (see e.g. \cite[Proposition II.11.6]{BH99}). By the Cartan-Hadamard Theorem, the universal cover $\widetilde M$ with the induced length metric from $M$ is a CAT(0) space. Let $G$ be the fundamental group of $M$. The action $G \curvearrowright \widetilde M$ is geometric, and $G$ is an example of a Croke-Kleiner admissible group.  
\end{enumerate}
\end{exmp}

\begin{rem}
All graph 3-manifold groups are admissible, but there are closed graph 3-manifold groups that  are not CAT(0) groups (see \cite{KL96}), and thus are not CKA groups.  The following is another example.  Take two non-virtually split central extensions of hyperbolic groups by $\mathbb Z$ (e.g. $\widetilde{SL(2, \mathbb{R})}$ lattices) and amalgamate them over $\mathbb Z^2$ to get an admissible group. This group cannot act properly on CAT(0) spaces, since  central extensions acting on CAT(0) spaces must  virtually split as direct products (\cite[Thm. II.7.1]{BH99}).    
\end{rem}

A collection of subgroup $\{K_1,\cdots, K_n\}$ in a group $H$ is called \textit{almost malnormal} if $\sharp (gK_ig^{-1}\cap K_j)=\infty$ implies   $i=j$ and $g\in K_i$. It is well-known that a hyperbolic group is hyperbolic relative to any almost malnormal collection of quasi-convex subgroups (\cite{Bow12}).
\begin{lem}\label{MalnormalEdgegroupsLem}
Let $K_e$ be the image  of an edge  group $G_e$ into $G_v$ and $\bar K_e$ be its projection in $H_v$ under $G_v\to H_v=G_v/Z(G_v)$. Then $\mathbb P:=\{\bar K_e: e_-=v, e\in \mathcal G^1\}$ is an almost malnormal collection of virtually cyclic subgroups in $H_v$.

In particular, $H_v$ is hyperbolic relative to $\mathbb P$.
\end{lem}
\begin{proof}
Since $Z(G_v)\subset K_e \cong \mathbb Z^2$, we have $\bar K_e=K_e/Z(G_v)$ is virtually cyclic. The almost malnormality follows from non-commensurability of $K_e$ in $G_v$. Indeed,  assume that $\bar K_e\cap h\bar K_{e'}h^{-1}$ contains an infinite order element by the hyperbolicity of $H_v$. If  $g\in G_v$ is sent to $h$, then $K_e\cap gK_{e'}g^{-1}$ is sent to $\bar K_e\cap h\bar K_{e'}h^{-1}$.    Thus, $K_e\cap gK_{e'}g^{-1}$ contains an abelian group of rank 2. The non-commensurability of $K_e$ in $G_v$ implies that  $e=e'$ and $g\in K_e$. This shows that $\mathbb P$ is almost malnormal.
\end{proof}

Let $G \curvearrowright X$ be a CKA action where $G$ is the fundamental group of an admissible graph of groups $\mathcal{G}$, and let $G \curvearrowright T$ be the action of $G$ on the associated Bass-Serre tree $T$ of $\mathcal{G}$ (we refer the reader to Section~2.5 in \cite{CK02} for a brief discussion). Let $T^0$ and $T^1$ be the vertex and edge sets of $T$.  By CAT(0) geometry,
\begin{enumerate}
    \item
    for every vertex $v\in T^0,$ the minimal set $Y_{v} :=  \cap_{g \in Z(G_v)} Minset(g)$ of $X$ splits as metric product $\bar Y_v\times \mathbb R$ where $Z(G_v)$ acts by translation on the $\mathbb R$--factor and  $H_v=G_v/Z(G_v)$ acts geometrically on the Hadamard space $\bar Y_v$. Since $H_v$ is a hyperbolic group, it follows that $\bar{Y}_v$ is a hyperbolic space.
    \item for every edge $e \in T^1$, the minimal set $Y_{e} :=  \cap_{g \in G_e} Minset(g)$ of $X$ splits as $\bar Y_e \times \mathbb R^2\subset Y_v$ where $\bar Y_e$ is a compact Hadamard space and $G_e=\mathbb Z^2$ acts cocompactly on the Euclidean plane $\mathbb R^2$.

\end{enumerate}

We note that the assignments $v \to Y_v$ and $e \to Y_e$ are $G$--equivariant with respect to the natural $G$ actions.


We summarize results in Section~3.2 of \cite{CK02} that will be used in this paper.
\begin{lem}
\label{defn:vertexedgespace}
Let $G \curvearrowright X$ be a CKA action. Then there exists a constant $D >0$ such that 
\begin{enumerate}
    \item $X=\cup_{v \in T^0} {N}_{D}(Y_v) = \cup_{e \in T^1} {N}_{D}(Y_e)$. 
    \item If $\sigma, \sigma' \in T^0 \cup T^1$ and $N_D(Y_{\sigma}) \cap N_D(Y_{\sigma'}) \neq \emptyset$ then $|\sigma- \sigma'|_T < D$.
\end{enumerate}
\end{lem}
We shall refer $\tilde Y_v=N_D(Y_v)$ and $\tilde Y_e=N_D(Y_e)$  to as vertex and edge spaces for $X$. 

\subsubsection{Strips in CKA spaces}
(Section~4.2 in \cite{CK02})
\label{subsec:strip}
 We first choose, in a $G$--equivariant way, a plane  $F_e \subset Y_e$ (which we will call \textit{boundary plane}) for each edge $e \in T^1$.  For every pair of adjacent edges $e_1$, $e_2$, we choose, again equivariantly, a minimal geodesic from $F_{e_1}$ to $F_{e_2}$; by the convexity of $Y_v = \bar{Y}_{v} \times \R$ where $v:= e_{1} \cap e_{2}$, this geodesic determines a Euclidean strip $\mathcal{S}_{e_1e_2} : = \gamma_{e_1e_2} \times \R$ (possibly of width zero) for some geodesic segment $\gamma_{e_1 e_2} \subset \bar{Y}_v$. 

 Note that $\mathcal{S}_{e_1e_2} \cap F_{e_i}$ is an axis of $Z(G_v)$. Hence if $e_1, e_2, e \in E$, $e_{i} \cap e = v_{i} \in V$ are distinct vertices, then the angle between the geodesics $\mathcal{S}_{e_1e} \cap F_{e}$ and $\mathcal{S}_{e_2e} \cap F_{e}$ is bounded away from zero.
If $\langle f_1\rangle=Z(G_{v_1}), \langle f_2\rangle=Z(G_{v_2})$ then $\langle f_1, f_2\rangle$ generates a finite index subgroup of $G_e$. We remark that the intersection of two strips $\mathcal{S}_{e_1e}$ and $\mathcal{S}_{e_2e}$ is a point. Indeed, we have $\mathcal{S}_{e_1e} \cap \mathcal{S}_{e_2e} = (\mathcal{S}_{e_1e} \cap F_{e}) \cap (\mathcal{S}_{e_2e} \cap F_e)$. As two lines $\mathcal{S}_{e_1e} \cap F_{e}$ and $\mathcal{S}_{e_2e} \cap F_e$ in the plane $F_e$ are axes of $\langle f_{v_1} \rangle = Z(G_{v_1})$, $\langle f_{v_1} \rangle = Z(G_{v_2})$ respectively and $\langle f_1, f_2\rangle$ generates a finite index subgroup of $G_e$, it follows that these two lines are non-parallel, and hence their intersection must be a point.

 We note that the intersection of a boundary plane $F_e$ of $Y_v$ with the hyperbolic space $\bar{Y}_v$ is a line. The {\textit{boundary lines}} $\mathbb{L}_v$ of the hyperbolic space $\bar{Y}_v$ are the following collection of lines:
$
\mathbb{L}_{v} = \{ \ell_e := F_{e} \cap \bar{Y}_v \, | \, e_{-} = v \}
$.
\begin{defn}
\label{defn:flip}
If for each edge $e:=[v,w]\in T$,  the   boundary line  $\ell=\bar Y_v \cap F_e$ is parallel  to the $\R$--line in $Y_w = \bar{Y}_w \times \R$, then the CKA action is called \textit{flip}.
\end{defn}

In the sequel, it will be   useful to choose.
\begin{defn}
\label{rem:indexfunction}
An {\it indexed map} $\rho \colon X \to T^0$ is a $G$--equivariant coarsely Lipschitz map    such that  $x \in X_{\rho(x)}$ for all $x \in X$.

 \end{defn}
If $G$ acts freely on $X$, such a map $\rho$ can be constructed as follows.  Choose a fundamental set $\Sigma$  so that $\Sigma$ contains exactly one point from each orbit. Define $\rho: \Sigma\to T^0$ so that  $\rho(x)=X_{\rho(x)}$, and extend equivariantly $\rho$ to the whole space $X$. By Lemma \ref{defn:vertexedgespace}.(2), one can show that $\rho$ is a coarsely Lipschitz map: $|\rho(x)-\rho(y)|_T\le L |x-y|_X+L$ for some $L>0$. See     \cite[Section~3.3]{CK02} for more details. 

If $G$ acts only geometrically on $X$, we could replace $X$ with a $G$-orbit $Go$ for a basepoint $o$ with trivial stabilizer. This does not matter much as we are only interested in the coarse geometry  hereafter.  By modifying $X$, we could always assume such a basepoint $o$ exists. Indeed, attach a Euclidean cone to a point $o$ so that its nontrivial but finite stabilizer  acts freely on its boundary circle. We do the modification equivariantly for all translates in $Go$.


\subsubsection{Special paths in CKA spaces}
\label{subsection:specialpath}
Let $G \curvearrowright X$ be a CKA action.
We now introduce the class of \emph{special paths}  in $X$. 

\begin{defn}[Special paths in $X$]
\label{SpecialPathDefn}
Let $\rho \colon X \to T^0$ be the indexed map
given by Definition~\ref{rem:indexfunction}.
Let $x$ and $y$ be two   points in $X$. If $\rho(x) = \rho(y)$, a \emph{special path} in $X$ connecting $x$ to $y$ is the geodesic $[x,y]$. Otherwise, let $e_{1} \cdots e_{n}$ be the geodesic edge path connecting  $\rho(x)$ to $\rho(y)$ and let  $p_i=\mathcal S_{e_{i-1}e_i}\cap \mathcal S_{e_{i}e_{i+1}}$ be the intersection point of adjacent strips, where $e_{0}:=x$ and $e_{n+1}:=y$.  A \textit{special path} connecting $x$ to $y$ is the concatenation of the geodesics $$[x, p_{1}][p_{1}, p_{2}]\cdots [p_{{n-1}}, p_{n}][p_{n}, y]$$
\end{defn}

\begin{rem}
\label{rem:independent}
By definition, the special path except the $[x, p_{1}]$ and $[p_{n}, y]$ depends only on the geodesic $e_{1} \cdots e_{n}$ in $T$, the choice of planes $F_e$ and the indexed map $\rho$.
\end{rem}

\begin{figure}[htb] 
\centering \scalebox{0.6}{
\includegraphics{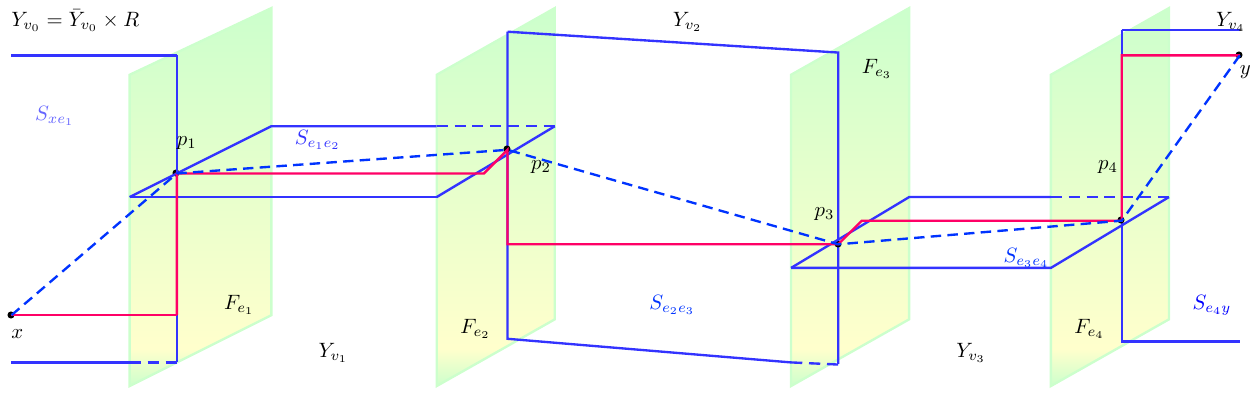} 
} \caption{The dotted and blue path from $x$ to $y$ is a special path, and the red path is one $L^1$-version of it.} \label{figure1}
\end{figure}

\begin{prop}\cite[Prop. 3.8]{NY20}
 \label{prop:spepathisqg}
There exists a constant $\mu >0$ such that every special path $\gamma$ in $X$ is a $(\mu, \mu)$--quasi-geodesic.
\end{prop}

Assume that $v_0=\rho(x),v_{2n}=\rho(y)\in\mathcal V$ so that $d(v_0, v_{2n})=2n$ for $n\ge 0$.
If $\gamma$ is a special path between $x$ and $y$, we then define 
\begin{equation}\label{hvdistancedefn}
\bigl|x-y\bigr|_X^{\hor} :=\sum_{i=0}^{2n} \bigl|p_{i}- p_{i+1}\bigr|_{Y_{v_i}}^{\hor}, \quad  \bigl|x-y\bigr|_X^{\ver} :=\sum_{i=0}^{2n} \bigl|p_{i}- p_{i+1}\bigr|_{Y_v}^{\ver}
\end{equation}
{where $p_0: = x$ and $p_{n+1}:= y$.} By Proposition \ref{prop:spepathisqg},  we have
$$
|x-y|_X\sim \bigl|x-y\bigr|_X^{\hor} +\bigl|x-y\bigr|_X^{\ver}.
$$

By definition, the system of special paths is $G$-invariant, so  the symmetric functions $d^h(x,y)$ and $d^v(x,y)$ are $G$-invariant for any $x,y\in X$.

We partition the vertex set $T^0$ of the Bass-Serre tree into two {disjoint} classes of vertices $\mathcal V_1$ and $\mathcal V_2$ such that   if $v$ and $v'$ are in $\mathcal{V}_i$ then $d_{T}(v,v')$ is even.

\begin{lem}\cite[Lemma 4.6]{NY20}
\label{lem:index2subgroup}
There exists a subgroup $\dot G$ of   index at most 2 in $G$     preserving  $\mathcal V_i$  for $i=1,2$ so that    $G_v\subset \dot G$ for any $v\in T^0$.
\end{lem}

\subsection{Projection axioms}
\label{sub:projectionaxioms}
In this subsection, we  briefly recall  the work of Bestvina-Bromberg-Fujiwara  \cite{BBF} on constructing  a quasi-tree of spaces.

\begin{defn}[Projection axioms]
\label{defn:projaxioms}
Let $\mathbb{Y}$ be a collection of geodesic spaces equipped with projection maps  $$\{\pi_{Y}:  \mathbb Y\setminus \{Y\}\to Y\}_{Y\in \mathbb Y}.$$  Denote $d_{Y}(X,Z) = diam (\pi_{Y}(X) \cup \pi_{Y}(Z))$ for $X\ne Y\ne Z \in \mathbb{Y}$.  The pair $(\mathbb{Y}, \{\pi_{Y}\}_{Y\in \mathbb Y})$ satisfies {\it projection axioms} for a {\it projection constant} $\xi \ge 0$ if
 \begin{enumerate}
     \item
     \label{axiom1}  $diam (\pi_{Y}(X)) \le \xi$ when $X \neq Y$.
     \item
     \label{axiom2} if $X,Y,Z$ are distinct and $d_{Y}(X, Z) > \xi$ then $d_{X}(Y,Z) \le \xi$.
     \item
     \label{axiom3} for $X \neq Z$, the set $\{ Y \in \mathbb{Y} \,:\, d_{Y}(X, Z) > \xi \}$ is finite.
\end{enumerate}
 \end{defn}

 The following is a useful example to keep in mind throughout the paper.  For further details, we refer the reader to the introduction of \cite{BBF}.  In this example, the collection of metric spaces $\mathbb Y$ consists of subspaces of a singe metric space; however, we emphasize that this need not be the case in general. 
\begin{exmp}\label{ex:linesinhypplane}
Let $G$ be a discrete group of isometries of $\mathbb H^2$, and $\gamma\in G$ a loxodromic element with axis $\gamma$.  Let $\mathbb Y$ be the set of all $G$--translates of $\gamma$.  Given $Y\in \mathbb Y$, let $\pi_Y$ denote the closest point projection map in $\mathbb H^2$.  Since all translates of $\gamma$ are convex, this is a well-defined $1$--Lipschitz map.  One may check that $(\mathbb Y,\pi_Y)$ satisfies the projection axioms for some constant $\xi$.
\end{exmp}

\begin{rem}
\label{rem:proj}
Let $(\mathbb{Y}, \{\pi_{Y}\}_{Y\in \mathbb Y})$ satisfy  projection axioms. By \cite[Thm~4.1 and Lem~4.13]{BBFS}, there exists a variant $\pi_{Y}'$ of $\pi_{Y}$   so that $\pi_{Y}$ and $\pi'_{Y}$ are uniformly close in Hausdorff distance, and $(\mathbb{Y}, \{\pi_{Y}'\}_{Y\in\mathbb Y})$ satisfies strong projection axioms, i.e, axioms are the same as projection axioms execpt for replacing (\ref{axiom2}) in Definition~\ref{defn:projaxioms} with the following stronger statement:
if $X,Y,Z$ are distinct and $d_{Y}(X, Z) > \xi$ then $\pi_{X}(Y) = \pi_{X}(Z)$
 for a projection constant $\xi'$ depending only on $\xi$.
\end{rem}

The following results from \cite{BBF} will be used in this paper.
\begin{itemize}
    \item Fix $K>0$. In \cite{BBF}, a quasi-tree of spaces $\mathcal C_K(\mathbb Y) $ is constructed for given $(\mathbb{Y}, \{\pi_{Y}\}_{Y\in \mathbb Y})$ satisfying projection axioms with constant $\xi$. 

    \item If $K>4\xi$ and $\mathbb Y$ is a collection of uniform quasi-lines, then $\mathcal C_K(\mathbb Y)$ is a unbounded quasi-tree. If $\mathbb Y$ admits a group action of $G$ so that $\pi_{gY}=g\pi_Y$ for any $g\in G$ and $Y\in \mathbb Y$, then $G$ acts by isometry on $\mathcal C_K(\mathbb Y)$.
\end{itemize}

Set $[t]_K=t$   if $t\ge K$, otherwise $[t]_K=0$.
  Let  $x\in X, z\in  Z\in \mathbb Y$. If $X\ne Y\ne Z$ define  $d_Y(x, z)=d_Y(X, Z)$. If $Y=X,Y\ne Z$, define $d_Y(x, z)=diam(\pi_Y(x, Z))$.  If $X=Y=Z$, let $d_Y(x, z)$ be the distance in $Y$. The following distance formula  from \cite{BBF}  is crucial in what follows.

\begin{prop} \cite[Proposition~2.4]{BBF2}
\label{BBFDistanceProp}Let $(\mathbb Y, \{\pi_{Y}\}_{Y\in \mathbb Y})$ satisfy the strong projection axioms with constant $\xi$. {Then for any $x, y\in \mathcal C_K(\mathbb Y)$,}
$$
\frac{1}{4} \sum_{Y\in \mathbb Y}
[d_Y (x, y)]_K\le  \bigl|x-y\bigr|_{\mathcal C_K(\mathbb Y)} \le  2 \sum_{Y\in \mathbb Y} [d_Y (x, y)]_K + 3K
$$
for all $K \ge 4\xi$.
\end{prop}



\begin{defn}[Acylindrical action]\cite{Bow08}\cite{Osin}
Let $G$ be a group acting by isometries on a metric space $(X,d)$. The action of $G$ on $X$ is called {\it acylindrical} if for any $r \ge 0$, there exist constants $R, N \ge 0$ such that for any pair $a, b \in X$ with $|a- b|_X \ge R$ then we have
\[
\# \bigl \{  g \in G \,|\, |ga- a|_X \le r \,\, \textup{and} \,\, |gb- b|_X \le r \bigr \} \le N.
\]
\end{defn}

By \cite{Bow08}, any nontrivial isometry of acylindrical group action on a hyperbolic space is either elliptic  or loxodromic.  A $(\lambda, c)$-quasi-geodesic $\gamma$ for some $\lambda, c>0$ is referred to as a \textit{quasi-axis} for a loxodromic element $g$, if $\gamma, g\gamma$ have finite Hausdorff distance depending only on $\lambda, c$.

A group is called \textit{elementary} if it is neither finite nor  virtually cyclic.
\begin{prop}\cite{BBF2}
\label{FiniteDblCosetsProp}
Assume that a non-elementary hyperbolic group $H$ acts acylindrically on a hyperbolic space $\bar Y$. For a loxodromic element $g\in H$, consider    the set $\mathbb A$ of all $H$-translates of a given $(\lambda, c)$-quasi-axis of $g$ for given $\lambda, c>0$.  Then there exists a constant
$\theta =\theta(\lambda, c)> 0$  such that for any $\gamma\in \mathbb A$, the set $$\{h\in G: diam(\pi_\gamma(h\gamma))\ge  \theta\}$$ is a finite union of double $E(g)$-cosets.

In particular, there are only finitely many distinct pairs $(\gamma, \gamma')\in  \mathbb A\times   \mathbb A$  satisfying  $diam(\pi_\gamma(\gamma'))> \theta$ up to the action of $H$.
\end{prop}

\begin{lem}\cite[Lemma 2.14]{Yang}\label{ExtensionLem}
Let $H$ be a non-elementary   group admitting a co-bounded and acylindrical action on  a $\delta$--hyperbolic space $(\bar Y, d)$. Fix a basepoint $o$. Then there exist a set $F\subset H$ of three loxodromic elements and $\lambda, c>0$  with the following property.

For any $h\in H$ there exists $f\in F$ so that $hf$ is a loxodromic element and the   bi-infinite path $$\gamma=\bigcup_{i\in \mathbb Z} (hf)^{i}\left([o, ho][ho, hfo]\right)$$   is a $(\lambda, c)$--quasi-geodesic.
\end{lem}

\begin{conv}\label{ConvQuasiLine}
When speaking of quasi-lines in hyperbolic spaces with actions satisfying Lemma \ref{ExtensionLem},   we always mean  $(\lambda, c)$--quasi-geodesics where  $\lambda, c>0$  depend on $F$ and $\delta$.
\end{conv}



\section{Property (QT) of relatively hyperbolic groups}\label{QTRHGSec}
{In this section, we are going to prove Theorem \ref{QTRHGThm}.} The notion of relatively hyperbolic groups can be  formulated from a number of equivalent ways. Here we shall present a quick definition due to Bowditch \cite{Bow12} and recall the relevant facts we shall need without proofs.

Let $H$ be a finitely generated group with a finite collection of subgroups $\mathcal P$. Fixing a finite generating set $S$,  we  consider the corresponding  Cayley graph $\Cay(H, S)$ equipped with the word metric $||_H$.

Denote by $\mathbb P=\{hP: h\in H, P\in\mathcal P\}$ the collection of peripheral cosets.  Let    $\hat{H}(\mathbb P)$ be the  \textit{coned-off Cayley graph}  obtained from $\Cay(H, S)$ as follows. A \textit{cone point} denoted by $c({P})$ is added  for each peripheral coset   $P \in\mathbb P$ and is joined  by \textit{half edges} to each element in $P$. The union of two half edges at a cone point  is called a   \textit{peripheral edge}. Denote by $|\cdot|_{\hat H}$ the induced length metric after coning-off.  

The pair $(G, \mathcal P)$ is said to be \textit{relatively hyperbolic} if the coned-off Cayley graph $\hat{H}(\mathbb P)$ is hyperbolic and \textit{fine}: any edge is contained in finitely many simple circles with uniformly bounded length.

By  \cite[Lemma 3.3]{Bow08}, \cite[Prop. 5.2]{Osin}, the action of $H$ on $\hat{H}(\mathbb P)$ is acylindrical.

Let $\pi_P$ denote the shortest projection in word metric to $P\in\mathbb P$ in $H$ and  $d_P(x, y)$ the $|\cdot|_H$-diameter of the projections of the points $x, y$ to $P$.
Since  $\mathbb P$ has the strongly contracting property with bounded intersection property,  the projection axioms with a constant $\xi>0$ hold for $\mathbb P$ (see \cite{S13}).

\subsection{Thick distance formula}

A geodesic edge path $\beta$  in the coned-off Cayley graph $\hat H(\mathbb P)$ is \textit{$K$-bounded} for $K>0$ if the end points of every peripheral edge  have  $d$-distance   at most $K$.

By definition,  a geodesic     $\beta=[x,y]$ can be subdivided into maximal  $K$-bounded non-trivial segments   $\alpha_i$ ($0\le i\le n$) separated by peripheral edges $e_j$ ($0\le j\le m$) where $|(e_j)_- - (e_j)_+|_{H}>K$. It is possible that $n=0$:   $\beta$ consists of only peripheral edges.

Define $$|\beta|_K:=\sum_{0\le i\le n} [Len(\alpha_i)]_K,$$  which   sums up the lengths of $K$-bounded subpaths of length at least $K$. It is possible that $n=0$, so $|\beta|_K=0$.  Define the \textit{$K$-thick distance}
\begin{equation}\label{RHGKthickdistEQ}
\bigl|x- y\bigr|_{\hat H}^{K}=\max\{|\beta|_K \}  
\end{equation}over   all relative geodesics $\beta$ between $x,y$. Thus, $\bigl|x- y\bigr|_{\hat H}^{K}$ is $H_v$-invariant.

A relative path without backtracking in $\hat{H}(\mathbb P)$ admits non-unique \textit{lifts} in  $\Cay(H, S)$ which are obtained by replacing the peripheral edge by a geodesic in $\Cay(H, S)$ with the same endpoints. The  distance formula follows from the fact that the lift of a relative quasi-geodesic is a quasi-geodesic (see \cite{DS05}, \cite[Prop. 6.1]{GP16}). The following formula is made explicitly in \cite[Theorem 0.1]{S13}.

\begin{lem}\label{RHGDistFLem}
For any sufficiently large $K>0$ and for any $x, y\in H$,
\begin{equation}\label{ConeoffDistFormulaEQ1}
|x- y|_H \sim_K \bigl|x- y\bigr|_{\hat H}^{K} + \sum_{P\in \mathbb P }
[ d_P(x, y)]_K.
\end{equation}
\end{lem}

The following result is proved in \cite[Lemma 5.5]{NY20} under the assumption that $H$ is hyperbolic relative to a set of virtually cyclic subgroups. However, the same proof works for any relatively hyperbolic group.

\begin{lem} \label{ConeofYDistFormulaLem}
For any sufficiently large $K>0$,  there exists  an  $H$--finite collection $\mathbb A$  of quasi-lines    in $\hat H$  and a constant $N=N(K, \hat H, \mathbb A)>0$, such that  for any two vertices $x, y\in  \hat H$, the following holds
\begin{equation}\label{ThickDistanceEQ}
\bigl|x- y\bigr|_{\hat H}^{K} \sim_N \sum_{\ell\in \mathbb A}
[\hat d_\ell(x, y)]_K
\end{equation}
\end{lem}

A group $H$ endowed with the {\it profinite topology} is a topological group  so that the set of all finite index subgroups is a (close/open) neighborhood base of the identity. A subgroup $P$ is called \textit{separable} if it is closed in the profinite topology. Equivalently, it is the intersection of all finite index subgroups  containing $P$. A group  is called \textit{residually finite} if the trivial subgroup is closed.

A maximal abelian subgroup of a residually finite group is separable (see \cite[Proposition~1]{Hamilton}). Note that a maximal elementary (i.e. virtually cyclic) group $E$ in a   relatively hyperbolic group $H$ contains a maximal abelian group (of rank 1) as a finite index subgroup. If $H$ is residually finite, then $E$ as a finite union of closed subsets is  closed and thus separable.

{We will use the following corollary in the proof of Theorem \ref{QTRHGThm}.}
\begin{cor}\label{QIEThickDistCor}
Assume that $H$ is a residually finite relatively hyperbolic group. Then for any $K\gg 0$, there exists a finite index subgroup $\dot H$ acting on  finitely many quasi-trees $T_i$ $(1\le i\le n)$ such that the orbital map of the $\dot H$-action on $\prod_{i=1}^n T_i$ is a quasi-isometric embedding from $(\dot H, |\cdot|^K_{\hat H})$ to $\prod_{i=1}^n T_i$.
\end{cor}

This corollary is essentially proved in \cite{NY20}, inspired by the arguments in the setting of mapping class groups \cite{BBF2}. We sketch the proof at the convenience of the reader.

\begin{proof}[Sketch of proof]
Recall that for any $\theta>0$, a set $\mathbb T$ of (uniform) quasi-lines {in a hyperbolic space}  with $\theta$-bounded projection satisfies the projection axioms with projection constant $\xi$ for a constant $\xi=\xi(\theta)>0$.   Let $\lambda$ and $c$ be the constants given by Lemma~\ref{ExtensionLem} with respect to the acylindrical action $H \curvearrowright \hat{H}$. For our purpose, we will choose $\theta$ to be the constant given by Proposition~\ref{FiniteDblCosetsProp}.  Then the distance formula for the quasi-tree $\mathcal C_K(\mathbb T)$ constructed  from $\mathbb T$ holds for any $K\ge 4\xi$.

For  a  fixed large constant $K$, Lemma \ref{ConeofYDistFormulaLem}  provides an $H$-finite set of quasi-lines $\mathbb A$ so that (\ref{ThickDistanceEQ}) holds. We then use the separability to find a finite index subgroup $\dot H$ of $H$   so that $\mathbb A$ decomposes as a finite union of  $\dot H$-invariant $\mathbb T_i$'s each of which satisfies the projection axioms with projection constant $\xi$. To be precise, the stabilizer $E$ of a quasi-line $\ell$ in $\mathbb A$  is a maximal elementary subgroup of $H$ and thus is separable in $H$ if $H$ is residually finite (since a maximal abelian group in a residually finite group is separable).  By  Proposition \ref{FiniteDblCosetsProp} {and the paragraph after Lemma~2.1 in \cite{BBF2}}, the separability of $E$  allows one to choose a finite index subgroup $\dot H$ containing $E$ such that any $\dot H$-orbit $\mathbb T_i$ in the collection of quasi-lines $H\ell$ satisfies the projection axioms with projection constant $\xi$. We take a common finite index subgroup $\dot H$ for finitely many quasi-lines $\ell$ in $\mathbb A$ up to $H$-orbits and therefore have found all $\dot H$-orbit $\mathbb T_i$ so that their union covers $\mathbb A$.

Finally, it is straightforward to verify that the right-hand term of (\ref{ThickDistanceEQ}) coincides with the sum of  distances over the finitely many quasi-trees $T_i:=\mathcal C_K(\mathbb T_i)$. Thus, the thick distance $d_{\hat H}^{K}(x, y)$ is quasi-isometric to the distance on a finite product of quasi-trees.
\end{proof}

All our discussion generalizes to the geometric action of $H$ on a geodesic metric space $Y$, since there exists a $H$-equivariant quasi-isometry between $\Cay(H,S)$ and $Y$. Therefore, replacing $H$ with $Y$, we have the same thick distance formula. This is the setup for CKA actions in next sections.

In next subsection, we obtain the property (QT) for relatively hyperbolic groups provided peripheral subgroups do so.

\subsection{Proof of property (QT) of relatively hyperbolic groups}

\begin{proof}[Proof of Theorem~\ref{QTRHGThm}]

Recall that $\mathcal P$ is a finite set of subgroups. For each $P\in \mathcal P$, choose a full set $E_P$ of left $P$-coset representatives in $H$ so that $1\in E_P$. For   given $P$ and  $1\le i\le n_P$, we define the collection of quasi-trees  $$\mathbb T_P^i :=\{fT_i: f\in E_P \}$$ where $T_i$ are quasi-trees associated to $P$ given in assumption. Then $H$ preserves $\mathbb T_P^i$ by the following action: for any point $f(x)\in fT_i$ and $h\in H$,
$$
h\cdot f (x) := f'p(x)\in f'T_i
$$
where $p\in P$ is given by $hf=f'p$ for $f'\in E_P$.

We are now going to define projection maps $\{\pi_{fT_{i}} \}$ as follows.

By assumption, we fix an orbital embedding $\iota^i_P$ of $P$ into $T_i$ so that the induced map $\prod_{i=1}^{n_P} \iota^i_P: P\to \prod_{i=1}^{n_P} T_i$ is a quasi-isometric embedding. We then define  an equivariant family of orbital maps $\iota_{fP}^i: fP\to fT_i$ so that $$\forall x\in fP,\; \iota_{fP}^i(x):= f\iota_P^i(f^{-1}x).$$
Then for any $h\in H$ and $x\in fP$, $h\cdot \iota_{fP}^i(x)=\iota_{f'P}^i(hx)$ where $f'\in E_P$ with $hf=f'p$ and $p\in P$.

Let $\pi_{fP}$ be the shortest projection to the coset $fP$  in $H$  with respect to the word metric.   For any two distinct $fT_i, f'T_i\in \mathbb T_P^i$, we set $$\pi_{fT_i}(f'T_i):=\iota_{fP}^i(\pi_{fP}(f'P))$$

Recall that $\mathbb P=\{fP: f\in H, P\in\mathcal P\}$ satisfies the projection axioms with shortest projection maps $\pi_{fP}$'s. It is readily checked that the projection axioms pass to the collection $\mathbb T^i_P$ under equivariant Lipschitz maps $\{\iota_{fP}^i\}_{fP\in\mathbb P}$.

We can therefore build the projection complex for $\mathbb T_P^i$ for a fixed $K\gg 0$. By Proposition~\ref{BBFDistanceProp}, the following distance holds for any $x',y'\in \mathcal C_K(\mathbb T_P^i)$:
\begin{align}\label{ParabolicQTEQ}
\bigl|x'- y'\bigl|_{\mathcal C_K(\mathbb T_P^i)} \sim_K \sum_{T\in \mathbb T_P^i} [d_T(x',y')]_K.
\end{align}

Note that  $\prod_{i=1}^{n_P} \iota^i_P: P\to \prod_{i=1}^{n_P} T_i$ is a quasi-isometric embedding for each $P\in\mathcal P$. Thus, for any $x, y\in G$ and $P\in\mathbb P$,
\begin{align}\label{dpxyEQ}
d_P(x,y)=|\pi_P(x)- \pi_P(y)|_P \sim \sum_{i=1}^{n_P} \bigl|\iota^i_P(\pi_P(x))- \iota^i_P(\pi_P(y))\bigr|_{T_i}.
\end{align}

Setting $x'=\iota^i_P(\pi_P(x))$ and $y'=\iota^i_P(\pi_P(y))$ in (\ref{dpxyEQ}), we deduce from (\ref{ParabolicQTEQ}) that  {
\begin{align}\label{dpxyUBDEQ}
d_P(x, y) \preceq_K \sum_{i=1}^{n_P} \bigl|\iota_P^i(\pi_P(x))- \iota^i_P(\pi_P(y))\bigr|_{\mathcal C_K(\mathbb T_P^i)}.
\end{align}}


Recall from Lemma~\ref{RHGDistFLem} that for any $x,y\in H$, we have
$$
|x- y|_H \sim_K \bigl|x- y\bigr|_{\hat H}^{K} + \sum_{P\in \mathbb P }
[ d_P(x, y)]_K.
$$
Note that the orbital map of any   isometric action is Lipschitz. {To prove property (QT) of $H$, it suffices to give an upper bound of $|x-y|_H$. Taking account of (\ref{dpxyUBDEQ}), it remains to construct a finite product of quasi-trees to bound $\bigl|x- y\bigr|_{\hat H}^{K}$  as follows.}

Since $H$ is residually finite, by Corollary \ref{QIEThickDistCor}, there exists a finite index subgroup,  still denoted by $H$, and  a finite product $Y$ of quasi-trees  so that the orbital map $\Pi_0$ from $ H$ to $Y$ gives a quasi-isometric embedding of $H$ equipped with $|\cdot|^K_{\hat H}$-function into $Y$.

Recall that $\pi_P$ is the shortest projection to $P\in \mathbb P$. For $1\le i\le n_P$, define $$
\Pi_i:  H\to \mathcal C_{K}(\mathbb T_P^i)
$$
by sending an element $h\in H$ to $\iota^i_P(\pi_P(h))$.
We then have $n$ equivariant maps $\Pi_i$ of $H$ to quasi-trees after re-indexing, where $n:=\sum_{P\in \mathcal P} n_P$.



Let $\Pi :=\Pi_0\times \prod_{i=1}^{n} \Pi_i$ be the map from $H$ to $Y\times \prod_{i=1}^{n} \mathcal{C}_{K}(\mathbb T_P^i)$, where $Y$ is the finite product of quasi-trees as in the previous paragraphs. As fore-mentioned,  the product map $\Pi$ gives an upper bound on $d_H(x,y)$, so is a quasi-isometric embedding of $H$. Therefore, $H$ has property (QT).
\end{proof}

\begin{rem}
{{An immediate corollary of Theorem~\ref{QTRHGThm} is that the fundamental group of a finite volume hyperbolic 3-manifold has property (QT). Alternate proof is that $\pi_1(M)$ is virtually compact special by deep theorems of Agol and Wise (see \cite{Agol} \cite{Wise20}), thus $\pi_1(M)$ has property (QT).}}
\end{rem}

We say that the profinite topology on $H$ induces a \textit{full profinite topology} on a subgroup $P$ if every finite index subgroup of $P$ contains the intersection of  $P$ with a finite index subgroup in $H$.

\begin{thm}
\label{thm:sepaQT}
Suppose that $H$ is residually finite and each $P\in\mathcal P$ is separable. Assume furthermore that $H$ induces the full profinite topology on each $P\in\mathcal P$. If each $P\in\mathcal P$ acts by isometry on  a finite product of quasi-trees without $\mathbb R$-factor {such that orbital maps are quasi-isometric embeddings}, then $H$ has property (QT).
\end{thm}
\begin{proof}
By \cite[Corollary 1.3]{FL08}, there is a finite index subgroup $\dot P$ of $P$ acting on each quasi-tree $T_i$ so that the diagonal action of $\dot P$ on $\prod_{i=1}^n T_i$ induces a quasi-isometric embedding orbital map $\prod_{i=1}^n\iota_{\dot P}^i$.

{By the assumption, $H$ induces the full profinite topology on $P\in \mathcal P$,  so every finite index  subgroups of a separable subgroup $P$ is also separable.} Thus, there are finite index subgroups $\dot H_P$ of $H$ for $P\in\mathcal P$ such that $\dot P=\dot H_P\cap P.$

Consider the finite index normal subgroup    $\dot H:=\cap \{h\dot H_Ph^{-1}: P\in\mathcal P\}$ in $H$.    Since $\dot H$ is normal in $H$,  we see that  $\dot H\cap hPh^{-1}\subset h \dot Ph^{-1}$ is equivalent to   $\dot H\cap P\subset \dot P$. The later holds by the choice  of $\dot H_P$.  Hence, for every $h\in  H$, $\dot H\cap hPh^{-1}$ preserves   the factors of the product decomposition.  Note that  $\dot H$ is hyperbolic relative to $\{\dot H\cap hPh^{-1}: h\in  H\}$.  The conclusion follows from Theorem \ref{QTRHGThm}.
\end{proof}

{In next sections (Sections~\ref{ConeoffSpaceSec}, ~\ref{FiberLineSystemSec}, ~\ref{CKADistFormulaSec} and ~\ref{ProofQTCKASec}), the proof of property (QT) of   CKA groups   will be discussed, which may be considered as the technical heart of this paper.} 

\section{Coning off CKA spaces}\label{ConeoffSpaceSec}

In this section, we recapitulate  the content of  \cite[Sect. 5]{NY20} and give an outline of  the proof of Theorem~\ref{QTCKAThm}.

Let $G \curvearrowright X$ be a CKA action where $G$ is the fundamental group of an admissible graph of groups $\mathcal{G}$ (see Subsection~\ref{CKA}), and let $G \curvearrowright T$ be the action of $G$ on the associated Bass-Serre tree $T$ of $\mathcal{G}$. Let $T^0$ and $T^1$ be the vertex and edge sets of $T$.

Let $\{F_e \}$ be the collection of boundary planes of the space $Y_v$ (see  Subsection~\ref{CKA}). We note that the intersection of a boundary plane $F_e$ of $Y_v$ with the hyperbolic space $\bar{Y}_v$ is a line. We define the collection of lines $\mathbb{L}_v$ of the hyperbolic space $\bar{Y}_v$ as   follows:
$$
\mathbb{L}_{v} = \{ \ell_e := F_{e} \cap \bar{Y}_v \, | \, e_{-} = v \}
$$
which shall be referred {\textit{boundary lines}}.

\subsection{Construction of coned-off spaces}
\label{coned-off spaces}
Recall that $T^0=\mathcal V_1\cup\mathcal V_2$ where $\mathcal V_i$ consists of vertices in $T$ with pairwise even distances.  Let $\dot G < G$ be the subgroup of index at most $2$ preserving $\mathcal V_1$ and $\mathcal V_2$ given by Lemma \ref{lem:index2subgroup}.

Fix a large $r>0$. A \textit{hyperbolic $r$-cone} by definition is the metric completion of the (incomplete) universal cover of a punctured hyperbolic disk of radius $r$. Let $\mathbb Y_i=\{\bar Y_v: v\in \mathcal V_i\}$ be the collection of hyperbolic spaces and $\dot{ \mathbb Y}_i=\{\dot Y_v: v\in \mathcal V_i\}$  their coned-off spaces (which are uniformly hyperbolic for $r\gg 0$) by attaching hyperbolic $r$-cones along the boundary lines of $\bar Y_v$.

Note that $\dot G$ preserves $\mathbb Y_i$ and $\dot{ \mathbb Y}_i$ by the action on the index $gY_v=Y_{gv}$ for any $g\in \dot G$.
For each $w\in T^0$, let $St(w)$ be the star of $w$ in $T$ with adjacent vertices as \textit{extremities}. Then $St(w)$ admits the action of $G_w$ so that the stabilizers of the extremities are the corresponding edge groups.

Define $\dot {\mathcal X_i}$ to be the space obtained from the disjoint union of coned-off spaces $\dot Y_v$ $(v\in \mathcal V_i)$ with cone points identified with the extremities of the stars $St(w)$ with $v\in Lk(w)$. Endowed with induced length metric, the space $\dot {\mathcal X_i}$ is a Gromov-hyperbolic space.

\begin{lem}\label{ActOnConeoffHypspaceLem}
Fix a sufficiently large $r>0$ and $i\in \{1,2\}$. The space $\dot{\mathcal  X_i}$  is a $\delta$--hyperbolic space where $\delta>0$ only depends on the hyperbolicity constants of $\dot Y_v$  ($v\in \mathcal V_i$).

The subgroup $\dot G$   acts on  $\dot{\mathcal X_i}$    with   the following properties:
\begin{enumerate}
\item
for each $v\in \mathcal V_i$, the stabilizer of $\dot Y_v$ is isomorphic to $G_v$ and $H_v$ acts co-boundedly on $\dot Y_v$, and
\item
    for each $w\in  T^0\setminus \mathcal V_i$, $G_w$ acts on $St(w)$ in the same manner of the action on Bass-Serre tree $T$.
\end{enumerate}
\end{lem}
\begin{proof}
Note that the stabilizers of the cone points of $\dot Y_v$ under the action of $G_v$ on $\dot Y_v$ are the same as that of the extremities of stars $St(w)$, which are both the edge groups $G_e$ for $e=[v,w]$.  By construction, the cone points of $\dot Y_v$ are identified  with  the extremities of stars $St(w)$, so    the actions of $G_v$ on $\dot Y_v$ ($v\in \mathcal V_i$) and of $G_w$ on $St(w)$ ($w\in T^0\setminus \mathcal V_i$) extend over $\dot {\mathcal X_i}$, and hence $\dot G$ acts by isometries on $\dot {\mathcal X_i}$.
\end{proof}
\begin{rem}
Our construction of coned-off spaces is slightly different from the one in \cite[Section 5.1]{NY20}, where the cone points are identified directly between different spaces $\dot Y_v$ and $\dot Y_{v'}$. Thus certain assumption on vertex groups is necessary in \cite{NY20} to ensure an action on the coned-off space.
\end{rem}

We now define the thick distance on $\dot {\mathcal X_i}$ $(i=1,2)$ by taking the sum of thick distances through $\dot Y_v$ as follows.

If $x$ is a point in a coned-off space $\dot Y_{v} \subset \dot {\mathcal{X}}$, we denote $\rho(x)$ by $v$ (by abuse of notations).
By the above tree-like construction,  any path between  $x, y\in \dot {\mathcal X_i}$ has to pass through in order a pair  of boundary lines $\ell_{v}^-, \ell_v^+$ of $\bar Y_{v}$ for each $v\in [\rho(x),\rho(y)]$.
By abuse of language, if $x$ is not contained in a hyperbolic cone, set $\ell_{v}^-=x$ for $v=\rho(x)$. Similarly, if $y$ is not contained in a hyperbolic cone, set $\ell_{v}^+=y$ for $v=\rho(y)$.

Let $(x_v,y_v)$ be a pair of points in the boundary lines $(\ell_{v}^-, \ell_v^+)$  so that $[x_v,y_v]$ is orthogonal to $\ell_{v}^-$ and $\ell_v^+$. Recall that $\bigl|x_v-y_v\bigr|_{\dot Y_v}^{K}$ is the $K$-cut-off thick distance defined in (\ref{RHGKthickdistEQ}). 
\begin{defn}
\label{defn:Kthickdistance}
For any $K\ge 0$, the \textit{$K$-thick distance} between $x$ and $y$ is defined by
\begin{equation}\label{ThickDistEQ}
\bigl|x- y\bigr|^{K}_{\dot {\mathcal X_i}}:=\sum_{v\in[\rho(x),\rho(y)]\cap \mathcal V_i} \bigl|x_v -y_v\bigr|_{\dot Y_v}^{K}.
\end{equation}

\end{defn}

Since $|\cdot|_{\dot Y_v}^{K}$ is $H_v$-invariant, we see that $\bigl|x- y\bigr|^{K}_{\dot {\mathcal X_i}}$ is $\dot G$-invariant.

\begin{rem}
The definition of $|\cdot |^{K}_{\dot {\mathcal X_i}}$ is designed to ignore the parts in hyperbolic cones between different pieces. One consequence is that perturbing $x,y$ in hyperbolic cones does not change their $K$-thick distance.
\end{rem}

\subsection{Construct the collection of quasi-lines in $\dot {\mathcal X}_i$}
\label{SSQuasilines}

If $E(\ell)$ denotes the stabilizer in $H_v$ of a boundary line $\ell$ of $\bar Y_v$,  then $E(\ell)$ is virtually cyclic and almost malnormal. Since $\{E(\ell)\}$ is $H_v$-finite by conjugacy, let $\mathbb E_v$ be a complete finite set  of conjugacy representatives.  By Lemma \ref{MalnormalEdgegroupsLem}, $H_v$ is hyperbolic relative to peripheral subgroups $\mathbb E_v$.  Hence, the results in Section \ref{QTRHGSec} apply here.

Let $\lambda, c>0$ be the universal constants given by Lemma \ref{ExtensionLem} applied to the actions of $H_v$ on $\dot Y_v$ for all $v\in T^0$ (since there are only finitely many actions up to conjugacy). By convention, the quasi-lines in coned-off spaces are understood as $(\lambda, c)$--quasi-geodesics in $\dot {\mathcal X_i}$ and  $\dot Y_v$'s.

The coning-off construction has the following consequence (\cite[Lemma 5.14]{NY20}): the shortest projection of any  quasi-line  $\alpha$ in $\dot Y_v$ to a quasi-line $\beta$ in $\dot Y_{v'}$ has to pass through the cone point attached to $\dot Y_{v'}$, and thus has uniformly bounded diameter by $\theta=\theta(\lambda, c)>0$.

For simplicity, we also assume that   $\theta=\theta(\lambda, c)>0$ satisfies the conclusion of  Proposition \ref{FiniteDblCosetsProp}. Consequently, this determines a constant $\xi=\xi(\theta)>0$ such that any set of quasi-lines with $\theta$-bounded projection satisfies the projection axioms with projection constant $\xi$.

Fix $K>\max\{4\xi, \theta\}$. For each $v\in \mathcal V$,  there exists  an $H_v$-finite collection of quasi-lines $ \mathbb A_v$    in $\dot Y_v$ and a constant $N=N(\mathbb A_v, K)$ such that $d^K_{\hat H_v}$-distance formula  holds by Lemma \ref{ConeofYDistFormulaLem}.

Since $\dot G$ acts co-finitely on $\mathcal V_1$ and $\mathcal V_2$,     we can assume  $\mathbb A_w=g\mathbb A_v$ if $w=gv$ for $g\in \dot G$.
Let $$\mathbb A_i :=\cup_{v\in \mathcal V_i} \mathbb A_v$$ be the union of $\mathbb A_v$'s for $i=1,2$ which are both $\dot G$-invariant. We now equip $\mathbb A_i$ with projection maps as the shortest projection maps between two quasi-lines in $\dot {\mathcal X_i}$ for $i=1,2$.

If $\gamma$ is a quasi-line in $\dot {\mathcal X_i}$ for $i=1,2$, denote by $\dot d_\gamma(x,y)$ the $|\cdot|_{\dot {\mathcal X_i}}$--diameter of the shortest projection of $x, y\in \dot {\mathcal X_i}$ to $\gamma$.

The following result  shows that the thick distance is captured by the projections of $\mathbb A_i$. Recall that $r$ is the radius of the hyperbolic cones in constructing $\dot {\mathcal X_i}$.
\begin{prop}\cite[Prop. 5.9]{NY20}\label{ConeoffDistFProp}
For any $x, y\in \dot{\mathcal X_i}$, the following holds
\begin{equation}\label{ConeoffDistFormulaEQ}
\begin{array}{cc}
  \bigl|x- y\bigr|_{\dot{\mathcal X_i}}^K \;\sim_{r, K} \; \sum_{\gamma\in \mathbb A_i } [\dot d_\gamma(x, y)]_{K} +   |\rho(x)- \rho(y)|_T.
\end{array}
\end{equation}
\end{prop}

In the next subsection, we  construct a suitable finite subgroup of $G$ such that it acts isometrically on a finite product of quasi-trees $T_1,\cdots, T_n$ under some assumptions on vertex groups. This allows rewriting the right-hand side of the distance formula (\ref{ConeoffDistFormulaEQ}) as the product distance of $T_i$'s. 

\subsection{Isometric action of a suitable finite index subgroup of $G$}
\label{subsec:indexaction}
In a group, two elements  are \textit{independent} if they do not have conjugate powers   (see \cite[Def.~3.2]{Wise00}).

\begin{defn}
\label{defn:omnipotent}
A group $H$ is \textit{omnipotent} if for any non-empty set of pairwise independent   elements $\{h_1,\cdots, h_r\}$ ($r\ge 1$) there is a number $p\ge 1$ such that for every choice of  positive natural numbers $\{n_1,\cdots, n_r\}$, there is a finite quotient $H\to \hat H$  such that $\hat h_i$ has order $n_ip$ for each $i$.
\end{defn}


Let $G \curvearrowright X$ be a CKA action, where $G$ is the admissible  graph of groups $\mathcal G$ so that every vertex group $G_v$ is a central extension of an omnipotent hyperbolic group. By Lemma \ref{ActOnConeoffHypspaceLem}, the finite index subgroup  $\dot G$   acts on $\dot{\mathcal X_1}\times \dot{\mathcal X_2}\times T$ that is equipped with the $\dot G$-invariant function $|\cdot|_{\dot{\mathcal X_1}}^K\times |\cdot|_{\dot{\mathcal X_2}}^K \times |\cdot|_T$. The main result of this subsection is the following.  

\begin{prop}\label{OmnipotentProp}
The group $\dot G$ admits finitely many isometric actions on  quasi-trees $T_i$ for $1\le i\le n$ such that there exists a $\dot G$-equivariant quasi-isometric embedding from $\dot{\mathcal X_1}\times \dot{\mathcal X_2}\times T$ to $T_1\times T_2\cdots \times T_n\times T$. 
\end{prop} 

We emphasize here that $|\cdot|_{\dot{\mathcal X_1}}^K\times |\cdot|_{\dot{\mathcal X_2}}^K \times |\cdot|_T$ on the domain  for the quasi-isometric embedding is  not  a distance function, but the target is equipped with product distance. 



By  \cite[Theorem II.6.12]{BH99}, $G_v$ contains a subgroup $K_v$ intersecting ~trivially with $Z(G_v)$ so that the direct product $K_v\times Z(G_v)$ is a finite index subgroup. Thus, the image of $K_v$ in $G_v/Z(G_v)$ is of finite index in $H_v$ and $K_v$ acts geometrically on  hyperbolic spaces $\bar Y_v$. Since $H_v$ is omnipotent and then is residually finite, we can assume that $K_v$ is torsion-free.

Recall  the $\dot G$-invariant collection of quasi-lines in Subsection \ref{SSQuasilines}:
$$\mathbb A_i=\cup_{v\in \mathcal V_i} \mathbb A_v$$  where $\mathbb A_v$ is  the collection of quasi-lines  so that $d^K$-distance formula  holds by Lemma \ref{ConeofYDistFormulaLem}. By the residual finiteness of $K_v$, there exists a  finite index subgroup $\dot K_v$  so that $\mathbb A_v$ is  partitioned into $\dot {K_v}$-invariant sub-collections with    projection constants $\xi$.

To prepare the proof, we need to introduce a compatible condition of glueing finitely index subgroups. 
A collection of finite index subgroups $\{G_e', G_v': v\in \mathcal G^0, e\in \mathcal G^1\}$ is called \textit{compatible} if whenever $v = e_-$, we have
$$G_v \cap G_e' = G_v' \cap G_e.$$
By \cite[Theorem 7.51]{DK18},   a compatible collection of finite index subgroups gives a finite index subgroup of $G$. The following result says that upon taking finite index subgroups, we can assume that each vertex group is a direct product in a CKA group.

\begin{lem}\label{FindexSubgrpsLem}
Let $\{\dot K_v<K_v: v\in \mathcal G^0 \}$ be a   collection of finite index subgroups. Then there exist   finite index subgroups $\ddot K_v$ of $\dot K_v$, $G_e'$ of $G_e$ and $Z_v$ of $Z(G_v)$ so that the  collection of finite index subgroups $\{G_e', G_v'=\ddot K_v\times Z_v: v\in \mathcal G^0, e\in \mathcal G^1\}$ is compatible.
\end{lem}

Assuming Lemma \ref{FindexSubgrpsLem}, we now complete the proof of Proposition \ref{OmnipotentProp}.  
\begin{proof}
We pass to further finite index subgroups $\ddot K_v<\dot K_v$ satisfying  compatible conditions, which then gives  a further index subgroup $\ddot G\subset \dot G$. For $i=1,2$, let us partition $\mathbb A_i=\cup_{k=1}^{n_i} \mathbb A_k^i$ into $\ddot G$-obits $\mathbb A_k^i$. By the construction of $\ddot G$, we know that $\ddot G$ intersects each vertex group $G_v$ of the Bass-Serre tree in a  (conjugate) subgroup $\ddot K_v$. Thus, for each $k$, $\mathbb A_k^i$  are the union of  certain $\ddot {K_v}$-invariant sub-collections where $v$ are varied in $\mathcal V_i$.  

Recall that $\mathbb A^i$ for $i=1,2$ satisfies the projection axioms with a uniform projection constant $\xi$  in Subsection \ref{SSQuasilines}. We can then build the quasi-trees $T_k^i:=\mathcal C_K(\mathbb A_k^i)$ where $1\le k\le n_i$. Setting  $n=n_1+n_2$, this thus yields isometric group actions of $\ddot G$ on quasi-trees $T_{i}$ $(1\le i\le n)$.

We first construct a $\ddot G$-equivariant   map $\Phi$ from $\dot{\mathcal X_1}\times \dot{\mathcal X_2}\times T$ to $T_1\times T_2\cdots \times T_n\times T$. By equivariance, it suffices to fix a basepoint in each $\mathcal X_1, \mathcal X_2, T$ and $T_i$'s so that $\Phi$ sends basepoints to basepoints. The quasi-isometric embedding property follows from  the distance formula (\ref{ConeoffDistFormulaEQ}), where the right-hand side is now replaced by the distance in the corresponding quasi-trees.

Note that $\ddot G$ is of finite index in $\dot G$. By taking more copies of quasi-trees $T_i$ in the target, the map $\Phi$ can be made $\dot G$-equivariant. Indeed, if a finite index subgroup $H\subset G$ acts on some space $X$ then $G$ acts on a finite product of $[G: H]$ copies of $X$ without preserving the factors. The map $\Phi$ can be extended to these copies as well. The proof of the Proposition is complete.
\end{proof}

We now give the proof of Lemma \ref{FindexSubgrpsLem}.
\begin{proof}[Proof of Lemma \ref{FindexSubgrpsLem}]
Assume that $\langle f_v\rangle =Z(G_v)$ for any $v\in \mathcal G^0$. Then for an oriented edge $e=[v,w]$ from $v$ to $w$, the subgroup   $\langle f_v, f_w\rangle$ is of finite index in $G_e$.

Note that $G_e\cong\mathbb Z^2$ admits a base $\{\hat f_v, \hat b_e\}$ where $\hat f_v$ is primitive so that $f_v$ is some power of $\hat f_v$.  Let $\pi_v: G_v\to H_v=G_v/Z(G_v)$. Thus,   $\pi_v(G_e)$ is a direct product of a torsion group with $\langle b_e\rangle$ in $H_v$, where $b_e=\pi_v(\hat b_e)$ is a loxodromic element.

Similarly,  let $\hat f_w, \hat b_{\bar e}\in G_e$ such that  $\langle \hat f_w, \hat b_{\bar e}\rangle=G_e$. Keep in mind that for any integer $n\ne 0$, $\langle \hat f_v^n, \hat b_e^n\rangle=\langle \hat b_{\bar e}^n, \hat f_w^n\rangle$  is of finite index in $G_e$.

We choose an integer $m\ne 0$ such that $\hat b_e^m\in \dot K_v$ for every vertex $v\in \mathcal G^0$ and every oriented edge $e$ from $e_-=v$. Such an integer $m$ exists since $\dot K_v$ injects into $H_v$ as a finite index subgroup, and $\mathcal G$ is a finite graph of groups.

Apply the omnipotence of $H_v$ to the independent set of elements $\{ b_e: e_-=v\}$. Let $p_v$ be the constant given by Definition~\ref{defn:omnipotent}.
Set
$$
s:=m\prod_{v\in \mathcal G^0} p_v
$$
Set $l_v=s/p_v$.  Thus, for the collection $\{ b_e: e_-=v\}$, there exists a finite quotient $\xi_v: H_v\to \bar H_v$ such that $\xi_v( b_e)$ has order $s=l_vp_v$ and $b_e^s\in \ker (\xi_v)$.  Then $\ddot K_v:=\dot K_v\cap \pi_v^{-1} \ker(\xi_v)$ is of finite index in $\dot K_v$. Recall that $\pi_v|_{K_v}: K_v\to H_v$ is injective (see the paragraph before Lemma \ref{FindexSubgrpsLem}). Since  $ \pi_v(\hat b_e^s) = b_e^s$ is loxodromic in $H_v$ and $\hat b_e^s\in\dot  K_v$ for $m|s$, we have $\hat b_e^s$ is a  loxodromic element in $\ddot K_v$.

For each oriented edge $e=[v,w]\in \mathcal G^1$, define $G_v':=\langle \hat f_v^{s}\rangle \times \ddot K_v$, $G_w':=\langle \hat f_w^{s}\rangle \times \ddot K_w$ and $G_e':=\langle \hat f_v^{s},  \hat b_e^s\rangle=\langle \hat b_{\bar e}^s,  \hat f_w^{s}\rangle<G_v'$.

Let  $g\in G_e\cap G_v'$ be any element so we can write   $g=\hat f_v^{sm} k$ for some $m\in \mathbb Z$ and $k\in \ddot K_v$. Recall that $\pi_v(G_e)$ is a direct product of $\langle b_e\rangle$ and a torsion group, and $\ddot K_v$ is torsion-free. So $\pi_v(g) =\pi_v(k)\in \pi_v(G_e)\cap \pi_v(\ddot K_v)$ is some power of $b_e$: $\pi_v(k)=b_e^l$ for some $l\in \mathbb Z$. Note that $b_e^l=\pi_v(k)\in \ker (\xi_v)$ so  the omnipotence implies that $s|l$, i.e. $l=ns$ for some $n\in \mathbb Z$.  Since $b_e=\pi_v(\hat b_e)$ and $\pi_v: \ddot K_v\to H_v$ is injective, we obtain that $k=\hat b_e^{ns}$. Therefore, $g=\hat f_v^{sm}\hat b_e^{ns}\in G_e'$ which implies
$$G_v \cap G_e' = G_v' \cap G_e.$$

Therefore, the collection $\{G'_v, G'_e \,\,| \, v \in \mathcal{G}^{0}, e \in \mathcal{G}^{1} \}$ is verified to be  compatible.
\end{proof}

\subsection{Outline of the proof of Theorem~\ref{QTCKAThm}}
 \label{OutlineProofSSec}
Let $G \curvearrowright X$ be a CKA action where $G$ is the fundamental group of an admissible graph of groups $\mathcal{G}$ such that for every vertex group the central extension (\ref{centralExtEQ}) has omnipotent hyperbolic quotient group. Recall that property (QT) is preserved undertaking finite index subgroups (see Lemma~\ref{lem:passfiniteindex}).  Upon passing to further index subgroups in Lemma \ref{FindexSubgrpsLem}, we can assume that $G_v=H_v\times \mathbb Z$, where $H_v$ acts geometrically on $\bar Y_v$ and also we can assume $\dot G = G$.  To show the property (QT) of $G$, we must  find not only a suitable action on a finite product of  quasi-trees, but also ensure the distance of points in the image can recover word distance in the ambient group.
We briefly describe here the strategy of the proof. Details are performed in Section~\ref{FiberLineSystemSec} and Section~\ref{CKADistFormulaSec}.

Thanks to Proposition~\ref{OmnipotentProp}, we know that there exists a $G$-equivariant quasi-isometric embedding (note that $\dot G = G$) $$\dot{\mathcal X_1}\times \dot{\mathcal X_2}\times T \to T_1\times T_2\cdots \times T_n\times T$$ Here  $T_i$ (with $i \in \{1, 2, \cdots, n \}$) is a quasi-tree. As the geometry 
of space $\dot{\mathcal X_1}\times \dot{\mathcal X_2}\times T$ does not capture the distance from vertical parts of $X$, there is no way finding a quasi-isometric embedding from the orbit $Go$ to $\dot{\mathcal X_1}\times \dot{\mathcal X_2}\times T$. To overcome this obstacle, in Section~\ref{FiberLineSystemSec}, we will construct two additional quasi-trees, denoted by $\mathcal{C}_{K}(\mathbb{F}_1)$ and $\mathcal{C}_{K}(\mathbb{F}_2)$, and will show that there is indeed a $G$-equivariant quasi-isometric embedding $$\Phi \colon G o \to \mathcal{C}_{K}(\mathbb{F}_1) \times \mathcal{C}_{K}(\mathbb{F}_2) \times \dot{\mathcal X_1}\times \dot{\mathcal X_2}\times T$$ (Section~\ref{CKADistFormulaSec} is devoted to constructing $\Phi$ and verifying $G$-equivariant quasi-isometric embedding of $\Phi$). As a consequence, we obtain the desirable $G$--equivariant quasi-isometric embedding
$$
Go \to \mathcal{C}_{K}(\mathbb{F}_1) \times \mathcal{C}_{K}(\mathbb{F}_2) \times T_1\times T_2\cdots \times T_n\times T
$$ which entails property (QT) of $G$.





\section{Projection system of fiber lines}
\label{FiberLineSystemSec}

Recall we partition $T^0=\mathcal V_1\cup \mathcal V_2$ where $\mathcal V_i$ consists of vertices in $T$ with pairwise even distances. For convenience, we sometimes write $\mathcal V=\mathcal V_1$ and $\mathcal W=\mathcal V_2$.  
We note that property (QT) of a group is preserved under taking a finite index subgroup (see Lemma~\ref{lem:passfiniteindex}). Thus passing to a finite index subgroup (see Lemma \ref{lem:index2subgroup}) if necessary we could assume that $G$ is torsion-free and preserves $\mathcal{V}_i$ with $i = 1,2$. 

Note that $e=[w, v]$ is an oriented edge from $w$ towards $v$, and  $\bar e=[v,w]$   the oriented   edge from $v$ towards $w$.  For each oriented edge $e$, let $F_e$ be the corresponding boundary plane. It is clear that  $F_e=F_{\bar e}$ does not depend on the orientation.

\subsection{Desired quasi-lines}


By   Lemma  \ref{defn:vertexedgespace}, the CKA space $X$ decomposes as the union of vertex spaces $\tilde Y_v=N_D(Y_v)$ for $v\in T^0$, on which the vertex groups $G_v$ act geometrically. The center $Z(G_v)\simeq \mathbb Z$ allows to split $Y_v$ as a metric product $\bar Y_v\times \mathbb R$.  Upon passing to further finite index subgroups in Lemma \ref{FindexSubgrpsLem}, we can assume that $G_v=H_v\times \mathbb Z$, where $H_v$ acts geometrically on $\bar Y_v$. 
If the CKA action $G\curvearrowright X$ is not flip (as in \cite{NY20}), the system of fiber lines $\mathbb R$ in $Y_v=\bar Y_v\times \mathbb R$ does not  behave well with respect to the $G$-action. Following the recent work \cite{HRSS22}, we  introduce a better geometric model for vertex subgroups $G_v$, still as the metric product of $\bar Y_v$ with a quasi-line, to resolve the $G$-action on the original fiber lines.

We first explain the construction of the quasi-line obtained from a quasi-homeomorphism. The following lemma is cited from Lemma~4.2 and the proof of the Corollary~4.3 in \cite{HRSS22}.  We present their proof   as it is short and crucial for our discussion. 

\begin{lem}
\label{lem:HRSS22}
Let  $H$ be a hyperbolic group relative to a finite collection of virtually cyclic subgroups $\{E_i: 1\le i\le n\}$. Consider $G=H\times \Z$ and fix a set of elements $c_i\in E_i\times \Z$ for each $1\le i\le n$ such that $\langle c_i\rangle$ has unbounded projection to $E_i$. Then there exist  a generating set ${S}$ of $G$ and a   $(\lambda,\lambda)$-quasi-isometry $\varphi: \Cay(G, {S})\to \R$ such that  the following holds.

\begin{enumerate}
 
    \item 
    If $g\langle c_i\rangle, g'\langle c_i\rangle$ are two $\langle c_i\rangle$-cosets for $g, g'\in E_i\times \Z$, then 
    $$
    \lambda^{-1} |g\langle c_i\rangle-g'\langle c_i\rangle|_G-\lambda \le |\varphi(g\langle c_i\rangle)-\varphi(g'\langle c_i\rangle)|
    \le \lambda |g\langle c_i\rangle-g'\langle c_i\rangle|_G +\lambda$$
    where $|g\langle c_i\rangle-g'\langle c_i\rangle|_G$ denotes the distance between two subsets in $G$ equipped with a word metric relative to a finite generating set (so not the distance on $\Cay(G, {S})$).
    \item 
    With the natural action of $G\to H$, the diagonal action of $G=H\times \Z$ on $H \times \Cay(G, {S})$ is metrically proper and cobounded, where   $\Z\subset G$ acts loxodromically on $\Cay(G, {S})$ but $\langle c_i\rangle $ acts boundedly.
\end{enumerate}
\end{lem}

In applications, the choice of elements $c_i$ shall come from the fiber generator of the adjacent pieces. See Lemma \ref{lem:boundedfiberline} below.

\begin{proof}
Let $\pi_H: G=H\times \mathbb Z \to H$ and $\pi_\Z: G=H\times \mathbb Z \to \mathbb Z$ be the natural projections. Let $t_i=\pi_H(c_i)\in E_i$ be the projection to $H$ of the  element $c_i$. We then choose a quasi-homomorphism $\phi_i: H\to\mathbb R$ by \cite{HO} such that $\phi_i(t_i)=1$  but $\phi_i(E_k)=0$ if $E_k\ne E_i$. Define the  quasi-homomorphism of $G\to \R$ as follows:  for any $x\in G,$ $$\varphi(x) :=\pi_\Z(x) -\sum_{i=1}^n \pi_\Z(c_i) \cdot \bigl(\phi_i\circ \pi_H(x)\bigr)$$  By definition, $\varphi$ takes the   constant value on each $\langle c_i\rangle$-cosets. Moreover, the distance $|g\langle c_i\rangle-g'\langle c_i\rangle|_G$  is bi-Lipschitz to $|\varphi(g \langle c_i\rangle)- \varphi( g'\langle c_i\rangle)|$ with a  constant depending only on $\langle c_i\rangle$. 

To find the generating set $ S$, notice that   the homogenization   of $\varphi$ (still denoted by $\varphi$)   has a bounded distance to the original one. As $\varphi$ is unbounded, there exists $h\in G$ so that $\{\varphi(h^n)=n\varphi(h):n\in \mathbb Z\}$ is an infinite cyclic subgroup. Denote $ S:=\varphi^{-1}([0, 2\varphi(h)])$ a (possibly infinite) subset   of $G$. One can prove that   $ S$   generates $G$, and $\varphi: G\to\R$ induces a desired quasi-isometry $\varphi: \Cay(G, {S})\to\R$. See    \cite[Lemma 4.15]{ABO} for details.
\end{proof}


\subsection{New geometric model for vertex spaces}
\label{sub:basicfacts}
Recall that $G$ acts on the  Bass-Serre tree $T$ with finitely many vertex orbits. Let $\{v_1, v_2, \cdots, v_n\} \subset T$ be the full set of vertex representatives, and let $S_{v_i}$ be the (infinite) generating set  for $G_{v_i}$ given by Lemma~\ref{lem:HRSS22}. Then $G_{v_i}$ acts on the quasi-line $\mathfrak{fl}(v_i) :=\Cay(G_{v_i}, S_{v_i})$.  Let $v$ be an arbitrary vertex in $T$, so that $v = gv_i$ for some $g \in G$ and $i \in \{1, 2, \cdots, n \}$. By equivariance, we define the quasi-line $\mathfrak{fl}(v) :=g \mathfrak{fl}(v_i) = g \Cay(G_{v_i}, S_{v_i}) $, and the action of $G_{gv_{i}} = g G_{v_i} g^{-1}$ on $g \mathfrak{fl}(v_i)$ is induced from the action of $G_{v_i}$ on $\mathfrak{fl}(v_i)$.

Consider the word metric on $G$ given by a finite generating set of $G$ including a finite generating set of $G_{v_i}$ for each representative vertex $v_i$. 
Equipping each vertex group $G_v$  with a word metric,  the inclusion of $G_v$ into $G$ is a quasi-isometric embedding since $Y_v$ is quasi-isometric embedded in the CAT(0) space $X$. 

Write $X_v:=\bar{Y}_{v} \times \mathfrak{fl}(v)$ for the new geometric model for $G_v$. By Lemma~\ref{lem:HRSS22}, the diagonal action  $G_{v} \curvearrowright X_v$ is metrically proper and cobounded, and hence the induced orbital map $$G_v\longrightarrow G_vo'\subset X_v$$ is a $G_v$-equivariant quasi-isometry for any basepoint $o'=(o'_1,o_2')\in X_v$.

Let us fix a basepoint $o=(o_1,o_2)\in Y_v$. As $G_v$ acts freely and geometrically on $Y_{v} = \bar{Y}_{v} \times \R$, 
let  $$G_vo \longrightarrow G_v$$ be  a bijective $G_v$-equivariant quasi-isometry, a quasi-inverse to the orbital map of $G_v \curvearrowright Y_v$. 

Choose the same first coordinate $o_1=o_1'$ for the above basepoints $o, o'$.
Define a $G_v$-equivariant map   $\Lambda_v: {Y}_{v} \to X_v$ as the composite of the above two $G$-equivariant maps
\[
 \Lambda_v:\quad Y_{v} = \bar{Y}_{v} \times \R  \longrightarrow G_v \longrightarrow X_v= \bar{Y}_{v} \times \mathfrak{fl}(v).
\]
Define the horizontal and vertical projection maps
\begin{equation}\label{HorVerDefn}
\Lambda_v^{\hor} \colon Y_v \to \bar{Y}_v,\quad \Lambda_v^{\ver}: Y_v \to  \mathfrak{fl}(v)  
\end{equation} as the  composites  of the map $\Lambda_v$ with the   projections to the factor $\bar{Y}_v$ and $\mathfrak{fl}(v)$ respectively.  For the product space $X_v = \bar{Y}_v \times \mathfrak{fl}(v)$, we define similarly the   horizontal distance and vertical  distances  $|\cdot|^{\hor}_{X_v}$ and $|\cdot|^{\ver}_{X_v}$. In terms of these notations, we have  for any $x,y\in Y_v$:
$$
\bigl|\Lambda_{v}(x)-\Lambda_v(y)\bigr|_{X_v}^{\hor}=\bigl|\Lambda_{v}^{\hor}(x)-\Lambda_v^{\hor}(y)\bigr|_{\bar{Y}_v}$$
$$\bigl|\Lambda_{v}(x)-\Lambda_v(y)\bigr|_{X_v}^{\ver}= \bigl|\Lambda_{v}^{\ver}(x)-\Lambda_v^{\ver}(y)\bigr|_{\mathfrak{fl}(v)} 
$$
We now derive a few    important facts   from Lemma~\ref{lem:HRSS22} about   $\Lambda_v$.

\begin{lem}
\label{lem:boundedfiberline}
There exists a uniform constant $\lambda >0$ such that $  \Lambda_v$ is a $(\lambda, \lambda)$--quasi-isometry: for any $x, y \in Y_v$ then
$$
\frac{1}{\lambda} \bigl|\Lambda_v(x)- \Lambda_{v}(y)\bigr|_{X_v} - \lambda \le \bigl |x-y\bigr|_{Y_v} \le \lambda \bigl|\Lambda_v(x)- \Lambda_{v}(y)\bigr|_{X_v}+\lambda
$$
Moreover, let $Y_w$ be the adjacent piece of $Y_v$ in the CKA space $X$. Let $\ell,\ell'$ be a line in the plane $P  = Y_v \cap Y_w$ such that $\ell, \ell'$ are  fibers in $Y_w$. Then
\begin{enumerate}
    \item 
\label{lem:boundedfiberline1}
$\diam (\Lambda_{v}^{\ver}(\ell)) \le \lambda
$. In other words, $ \Lambda_{v}(\ell) \subset \bar{Y}_v \cap B(a, \lambda)$ in $\bar{Y}_v \times \mathfrak{fl}(v)$ for some $a \in \mathfrak{fl}(v)$.
\item 
\label{lem:boundedfiberline2}
Let $p  \in Y_v = \bar{Y}_v \times \R$ be any point and    $\pi_v(p)$ be the projection of $p$ into the factor $\bar{Y}_v$.  Then $\bigl | \pi_v(p)  - \Lambda_v^{\hor}(p) \bigr |_{\bar{Y}_v} \le \lambda$. 
\item
\label{lem:boundedfiberline3}
Denote by $\bigl | \ell - \ell' \bigr |_{Y_v}$   the distance of $\ell$ and $\ell'$ in $Y_v$. Then
\[
\lambda^{-1} \bigl | \ell - \ell' \bigr |_{Y_v} - \lambda \le \diam_{\mathfrak{fl}(v)} \bigl (\Lambda_{v}^{\ver}(\ell) \cup \Lambda_{v}^{\ver}(\ell') \bigr ) \le \lambda \bigl | \ell - \ell' \bigr |_{Y_v} + \lambda
\]
\end{enumerate}
\end{lem}

\begin{proof}
We first prove (\ref{lem:boundedfiberline2}). Choose the fixed basepoints $o = (o_1, o_2)$  in $Y_v$ and $o'=(o_1',o_2')$ in $\bar{Y}_v \times \mathfrak{fl}(v)$ so that their projections into the factor $\bar{Y}_v$ are the same:  $o_1=o'_1 \in \bar{Y}_v$.  Take any point $p = (a, t)$ in $Y_v = \bar{Y}_v \times \R$, so $\pi_v(p)=a$. By our definition of the $G_v$--equivariant quasi-inverse $Y_v \to G_v$, there exists a group element $g \in G_v$ so that $|go - p|_{Y_v} \le \lambda$ for some uniform constant $\lambda$. We write $g = (h,n)$ in $H_v \times \Z$. Note that $G_v$ acts on $\bar{Y}_v \times \mathfrak{fl}(v)$ diagonally, thus the image of the group element $g = (h,n)$ under the composition map $$G_{v} \to \bar{Y}_v \times \mathfrak{fl}(v) \to \bar{Y}_v$$ is $h\cdot o_1$, where the first one is the orbital map and the second one is the projection map.
If  $Y_v$ is equipped with   $L^1$-metric,  it follows that $|ho_1 - a|_{\bar Y_v}\le |g o - p|_{Y_v} \le \lambda$. As the map $\Lambda_v$ descends to the map $\bar{Y}_v \to \bar{Y}_v$ sending $a$ to $h(o_1)$. Our claim is confirmed.
$$
 \bigl|\Lambda_{v}^{\hor}(x)-\Lambda_v^{\hor}(y)\bigr|_{\bar{Y}_v} - \lambda \le d^{h}(x,y)  \le \lambda + \bigl|\Lambda_{v}^{\hor}(x)-\Lambda_v^{\hor}(y)\bigr|_{\bar{Y}_v}
$$

For the part (\ref{lem:boundedfiberline1}),  as there are only finitely many isometric types of $Y_v$ of $X$, we only need to prove   that  $\diam (\Lambda_{v}^{\ver}(\ell)) \le \lambda$ for one given $Y_v$.
Indeed, recall  that $\Lambda_{v}^{\ver}: Y_v \to   \mathfrak{fl}(v)$ factors through  $Y_v \to X_v = \bar{Y}_{v} \times \mathfrak{fl}(v)$ as the natural projection $X_v \to \mathfrak{fl}(v)$. The latter      agrees with  the quasi-homomorphism $\varphi: G_v\to \R$ up to a bounded error in the proof of Lemma~\ref{lem:HRSS22},   vanishing  on the center $Z(G_w)$. If $B(a, \lambda)$ denotes the ball at some element $a \in \mathfrak{fl}(v)$ with radius $\lambda$, it follows that $Z(G_w) o \subset \bar{Y}_v \times B(a, \lambda)$. Every fiber line $\ell$ in $Y_w$ lies in a uniform neighborhood of the orbit of a
$Z(G_w)$-coset. Our second claim is thus verified.

The part~(\ref{lem:boundedfiberline3}) is clear from our construction.
\end{proof}

\subsection{Projection maps} 
Recall  $T^0=\mathcal V_1\cup \mathcal V_2$ where $\mathcal V_i$ consists of vertices in $T$ with pairwise even distances.  
Denote $\mathbb F_1=\{\mathfrak {fl}(v): v\in \mathcal V_1 \}$ and $\mathbb F_2=\{\mathfrak {fl}(w
): w\in \mathcal V_2 \}$.  It remains to define a family of projection maps for them.

\begin{defn}[Projection maps in $\mathbb{F}_i$]
\label{defn:projectionmap}
Let  $e_1=[v,w]$, $e_2=[w,v_2]$ denote the first two (oriented) edges in $[v,v']$.
Let $F_{e_1} = Y_v \cap Y_w$ and $F_{e_2} = Y_{v_2} \cap Y_w$ be the two boundary planes of $Y_w$.
Let $\mathcal{S}_{e_1e_2}$ be the strip in $Y_w$ joining two boundary plane $F_{e_1}$ and $F_{e_2}$ of $Y_w$ (see Section~\ref{subsec:strip} for the definition of strips). We note that $\mathcal{S}_{e_1e_2} \cap F_{e_1}$ is a line in $F_{e_1}$ that is parallel to a fiber in $Y_w$. We then define  \textit{projection from $\mathfrak{fl}(v')$ into $\mathfrak{fl}(v)$} to be
\[
\Pi_{\mathfrak{fl}(v)} (\mathfrak{fl}(v') : = \Lambda_{v}^{\ver} \bigl ( \mathcal{S}_{e_1e_2} \cap F_{e_1} \bigr )
\]
where $\Lambda_v^{\ver}$ defined in  (\ref{HorVerDefn}) is the vertical  projection to the quasi-line in $X_v=\bar Y_v\times \mathfrak{fl}(v).$
\end{defn}

\begin{lem}
\label{lem:easy1}
Let $\lambda >0$ be the constant given by Lemma~\ref{lem:boundedfiberline}.
Let $ {a},  {b}, {c}$ be distinct vertices in $ \mathcal V_i$ with $i = 1,2$. If $d_{T}( a, [b, c]) \ge 2$ then 
 $
 \Pi_{\mathfrak{fl}(a)}(\mathfrak{fl}(c)) =  \Pi_{\mathfrak{fl}(a)}(\mathfrak{fl}(b)) \le \lambda 
 $
\end{lem}
\begin{proof}
Let $[b, a]$ and $[c,a]$ be the geodesics in the tree $T$ connecting $b$ and $c$ to $w$ respectively. Let us denote $e \cdot e'$ be the last two edges in $[b,a]$ (that is also the last two edges in $[c,a]$). 
Let $\mathcal{S}_{ee'}$ be the strip in $Y_{e_{+}}$ connecting two boundary planes $F_e$ and $F_{e'}$ of $Y_{e_{+}}$
 By our definition of projection maps, we have that $
 \Pi_{\mathfrak{fl}(a)}(\mathfrak{fl}(c)) =  \Pi_{\mathfrak{fl}(a)}(\mathfrak{fl}(b)) = \Lambda_{a}^{\ver}(\mathcal{S}_{ee'} \cap P_{e'}) \le \lambda$.
 \end{proof}

\subsection{Projection axioms}
We are now going to verify that
$\mathbb{F}_i$ ($i = 1, 2$) with the above-defined projection maps in Definition~\ref{defn:projectionmap} satisfy the projection axioms (see Definition~\ref{defn:projaxioms}).
For each vertex $v \in T$, let $\mathbb{L}_v$ be the collection of boundary lines in the hyperbolic space $\bar{Y}_v$ defined at the beginning of Section~\ref{ConeoffSpaceSec}. Let $\ell_1, \ell_2$, and $\ell_3$ be three distinct boundary lines in $\mathbb L_v$. We denote 
\[
d_{\ell_1}( \ell_2, \ell_3)  = \diam \bigl (\pi_{\ell_1}(\ell_2) \cup  \pi_{\ell_1}(\ell_3) \bigr )
\]
where $\pi_{\ell_i}(\ell_j)$ is the shortest projection of $\ell_j$ to $\ell_i$ in the CAT(0) hyperbolic space $\bar{Y}_v$ (note that $\bar{Y}_v$ is a hyperbolic space since $H_v$ acts geometrically on $\bar{Y}_v$ and $H_v$ is a non-elementary hyperbolic group).
Recall that 
$$
d_{\mathfrak{fl}(v_1)} (\mathfrak{fl}(v_2), \mathfrak{fl}(v_3)) :=\diam \bigl (\Pi_{\mathfrak{fl}(v_1)} (\mathfrak{fl}(v_2)) \cup \Pi_{\mathfrak{fl}(v_1)} (\mathfrak{fl}(v_3)) \bigr ).
$$

\begin{lem}
\label{lem:easy3}
There exists a uniform constant $\lambda >0$ such that the following holds. Let $v_1, v_2, v_3$ be distinct vertices in $\mathcal{V}_1$ such that $v_1, v_2, v_3$ are in $\Lk(o)$ for some vertex $o$ in $\mathcal{V}_2$. Let $e_i$ denote the edge $[v_i, o]$ with $i =1,2,3$ and let $F_{e_i}$ be the plane in $X$ associated to $e_i$. For each $i = 1,2,3$, let $\ell_i$ denote the boundary line of $\bar{Y}_o$ that is the projection of $F_{e_i}$ into $\bar{Y}_o$. Then
\[
\frac{1}{\lambda} d_{\ell_1}( \ell_2, \ell_3) - \lambda \le d_{\mathfrak{fl}(v_1)} (\mathfrak{fl}(v_2), \mathfrak{fl}(v_3)) \le \lambda d_{\ell_1} (\ell_2, \ell_3) + \lambda
\]
\end{lem}

\begin{proof}
Let $\mathcal{S}_{e_1e_2}$ and $\mathcal{S}_{e_1e_3}$  be the strips in $Y_o$ connecting the planes $F_{e_1}$ to $F_{e_2}$ and $F_{e_1}$ to $F_{e_3}$ respectively.
We denote the line $\mathcal{S}_{e_1e_2} \cap F_{e_1}$ by $\ell$ and denote the line $\mathcal{S}_{e_1e_3} \cap F_{e_1}$ by $\ell'$. Note that both lines $\ell$ and $\ell'$ are fibers in $Y_o$. 
Recall that by our definition of projection maps, we have
$
\Pi_{\mathfrak{fl}(v_1)} (\mathfrak{fl}(v_2)) = \Lambda^{\ver}_{v_1} (\ell)$ and
$
\Pi_{\mathfrak{fl}(v_1)} (\mathfrak{fl}(v_3)) = \Lambda^{\ver}_{v_1} (\ell')
$.
By part~(\ref{lem:boundedfiberline3}) of Lemma~\ref{lem:boundedfiberline}, for some $\lambda>0$, we have that
$
\frac{1}{\lambda} \bigl | \ell -  \ell' \bigr | - \lambda \le \diam \bigl ( \Pi_{\mathfrak{fl}(v_1)} (\mathfrak{fl}(v_2)) \cup \Pi_{\mathfrak{fl}(v_1)} (\mathfrak{fl}(v_3)) \bigr ) \le \lambda \bigl | \ell -  \ell' \bigr | + \lambda
$.
Note that  $\bigl | \ell -  \ell' \bigr | = d_{\ell_1}(\ell_2, \ell_3)$
(indeed, let $\alpha$ and $\beta$ be the shortest geodesics joining $\ell_2$ to $\ell_1$ and $\ell_3$ to $\ell_1$ respectively, then $\ell$ and $\ell'$ are the product $\alpha_{+} \times \R$ and $\beta_{+} \times \R$ of endpoints of $\alpha$ and $\beta$ with the $\R$ direction in $Y_o = \bar{Y}_o \times \R$ respectively). Combining the above inequalities, we obtain a constant $\lambda' = \lambda'(\lambda) >0$ still denoted by $\lambda$ such that 
$
\frac{1}{\lambda} d_{\ell_1} (\ell_2, \ell_3) - \lambda \le \diam \bigl (\Pi_{\mathfrak{fl}(v_1)} (\mathfrak{fl}(v_2)) \cup \Pi_{\mathfrak{fl}(v_1)} (\mathfrak{fl}(v_3)) \bigr ) \le \lambda d_{\ell_1} (\ell_2, \ell_3) + \lambda
$.
The lemma is proved.
\end{proof}

We are now going to prove the following.
\begin{lem}\label{ProjSystemFibersLem1}
There exists a constant $\xi>0$ such that for each $i \in \{1, 2 \}$, the collection $\mathbb F_i$   with projection maps $\pi_{\mathfrak {fl}(v)}$'s satisfies the  projection axioms with projection constant $\xi$.
\end{lem}
\begin{proof}

We  verify in order the projection axioms (see Definition~\ref{defn:projaxioms})  for the projection maps defined on $\mathbb F_1$. The case for $\mathbb F_2$ is symmetric. The constant $\xi$ will be defined explicitly during the proof.

\textbf{Axiom 1:}
Let $\lambda >0$ be the constant given by Lemma~\ref{lem:boundedfiberline}. Since $\mathcal{S}_{e_1e_2} \cap F_{e_1}$ is a fiber line in $Y_w$, it follows from Lemma~\ref{lem:boundedfiberline} that 
$
\diam \Lambda^{\ver}_{v} \bigl ( \mathcal{S}_{e_1e_2} \cap F_{e_1} \bigr ) \le \lambda
$.
Thus $\diam \bigl ( \Pi_{\mathfrak{fl}(v)}(\mathfrak{fl}(v')) \bigr ) \le \lambda$. Axiom~1 in Definition~\ref{defn:projaxioms} is verified.

\textbf{Axiom 2:}
    Let $ {u},  {v}, {w}$ be distinct vertices in $ \mathcal V_1$.  We will show that there exists $\xi\geq 0$ sufficiently large satisfying the following property: if $d_{\mathfrak{fl}(w)}(\mathfrak{fl}(u), \mathfrak{fl}(v) ) > \xi$, then $d_{\mathfrak{fl}(u)}(\mathfrak{fl}(w), \mathfrak{fl}(v) )  \le \xi $ or $d_{\mathfrak{fl}(v)}(\mathfrak{fl}(w), \mathfrak{fl}(u) )  \le \xi $. The constant $\xi$ will be defined explicitly during the proof. 
    Since $d_{\mathfrak{fl}(w)}(\mathfrak{fl}(u), \mathfrak{fl}(v) ) > \xi$, it follows from Lemma~\ref{lem:easy1} that there is some restriction on $w$, i.e, $w$ is either lies on $[u,v]$ or $d_{T}(w, [u,v]) =1$. 

{\it Case~1:} $w$ lies on $[u, v]$. Since $u,w,v\in \mathcal V_1$, we have $d_{T}(u, [w,v])\ge 2$ and $d_{T}(v, [u,w])\ge 2$.   Axiom 2  thus follows from Lemma~\ref{lem:easy1}.

{\it Case~2:}  $d_{T}(w, [u,v])=1$. Without loss of generality, we can assume  that    $u,v,w$ lie in the same link $\Lk(o)$  for some vertex $o$ in $\mathcal{V}_{2}$. Indeed, let $o\in [u,v]$ be adjacent to $w$ and $u',v'\in Lk(o)\cap [u,v]$.   It is clear by definition that $\pi_{\mathfrak {fl}(u)}(\mathfrak {fl}(v'))=\pi_{\mathfrak {fl}(u)}(\mathfrak {fl}(v))$ and $\pi_{\mathfrak {fl}(v)}(\mathfrak {fl}(u'))=\pi_{\mathfrak {fl}(v)}(\mathfrak {fl} (u))$.  As a result, we can thus assume that $u=u'$ and $v=v'$ lie in the link $\Lk(o)$.

Recall that $\bar Y_o$ is a $\delta$-hyperbolic space whose boundary lines $\mathbb L_{o}$ satisfy the projection axioms for a constant $\xi_0$ \cite{S13}.   We claim that $\xi=\xi_0$ is the desired constant for Axiom 2.

\begin{figure}[htb] 
\centering \scalebox{0.8}{
\includegraphics{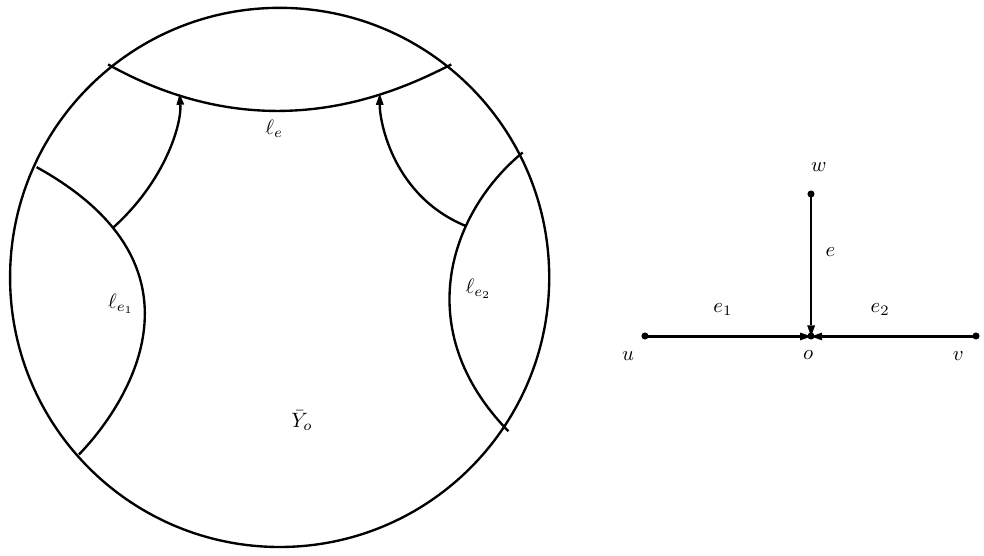} 
} \caption{Verification of Axiom 2}  \label{figure5}
\end{figure}

Denote $e=[w,o]$,  $e_1=[u,o], e_2=[v,o]$. Let $\ell_e, \ell_{e_1}, \ell_{e_2}$ be the corresponding boundary lines of $\bar Y_o$ to the oriented edges $e, e_1, e_2$. By Lemma~\ref{lem:easy3}, we have 
\[
\frac{1}{\lambda} d_{\ell_{e}}(\ell_{e_1}, \ell_{e_2}) -\lambda \le d_{\mathfrak{fl}(w)} (\mathfrak{fl}(u), \mathfrak{fl}(v))     \le \lambda d_{\ell_{e}}(\ell_{e_1}, \ell_{e_2}) + \lambda
\]

As $\mathbb  L_o$ satisfies  the projection axioms,  we see that if $d_{\ell_e}(\ell_{e_1}, \ell_{e_2})>\xi_0$, then   $d_{\ell_{e_1}}(\ell_{e}, \ell_{e_2})\le \xi_0$. Using Lemma~\ref{lem:easy3} again, we have that 
\[
\frac{1}{\lambda} d_{\ell_{e_1}}(\ell_{e}, \ell_{e_2}) - \lambda \le d_{\mathfrak{fl}(u)} (\mathfrak{fl}(w), \mathfrak{fl}(v) ) \le \lambda d_{\ell_{e_1}}(\ell_{e}, \ell_{e_2}) + \lambda
\]
Let $\xi$ be a constant such that $ \xi > \lambda \xi_{0} + \lambda$. It follows from the above inequalities that 
\[
d_{\mathfrak{fl}(u)}(\mathfrak{fl}(w), \mathfrak{fl}(v)) = \diam \bigl (\Pi_{\mathfrak{fl}(u)} (\mathfrak{fl}(w)) \cup \Pi_{\mathfrak{fl}(u)} (\mathfrak{fl}(v)) \bigr )    \le \xi 
\]
so Axiom~2 is verified.

\textbf{Axiom 3:}
 For $u\ne v\in \mathcal V_1$, the set
$$
\{w\in \mathcal V_1: d_{\mathfrak{fl}(w)}(\mathfrak {fl}({u}),\mathfrak {fl}({v}))>\xi\}
$$
is a finite set.

Indeed, by Lemma~\ref{lem:easy1},   such $w$ is either contained in the interior of $[u,v]$ or $d(w, [u,v])=1$.
The first case yields only $(d(u, v)-1)$ choices for $w$.
We now consider the case $d(w, [u,v])=1$. Since $u,v, w$ have pairwise even distance,  there exists $o\in \mathcal W\cap [u,v]^0$ and   two vertices $u',v'$  on $[u,v]$ adjacent to $o$ so that  $u', v', w\in Lk(o)$. By the projection axioms of boundary lines $\mathbb  L_o$ of $\bar Y_o$, the set of $w$ satisfying $d_{\mathfrak {fl}(w)}(\mathfrak {fl}({u}),\mathfrak {fl}({v}))>\xi$ is finite. Thus, in both cases, the set of such $w$'s is finite.  Axiom 3 is verified.
\end{proof}

\begin{lem}
\label{lem:action}
For each $i =1,2$, the collection $\mathbb F_i= \{\mathfrak {fl}({v} : v\in\mathcal V_i\}$ admits an action of the group $G$ so that
\[
\Pi_{g \mathfrak {fl}(v)}(g \mathfrak {fl}({u})) = g \Pi_{\mathfrak {fl}({v})}( \mathfrak {fl}({u}))
\]
for any $v, u\in \mathcal V_i$ and any $g\in G$.
\end{lem}

\begin{proof}
Firstly, let us recall some discussion in the beginning of Section~\ref{sub:basicfacts}.
Recall that
$\{v_1, v_2, \cdots, v_n\} \subset T$ be the full set of vertex representatives of $T$ and
for each representative vertices $v_1, v_2, \cdots v_n $ of $T$, the quasi-line $\mathfrak{fl}(v_j) $ is the Cayley graph $\Cay(G_{v_j}, S_{v_j})$ for some generating set $S_{v_j}$ of $G_{v_j}$ (see Lemma~\ref{lem:HRSS22}). 
Let $v$ be an arbitrary vertex in $T$, then $v = gv_i$ for some $g \in G$ and $i \in \{1, 2, \cdots, n \}$. The quasi-line $\mathfrak{fl}(v)$ is  given  by $g \mathfrak{fl}(v_i) = g \Cay(G_{v_i}, S_{v_i}) $, and the action of $G_{gv_{i}} = g G_{v_i} g^{-1}$ on $g \mathfrak{fl}(v_i)$ is induced from the action of $G_{v_i}$ on $ \mathfrak{fl}(v_i)$.
We are now going to show that \[
\Pi_{g \mathfrak{fl}(v)}(g \mathfrak{fl}(u)) = g \Pi_{\mathfrak{fl}(v)}( \mathfrak{fl}(u))
\]
Recall that  the family of maps $\Lambda_{gv}^{\ver} \colon Y_{gv} = g Y_v \to g \mathfrak{fl}(v)$ are $G$--equivariant:   $\Lambda_{gv}^{\ver} (gx) = g \Lambda_{v}^{\ver}(x)$ for all $x \in Y_v$. Let $e_1$ and $e_2$ be the first two edges in the geodesic $[v, u]$ with $v = (e_1)_{-}$ and $(e_1)_{+} = (e_2)_{-}$. By   Definition~\ref{defn:projectionmap} of projection map,  we have that 
$\Pi_{\mathfrak{fl}(gv)} (\mathfrak{fl}(gu)) = \diam \bigl (\Lambda_{gv}^{\ver} (\mathcal{S}_{ge_1 ge_2} \cap F_{ge_1} \bigr ) 
= \diam \bigr ( \Lambda_{gv}^{\ver}  (g (\mathcal{S}_{e_1e_2} \cap F_{e_1}) ) \bigr ) 
= \diam \bigl  ( g \Lambda_{v}^{\ver} (\mathcal{S}_{e_1e_2} \cap F_{e_1} ) \bigr ) \quad  \text{} 
= g \diam \bigl  ( \Lambda_{v}^{\ver}  (\mathcal{S}_{e_1e_2} \cap F_{e_1} ) \bigr ) 
= g \Pi_{\mathfrak{fl}(v)} (\mathfrak{fl}(u))$ for any $g\in G$.
\end{proof}

\begin{defn}
\label{rem:verticalQTtrees}
Let $\xi >0$ be the projection constant given by Lemma~\ref{ProjSystemFibersLem1}, so  the collection of  quasi-lines $\mathbb F_i= \{ \mathfrak{fl}(v) : v\in\mathcal V_i\}$ with $i=1, 2$ satisfies the projection axioms. For any fixed $K>4\xi$,   we obtain the unbounded quasi-trees of metric spaces $\mathcal C_{K} (\mathbb F_1)$ and $\mathcal C_{K} (\mathbb F_2)$ (see Section~\ref{sub:projectionaxioms}). 
 Combining Lemma~\ref{lem:action} with \cite[Theorem~4.4]{BBF}, the spaces $\mathcal C_{K} (\mathbb F_1)$ and $\mathcal C_{K} (\mathbb F_2)$ are quasi-trees and admit unbounded isometric actions  $G \curvearrowright \mathcal C_{K} (\mathbb F_1)$ and $G \curvearrowright \mathcal C_{K}(\mathbb F_2)$. The quasi-trees $\mathcal C_{K} (\mathbb F_1)$ and $\mathcal C_{K} (\mathbb F_2)$ are called {\it vertical quasi-trees} hereafter.
\end{defn}

\section{Distance formulas in the CKA space $X$}
\label{CKADistFormulaSec}
Let $\mathcal C_{K} (\mathbb F_1)$ and $\mathcal C_{K} (\mathbb F_2)$ be the vertical quasi-trees  in Definition ~\ref{rem:verticalQTtrees}. Let $\dot{\mathcal X_1}, \dot{\mathcal X_2}$ be the coned-off spaces defined in Section~\ref{coned-off spaces}.  According to the outline of the proof of Theorem~\ref{QTCKAThm}  in Section~\ref{OutlineProofSSec}, the last step to prove property (QT) of $G$ is to show that there is a $G$-equivariant quasi-isometric embedding 
$$\Phi \colon G o \to \mathcal{C}_{K}(\mathbb{F}_1) \times \mathcal{C}_{K}(\mathbb{F}_2) \times \dot{\mathcal X_1}\times \dot{\mathcal X_2}\times T$$
This section is devoted to constructing such a desired map $\Phi$ and verifying it is a quasi-isometric embedding.

{\bf We list here notations that will be used in the rest of this section.}
\begin{itemize}
    \item We fix an edge $[v_0, w_0]$ in the Bass-Serre tree $T$ such that $v_0 \in \mathcal{V}_1$.
Let  $o\in X$ be a base point in the common boundary plane $F_{[v_0, w_0]}$ between two pieces $Y_{v_0}$ and $Y_{w_0}$.

\item Assume that $x = o \in Y_{v_0}$ and $y = g o \in Y_{v_{2n}}$ for some $g \in G$ and $v_{2n} = go$. We list the vertices on the geodesic $[v_0, v_{2n}]$ by $\{v_0, v_1, \ldots, v_{2n} \}$ where $v_{2i} \in \mathcal{V}_1$ and $v_{2i+1} \in \mathcal{V}_2$. Denote $e_{i+1} = [v_i, v_{i+1}]$ the oriented edge towards $v_{i+1}$. By definition of special paths, let $p_i : = \mathcal{S}_{e_{i-1} e_i} \cap \mathcal{S}_{e_{i} e_{i+1}}$ be the intersection of two strips with $p_0 : x =o$ and $p_{2n+1} = y = go$.

\item Set $$\widetilde \alpha :=e_0\cup\alpha\cup e_{2n+1}$$ where $e_0=[w_{0}, v_0]$ and $e_{2n+1}=[v_{2n}, w_1]$. It is possible that $e_0=\bar e_1$ and $e_{2n+1}=\bar e_{2n}$, i.e,  $\widetilde \alpha$ contains backtracking at $e_0$ and $e_{2n}$.
\end{itemize}


\subsection{Construction of the desired map $\Phi$}
\label{sub:constructvartheta} It is a product of the following four maps with   the index map $\rho$ in Definition \ref{rem:indexfunction}.
\begin{itemize}
    \item We define $ \vartheta_1 \colon Go \to  \mathcal{C}_{K}(\mathbb{F}_1)$ as follows. Recall that each quasi-line   $\mathfrak{fl}(v)$ for $v\in\mathcal V_1$ embeds as a convex subset   into $\mathcal{C}_{K}(\mathbb{L}_1)$ and $\Lambda_v^{\ver}: G_vo\to \mathfrak{fl}(v)$ is a $G_v$-equivariant map. For   every $g \in G$, we set   $ \vartheta_1(go)  := \Lambda_{gv_0}^{\ver}(go)=g \Lambda_{v_0}^{\ver}(o)$. The second equality follows by $G_v$-equivariance. 

    \item 
Similarly, define $ \vartheta_2 \colon G o \to  \mathcal{C}_{K}(\mathbb{F}_2)$ by   $ \vartheta_2(go)  := \Lambda_{gw_0}^{\ver}(go)=g \Lambda_{w_0}^{\ver}(o)$ for every $g \in G$.

\item Define $\vartheta_3(o):=\pi_{\bar{Y}_{v_0}}(o)$ and extend the definition by equivariance so that $\vartheta_3(go):=g\vartheta_3(o)$ for any $g\in  G$. We thus obtain a $G$--equivariant map $\vartheta_{3} \colon G o \to \dot{\mathcal{X}_1}$.

\item Choose $\vartheta_4(o)$ to be the cone point of the hyperbolic cone  attached to the boundary line $\ell_{[v_0, w_0]}$ of $\bar{Y}_{w_0}$. We then extend   $\vartheta_4(go)=g\vartheta_4(o)$ for any $g\in G$ so that $g\vartheta_4(o)$ is the corresponding cone point to $\ell_{[gv_0, gw_0]}$ of $\bar{Y}_{gw_0}$. We thus obtain a $G$--equivariant map $\vartheta_{4} \colon G o \to \dot{\mathcal{X}_2}$.
\end{itemize}
We then define 
\begin{equation}
    \tag{$\clubsuit$}
    \label{map}
    \Phi \colon   G o \to \mathcal C_K(\mathbb F_1)\times \mathcal C_K(\mathbb F_2)\times (\dot{\mathcal X_1}, d^K_{\dot{\mathcal X_1}}) \times (\dot{\mathcal X_2}, d^K_{\dot{\mathcal X_2}}) \times T
\end{equation}
 by
\[
\Phi:= \vartheta_1\times  \vartheta_2 \times \vartheta_3\times \vartheta_4 \times \rho
\]
where  $\dot{\mathcal X_i}$ for $i=1,2$ are equipped with the $K$-thick distance $d^K_{\dot{\mathcal X_i}}$ (not genuine distance)  defined in (\ref{ThickDistEQ}), and the other three spaces are with length metric. By abuse of language, we call  the sum of the distances over the factors as $L^1$-metric on the     product space.  

The remainder of this section is to verify the following.
\begin{prop}
\label{prop;techinicalQIE}
    The map $\Phi$ in (\ref{map}) is a $G$--equivariant quasi-isometric embedding.
\end{prop}

\paragraph{\textbf{Idea of the proof of Proposition~\ref{prop;techinicalQIE}}:} Since the orbital map of any   isometric action is  Lipschitz (e.g. see  \cite[Lemma~I.8.18]{BH99}), we will only need  to give a linear upper bound on $|x-y|_X$. 
Recall from (\ref{hvdistancedefn}), for any $x,y\in X$, 
$$
|x-y|_X \sim \bigl|x-y\bigr|_X^{\hor}+\bigl|x-y\bigr|_X^{\ver}
$$
where $\bigl|x-y\bigr|_X^{\hor}=\sum_{i=0}^{2n} \bigl|p_{i} - p_{i+1}\bigr|_{Y_{v_i}}^{\hor}$ and $\bigl|x-y\bigr|_X^{\ver}=\sum_{i=0}^{2n} \bigl|p_{i} - p_{i+1}\bigr|_{Y_{v_i}}^{\ver}$.

Recall from Section~\ref{sub:basicfacts}, we build a new geometric model $X_v$ of $Y_v$   for each vertex $v$ in the Bass-Serre tree $T$. Namely, we have a $G_v$--equivariant quasi-isometric map 
$
\Lambda_{v} \colon Y_{v}=\bar{Y}_{v} \times  \R \to X_v= \bar{Y}_{v} \times \mathfrak{fl}(v)
$. 
For $x,y\in Go$, we shall  accordingly replace $\bigl|x-y\bigr|_X^{\ver}$ by the following quantity
\begin{align}
\label{verticaldistance}
V(x, y) := 
\sum_{0\le i\le 2n}
 \bigl |\Lambda_{v_i}^{\ver} (p_i) - \Lambda_{v_i}^{\ver} (p_{i+1}) \bigr |_{\mathfrak{fl}(v_i)}
\end{align}
To be precise,  we first prove in Lemma~\ref{lem:upperboundd} the following
$$
|x-y|_X \le \epsilon \left (|\rho(x)- \rho(y)|_T + \bigl|x-y\bigr|_X^{\hor} +  V(x,y) \right )
$$
We then find suitable upper bounds of $V(x,y)$ (see Proposition~\ref{vDistFormulaProp}) and $\bigl|x-y\bigr|_X^{\hor} $ (see Lemma~\ref{lem:upperdh}). 

\subsection{Verifying $\Phi$ is a quasi-isometric embedding}
In this section, we will verify that the map $\Phi$ in (\ref{map}) is a quasi-isometric embedding.

\subsubsection{Upper bound of the distance  $|x-y|_X$ on $X$}
\begin{lem}
\label{lem:upperboundd}
Let $x,y\in Go$. The exists a constant $\epsilon > 0$ such that
\begin{equation}
\label{equ:upperbound1}
    |x-y|_X \le \epsilon \left (|\rho(x)- \rho(y)|_T + \bigl|x-y\bigr|_X^{\hor} +  V(x,y) \right ) 
\end{equation}
\end{lem}

\begin{proof}
 Recall that $p_0 = x$ and $p_{2n+1} = y$. Using the triangle inequality we have
$$
|x-y|_X \le \sum_{i=0}^{2n} |p_{i} - p_{i+1}|_{Y_{v_i}}
$$
Note that $2n=|\rho(x)-\rho(y)|_T$ and $\bigl|x-y\bigr|_X^{\hor}=\sum_{i=0}^{2n} \bigl|p_{i} - p_{i+1}\bigr|_{Y_{v_i}}^{\hor}$. The proof is then completed by summing over $0\le i\le 2n$ the following inequality (\ref{equ:upperpipi+1}).

\begin{claim} 
There exists a uniform constant $\epsilon' >0$ such that for any $i \in \{0, 1, \cdots, 2n \}$,
\begin{equation}
\label{equ:upperpipi+1}
    |p_i- p_{i+1}|_{Y_{v_i}} \le \epsilon' + \epsilon'\, \bigl|p_i- p_{i+1}\bigr|_{Y_{v_i}}^{\hor} + \epsilon'\, \bigl |\Lambda^{\ver}_{v_{i}}(p_{i})- \Lambda^{\ver}_{v_{i}} (p_{i+1}) \bigr |_{\mathfrak{fl}(v_i)}. 
\end{equation}
\end{claim}
\begin{proof}[Proof of the Claim]
Indeed, since $ \Lambda_{v_i}: Y_{v_i} \to X_{v_i}$ is a quasi-isometry by Lemma~\ref{lem:boundedfiberline}, we then have
\begin{align*}
|p_i- p_{i+1}|_{Y_{v_i}} \sim_{\lambda} \bigl |\Lambda_{v_{i}}(p_{i})-  \Lambda_{v_{i}}(p_{i+1}) \bigr|_{X_{v_{i}}}  
\end{align*}
Using part~(2) of Lemma~\ref{lem:boundedfiberline} we have that $$\bigl|\Lambda^{\hor}_{v_{i}}(p_{i})- \Lambda^{\hor}_{v_{i}}(p_{i+1})\bigl|_{\bar{Y}_{v_i}} \sim_{\lambda} \bigl|p_i- p_{i+1}\bigr|_{Y_{v_{i}}}^{\hor}$$
It implies that
\begin{align*}
  & \quad \quad\bigl |\Lambda_{v_{i}}(p_{i}))-  \Lambda_{v_{i}}(p_{i+1}) \bigr |_{X_{v_{i}}}  \\
   &  \sim_{\sqrt{2}}\;  \bigl|\Lambda^{\hor}_{v_{i}}(p_{i})- \Lambda^{\hor}_{v_{i}}(p_{i+1})\bigl|_{\bar{Y}_{v_i}} +  \bigl |\Lambda^{\ver}_{v_{i}}(p_{i})- \Lambda^{\ver}_{v_{i}} (p_{i+1}) \bigr |_{\mathfrak{fl}(v_i)} \\
     & \sim_{\lambda}\;\;   \bigl|p_i- p_{i+1}\bigr|_{Y_{v_{i}}}^{\hor} +  \bigl |\Lambda^{\ver}_{v_{i}}(p_{i})- \Lambda^{\ver}_{v_{i}} (p_{i+1}) \bigr |_{\mathfrak{fl}(v_i)} 
    \end{align*}
where the first coarse equality holds by definition of $\Lambda_{v_i}^{\hor}$ and $\Lambda_{v_i}^{\ver}$.    
Hence there exists a uniform  constant $\epsilon' > 0$ such that the inequality (\ref{equ:upperpipi+1}) holds.
The above claim is proved.
\end{proof}

The lemma is proved.
\end{proof}

\subsubsection{Preparation for upper bounds of $V(x,y)$ and $|x-y|_X^{\hor}$}
Fix $K\ge 4\xi$ where the constant $\xi > \lambda$ is given by Lemma \ref{ProjSystemFibersLem1}. Let  $\mathcal{C}_K(\mathbb F_1)$ and $\mathcal{C}_{K}(\mathbb F_2)$ be the vertical quasi-trees given by Remark~\ref{rem:verticalQTtrees}.
With $i \in \{1,2\}$, Proposition \ref{BBFDistanceProp} gives the   distance formula
\begin{equation}
    \tag{$\circledast$}
    \label{eqn:distanceformula}
    \bigl|\vartheta_{i}(x)- \vartheta_{i}(y)\bigr|_{\mathcal C_K(\mathbb F_i)} \;\sim_K\; \sum_{\mathfrak {fl}(w) \in \mathbb F_i} [d_{\mathfrak {fl}(w)}(\vartheta_{i}(x), \vartheta_{i}(y))]_K
\end{equation}

To give an appropriate upper bound of $V(x,y)$, we need the following two technical lemmas (Lemma~\ref{VDistLem} and Lemma~\ref{WDistLem}).
\begin{lem}\label{VDistLem}
For any $v_{2i}\in  [v_0,v_{2n}]$ with $0\le i \le n$, we have
\begin{align}
\label{FiberDistEQ}
d_{\mathfrak{fl}(v_{2i})} \bigl ( \vartheta_1(x),  \vartheta_1(y) \bigr ) \sim_{\lambda}  \bigl | \Lambda_{v_{2i}}^{\ver} (p_{2i}) - \Lambda_{v_{2i}}^{\ver} (p_{2i+1}) \bigr |_{\mathfrak{fl}(v_{2i})}
\end{align}
For any $0\le i \le n-1$, we have
\begin{align}
\label{FiberDistEQ2}
d_{\mathfrak{fl}(v_{2i+1})} \bigl ( \vartheta_2(x),  \vartheta_2(y) \bigr )\sim_{\lambda}  \bigl  | \Lambda^{\ver}_{v_{2i+1}} (p_{2i+1}) -  \Lambda^{\ver}_{v_{2i+1}}(p_{2i+2}) \bigr |_{\mathfrak{fl}(v_{2i+1})}
\end{align}
\end{lem}

\begin{proof}
We first prove (\ref{FiberDistEQ}) for the case $0< i< n$.  The  cases $i=0$ or $i=n$ are similar.   

Note  that  $\ell_1:=\mathcal{S}_{e_{2i-1} e_{2i}} \cap F_{e_{2i}}$ is a fiber line of $Y_{v_{2i-1}}$ containing $p_{2i}$, and  similarly, $\ell_2:=\mathcal{S}_{e_{2i+1} e_{2i+2}} \cap F_{e_{2i+1}}$ contains $p_{2i+1}$.  By Definition~\ref{defn:projectionmap} of projection maps, we have
\[
\Pi_{\mathfrak{fl}(v_{2i})} ( \mathfrak{fl}(v_{0})) = \Lambda_{v_{2i}}^{\ver} \bigl (\ell_1 \bigr )
\] and
\[
\Pi_{\mathfrak{fl}(v_{2i})} ( \mathfrak{fl}(v_{2n})) = \Lambda_{v_{2i}}^{\ver} \bigl (\ell_2 \bigr )
\]

Let $\lambda >0$ be the constant given by Lemma~\ref{lem:boundedfiberline}, so the fiber lines $\ell_1, \ell_2$ are sent by $\Lambda_{v_{2i}}^{\ver}$       into $\mathcal L_{v_{2i}}$ as   subsets of diameter at most $\lambda$: $$\diam{\Lambda_{v_{2i}}^{\ver}(\ell_1)}, \;\diam{\Lambda_{v_{2i}}^{\ver}(\ell_2)}\le \lambda$$ By definition of $ \vartheta_1$, we have
 $ \vartheta_1(x) = \Lambda_{v_0}^{\ver}(x) \in \mathfrak{fl}(v_{0}) $ and $ \vartheta_1(y)=\Lambda_{v_{2n}}^{\ver}(y)\in \mathfrak{fl}(v_{2n})$. Thus 
$$
d_{\mathfrak{fl}(v_{2i})} \bigl ( \vartheta_1(x),  \vartheta_1(y) \bigr ) \sim_{\lambda} d_{\mathfrak{fl}(v_{2i})} \bigl (\mathfrak{fl}(v_{0}), \mathfrak{fl}(v_{2n})\bigr ) 
$$
As $p_{2i}\in \ell_1$ and $p_{2i+1}\in \ell_2$, we obtain  
$$
d_{\mathfrak{fl}(v_{2i})} \bigl (\mathfrak{fl}(v_{0}), \mathfrak{fl}(v_{2n})\bigr )   \sim_{\lambda}  \bigl | \Lambda_{v_{2i}}^{\ver} (p_{2i}) - \Lambda_{v_{2i}}^{\ver} (p_{2i+1}) \bigr |_{\mathfrak{fl}(v_{2i})}
$$
completing the proof of (\ref{FiberDistEQ}).

We are now going to prove (\ref{FiberDistEQ2}).  If  $w_0\ne v_1$ or  $1 \le i \le n-1$, the same proof for (\ref{FiberDistEQ}) proves (\ref{FiberDistEQ2}).
We now consider $w_0=v_1$ and  $i= 0$. In this case, we note that $e_0=\bar e_1$. By definition, we have that $ \vartheta_2(x) =  \vartheta_2(o) = \Lambda^{\ver}_{w_0}(o) \in \mathfrak{fl}(w_0)$, so we obtain $\Pi_{\mathfrak{fl}(v_1)}( \vartheta_2(x))= \vartheta_2(x)$. Recall that $\mathcal{S}_{xe_1}$ is the strip in $Y_{v_0}$ over the shortest arc from $x$ to $F_{e_1}$ (See construction of special path). As $x\in F_{e_0} = F_{e_1}$,  we have $\ell_1 := \mathcal{S}_{xe_1}$ is a fiber line of $Y_{v_0}$ that passes through $x$ and also $p_1$. Thus, $ \vartheta_2(x)\in \Pi_{\mathfrak{fl}(v_1)}(\ell_1)$. 

Recall that $\mathcal{S}_{xe_1}$ is the strip in $Y_{v_0}$ over the shortest arc from $x$ to $F_{e_1}$ (See construction of special path). As $x\in F_{e_0} = F_{e_1}$,  we have $\ell_1 := \mathcal{S}_{xe_1}$ is a fiber line of $Y_{v_0}$ that passes through $x$ and also $p_1$. Thus, $ \vartheta_2(x)\in \Pi_{\mathfrak{fl}(v_1)}(\ell_1)$. 
Let $\ell_2=\mathcal{S}_{e_2e_3} \cap F_{e_2}$ be the fiber line on $Y_{v_2}$ that passes through $p_2$.  If $w_1=v_1$, then  $\alpha = [v_0,v_1][v_1,v_2]$ and $y \in F_{e_2}$. By the same reason,   $\ell_2$ passes through $y$, so   $ \vartheta_2(y)\in \Pi_{\mathfrak{fl}(v_2)}(\ell_2)$. If   $w_1\ne v_1$, the projection $\Pi_{\mathcal L_{v_1}}( \vartheta_2(y))$ must be contained in $\Pi_{\mathfrak{fl}(v_1)}(\ell_2)$.   In both cases, we have $$
d_{\mathfrak{fl}(v_1)}( \vartheta_2(x),  \vartheta_2(y)) \sim_{\lambda} \diam (\Lambda^{\ver}_{v_1}(\ell_1) \cup \Lambda^{\ver}_{v_1} (\ell_2) )  
$$
where  we use    $\diam \bigl ( \Lambda^{\ver}_{v_1} (\ell_1) \bigr ),\; \diam \bigl ( \Lambda^{\ver}_{v_1} (\ell_2) \bigr ) \le \lambda$ by Lemma~\ref{lem:boundedfiberline}.
For $p_1\in \ell_1$ and $p_2\in \ell_2$,  we obtain
$$
\diam (\Lambda^{\ver}_{v_1}(\ell_1) \cup \Lambda^{\ver}_{v_1} (\ell_2) ) \sim_{\lambda} \bigl | \Lambda^{\ver}_{v_1}(p_1) - \Lambda^{\ver}_{v_1}(p_2) \bigr |_{\mathfrak{fl}(v_1)}
$$
concluding the proof of (\ref{FiberDistEQ2}).  
\end{proof}

Let us recall the notation from Section~\ref{sub:projectionaxioms}. Let $x\in \mathfrak{fl}(v), z \in  \mathfrak{fl}(u) \in \mathbb {F}_i $.

If $ \mathfrak{fl}(v) \ne  \mathfrak{fl}(u) \ne \mathfrak{fl}(w)$ then  $$d_{\mathfrak{fl}(w)}(x, z) := d_{\mathfrak{fl}(w)} (\mathfrak{fl}(v), \mathfrak{fl}(u))$$

If $\mathfrak{fl}(w) = \mathfrak{fl}(v) , \mathfrak{fl}(w) \ne \mathfrak{fl}(u)$, define $d_{\mathfrak{fl}(w)} (x, z) := \diam(\pi_{\mathfrak{fl}(w)} (x,  \mathfrak{fl}(u)))$.

If $\mathfrak{fl}(v) = \mathfrak{fl}(u) = \mathfrak{fl}(w)$, let $d_{\mathfrak{fl}(w)}(x, z)$ be the distance in $\mathfrak{fl}(w)$.

\begin{lem}\label{WDistLem}
Let $\vartheta_2$ be the map given by Section~\ref{sub:constructvartheta}. Let $v$ be a vertex in $v\in \Lk(v_{2i})\setminus  \widetilde{\alpha}$ and let $e=[v, v_{2i}]$. Let $\ell_{e}$, $\ell_{e_{2i}}$, and $\ell_{\bar e_{2i+1}}$ be the boundary lines of $\widetilde{F}_{v_{2i}}$ associated to distinct edges $e, e_{2i}$ and $\bar e_{2i+1}$ respectively. 
Then we have
\begin{align}\label{HDistEQ2}
d_{\mathfrak {fl}(v)}(\vartheta_2(x), \vartheta_2(y))\sim_{\lambda} d_{\ell_e}(\ell_{e_{2i}}, \ell_{\bar e_{2i+1}})
\end{align}
\end{lem}

\begin{proof}
Note that $\ell_{e}$, $\ell_{e_{2i}}$, and $\ell_{\bar e_{2i+1}}$ are the projection of planes $F_e$, $F_{e_{2i}}$, $F_{e_{2i+1}}$ of $Y_{v_{2i}}$ into the factor $Y_{v_{2i}}$.
We prove (\ref{HDistEQ2}) case by case, according to  the configuration of $e_0, e_{2n+1}$ with $\alpha$.

\textbf{Case 1}. $0<i<n$. By Definition~\ref{defn:projectionmap} of projection maps, the projection of $ \vartheta_2(x) =\Lambda^{\ver}_{w_0}(o) \in \mathfrak{fl}(w_0)$ to $\mathfrak{fl}(v)$ is the same as that of $\mathfrak{fl}(v_1)$ to $\mathfrak{fl}(v)$, and the projection of $ \vartheta_2(y) \in \mathfrak{fl}(w_1)$ to $\mathfrak{fl}(v)$ is the same as that of $\mathfrak{fl}(v_{2n-1})$ to $\mathfrak{fl}(v)$. That is to say,   $d_{\mathfrak{fl}(v)} \bigl ( \vartheta_2(x),  \vartheta_2(y) \bigr )=d_{\mathfrak{fl}(v)} \bigl (\mathfrak{fl}(v_1), \mathfrak{fl}(v_{2n-1}) \bigr ) $. Hence,  the Equation  (\ref{HDistEQ2})    follows by Lemma~\ref{lem:easy3}: 
$
d_{\mathfrak{fl}(v)} \bigl (\mathfrak{fl}(v_1), \mathfrak{fl}(v_{2n-1}) \bigr ) \sim_{\lambda} d_{\ell_e}(\ell_{e_{2i}}, \ell_{\bar e_{2i+1}})
$
for any $v\in \Lk(v_{2i})\setminus  \widetilde{\alpha}.$

\textbf{Case 2}. $i=0$ or $i=n$. We only consider the case $i=0$ and analyze the configuration of $w_0$ with $\alpha$.  The analyze for the case for $i=n$ and $w_1$ is symmetric. 

{Case~2.1:} $w_0\ne v_1$. In this case $e_0\cdot \alpha$ is a geodesic from $w_0$ to $v_{2n}$. By Definition~\ref{defn:projectionmap} of projection maps, no matter whether $\bar e_{2n+1}=e_{2n}$ holds or not,  the projection of $ \vartheta_2(x)  \in \mathfrak{fl}(w_0)$ to $\mathfrak{fl}(v)$ is the same as that of $\mathfrak{fl}(w_0)$ to $\mathfrak{fl}(v)$, and the projection of $ \vartheta_2(y) \in \mathfrak{fl}(w_1)$ to $\mathfrak{fl}(v)$ is the same as that of $\mathfrak{fl}(v_{2n-1})$ to $\mathfrak{fl}(v)$. By Lemma~\ref{lem:easy3}, we have 
$
d_{\mathfrak{fl}(v)} \bigl (\mathfrak{fl}(w_0), \mathfrak{fl}(v_{2n-1}) \bigr ) \sim_{\lambda} d_{\ell_e}(\ell_{e_{2i}}, \ell_{\bar e_{2i+1}})
$.

{Case~2.2: $w_0 = v_1$. No matter whether $w_0 = w_1$ or not,  we have 
$
d_{\mathfrak{fl}(v)}( \vartheta_2(x),  \vartheta_2(y))\le \Pi_{\mathfrak{fl}(v)}(\mathfrak{fl}(w_0))\le \xi
$
where $\xi$ is the projection constant given by Lemma~\ref{ProjSystemFibersLem1}.
On the right side of (\ref{HDistEQ2}), we have $d_{\ell_e}(\ell_{e_{2i}}, \ell_{\bar e_{2i+1}})$ is bounded above by $\xi$ for $i=0$ (as $e_0=\bar e_1$). Thus     (\ref{HDistEQ2}) holds as well in this case.
}
\end{proof}

\subsubsection{Upper bound of $V(x,y)$}
Let $ \vartheta_1$ and $ \vartheta_2$ be the maps defined in Section~\ref{sub:constructvartheta}.
We now have prepared all ingredients for the proof of the following result. 
\begin{prop}
\label{vDistFormulaProp}
Let $x, y\in Go$ and   $\alpha:=[\rho(x), \rho(y)]$  be the geodesic in $T$.  Then
\begin{equation}\label{vdistEQ}
V(x, y) \preceq_K \sum_{j=1,2} \left(\sum_{ v \in \alpha \cap \mathcal V_j} [d_{\mathfrak{fl}(v)} (\vartheta_j(x), \vartheta_j(y))]_{K} \right)  +d_{T}(\rho(x),\rho(y))
\end{equation}
\end{prop}

\begin{proof}
The goal is to recover the sum on the right side of (\ref{verticaldistance}), that is
$$
V(x,y)= 
\sum_{0\le i\le 2n}
 \bigl |\Lambda^{\ver}_{v_i} (p_i) - \Lambda^{\ver}_{v_i} (p_{i+1}) \bigr |_{\mathfrak{fl}(v_i)}
 $$
via the maps $ \vartheta_1$ and $ \vartheta_2$.
By Lemma \ref{VDistLem}, we  have desired inequalities   (\ref{FiberDistEQ})
for even indices $v_{2i}\in  [v_0,v_{2n} ] \cap \mathcal V_1$ with $0\le i \le n$, that is
\[
d_{\mathfrak{fl}(v_{2i})} \bigl ( \vartheta_1(x),  \vartheta_1(y) \bigr ) \sim_{\lambda}  \bigl | \Lambda_{v_{2i}}^{\ver} (p_{2i}) - \Lambda_{v_{2i}}^{\ver} (p_{2i+1}) \bigr |_{\mathfrak{fl}(v_{2i})}
\]

By Lemma \ref{WDistLem}, the  inequalities   (\ref{FiberDistEQ2}) recover the odd indices $v_{2i+1}\in  [v_0,v_{2n}] \cap \mathcal V_2$ with $0\le i \le n-1$ in (\ref{verticaldistance}), that is,
\[
d_{\mathfrak{fl}(v_{2i+1})} \bigl ( \vartheta_2(x),  \vartheta_2(y) \bigr )\sim_{\lambda}  \bigl  | \Lambda^{\ver}_{v_{2i+1}} (p_{2i+1}) -  \Lambda^{\ver}_{v_{2i+1}}(p_{2i+2}) \bigr |_{\mathfrak{fl}(v_{2i+1})}
\]

Plugging the inequalities (\ref{FiberDistEQ})  (\ref{FiberDistEQ2}) into (\ref{verticaldistance}), and using the term $|\rho(x)- \rho(y)|_T$ to count the additive errors in this process completes the proof of the desired inequality (\ref{vdistEQ}). Applying then the $K$-cutoff function $[\cdot]_{K}$ does not affect the inequalities.
\end{proof}

\subsubsection{Upper bound of $|x-y|_X^{\hor}$} 
The   horizontal distance $d^h$ defined in (\ref{hvdistancedefn}) of the special path $\gamma$ from $x$ to $y$ records the totality of the projected distances to the base hyperbolic spaces $\bar Y_v$:
$$
\begin{array}{rl}
 \bigl|x-y\bigr|_X^{\hor} & = \bigl|x- p_1\bigr|_{Y_{v_1}}^{\hor} + \bigl|p_1- p_2\bigr|_{Y_{v_2}}^{\hor} +\cdots+ \bigl|p_{2n}- y\bigr|_{Y_{v_{2n}}}^{\hor} \\
 & \\
     & = |\vartheta_3(x)- F_{e_1}|_{Y_{v_0}}+\displaystyle\sum_{i=1}^{2n-1} |F_{e_i}- F_{e_{i+1}}|_{Y_{v_i}} + |F_{e_{2n}}-\vartheta_3(y)|_{Y_{v_{2n}}}
\end{array}
$$
where the map $\vartheta_3$ defined in Section \ref{sub:constructvartheta} sends a point in $Y_v=\bar Y_v\times \mathbb R$  to the hyperbolic base  $\bar Y_v$.

Before moving on, let us introduce more notations to represent the horizontal distance.
Let $x_0=\vartheta_3(x), y_0\in F_{e_1}$ and $x_{2n}\in F_{e_{2n}}, y_{2n}=\vartheta_3(y)$ so that $[x_0, y_0]$ is orthogonal to $F_{e_1}$, and $[x_{2n}, y_{2n}]$ to $F_{e_{2n}}$.  Choose $x_i\in F_{e_i}, y_i\in F_{e_{i+1}}$ so that $[x_i, y_i]$ is a geodesic in $\bar Y_{v_i}$ orthogonal to $F_{e_i}$ and $F_{e_{i+1}}$. Thus, {
\begin{equation}\label{hdistEQ}
\bigl|x-y\bigr|_X^{\hor} =\sum_{i=0}^{2n} |x_i-y_i|_{\bar Y_{v_i}}.
\end{equation}
}

Recall that $\dot{\mathcal X_1}, \dot{\mathcal X_2}$ are the coned-off spaces defined in Section~\ref{coned-off spaces}. By  Definition~\ref{defn:Kthickdistance} of the $K$-thick distance of $\dot{\mathcal X_j}$  for any $K>0$ and the Remark after it, we have
\begin{equation}\label{ThickdistEQ}\sum_{i=0}^{2n}\bigl|x_{i}-y_{i}\bigr|_{\dot Y_{v_{i}}}^K   =  \bigl|\vartheta_{3}(x)-\vartheta_{3}(y)\bigr|_{\dot{\mathcal X_1}}^K+\bigl|\vartheta_{4}(x)-\vartheta_{4}(y)\bigr|_{\dot{\mathcal X_2}}^K
\end{equation}
where $|\cdot|_{\dot Y_{v_{i}}}^K$  defined in (\ref{RHGKthickdistEQ}) is the $K$-thick distance on the coned-off space $\dot Y_{v_{i}}$. The map $\vartheta_4$ defined in Section \ref{sub:constructvartheta} sends a point $go$ in $Go$ to the hyperbolic cone point to the boundary line $\ell_{g[v_0,w_0]}$ (recall that $o$ is chosen on a common boundary plane $F_{[v_0,w_0]}$).

Hence, the $K$--thick distance (\ref{ThickdistEQ}) differs from the horizontal distance (\ref{hdistEQ}) by the amount coned-off on boundary lines.   The purpose of this subsection is to recover the loss in the coned-off    from the projection system of fiber lines.

\begin{lem}
\label{lem:upperdh}
For any $x,y\in Go$, we have
$$
    \bigl|x-y\bigr|_X^{\hor} 
    \preceq_K  \sum_{i=0}^{2n} \bigl|x_i-y_i\bigr|_{\dot Y_{v_i}}^K+    \sum_{i=0}^{2n} \sum_{w\in \Lk(v_{i}) \setminus \alpha} [d_{\mathfrak{fl}(w)}(\vartheta_j(x), \vartheta_j(y))]_K
$$
where the index  $j=1$ is chosen if $i$ is odd, otherwise $j=2$.
\end{lem}

\begin{proof}
 We consider the equation (\ref{hdistEQ}) for the horizontal distance $\bigl|x-y\bigr|_X^{\hor} $. Let $\mathbb L_{v_i}$ be the set of boundary lines of $\bar Y_{v_i}$ corresponding to the set of oriented edges $e\in St(v_i)$ (i.e., the collection $\{ F_{e} \cap \bar{Y}_{v_{i}} : e \in St(v_i) \}$).
By Lemma \ref{RHGDistFLem}, for each $0\le i\le 2n$, we have
\begin{align}\label{xiyidistEQ}
\bigl|x_i-y_i\bigr|_{\bar Y_{v_i}} \sim_K \bigl|x_i-y_i\bigr|_{\dot Y_{v_i}}^K+\sum_{\ell_e\in \mathbb L_{v_i}} [d_{\ell_e} (x_i, y_i)]_K
\end{align}
for any sufficiently large $K\gg 0$.

Let $e=[w,v_i]\in St(v_i)$ and $\ell_e\in \mathbb L_{v_i}$ be the corresponding boundary line of $\bar Y_{v_i}$. Set  $j=1$ if $i$ is odd, otherwise $j=2$.

If $e= e_i$   or $e=\bar e_{i+1}$ for {$1\le i\le 2n-1$}, then $$d_{\ell_e} (x_i, y_i)\le \xi$$ since $[x_i, y_i]$ is orthogonal to $\ell_{e}$.

We remark that when $i=0$ (the case $i =2n$ is similar), it is possible that $[x_0, y_0]$ may not be perpendicular to $\ell_e$. However, we have $$d_{\ell_{e}}(x_0, y_0) \preceq d_{\mathfrak{fl}(w_0)}(\vartheta_{2}(x), \vartheta_{2}(y))$$ 

Otherwise, if $e\ne e_i$   and $e\ne \bar e_{i+1}$ for $1\le i\le 2n-1$, we have $e\notin \alpha$ for which the following holds by  {Lemma~\ref{WDistLem} for $j =2$ and by Lemma~\ref{lem:easy3} for $j=1$} , $$
d_{\mathfrak{fl}(w)}(\vartheta_j(x), \vartheta_j(y)) \sim  d_{\ell_e} (x_i, y_i).
$$
Note that $A\le \lambda B+C $ with $B\ge K\ge C$ implies $[A]_K\preceq_K [B]_K$. Thus, for each $0\le i\le 2n$, we   deduce from  (\ref{xiyidistEQ}) the following
\begin{align}\label{xyiHDistEQ}
\bigl|x_i-y_i\bigr|_{\bar Y_{v_i}} \sim_K \bigl|x_i-y_i\bigr|_{\dot Y_{v_i}}^K+  \sum_{w\in Lk(v_i) \setminus \alpha} [d_{\mathfrak{fl}(w)}(\vartheta_j(x), \vartheta_j(y))]_K
\end{align}
for any $K\gg 0$, where $j=1$ if $i$ is odd, and $j=2$ otherwise.
We sum up (\ref{xyiHDistEQ}) over $v_i\in \alpha$ to get the horizontal distance $d^h(x,y)$ in (\ref{hdistEQ}): 
\begin{align*}
    \bigl|x-y\bigr|_{X}^{\hor}  &  = \sum_{i=0}^{2n}  \bigl|x_i-y_i\bigr|_{\bar Y_{v_i}}  \\
    \preceq_K & \sum_{i=0}^{2n} \bigl|x_i-y_i\bigr|_{\dot Y_{v_i}}^K+    \sum_{i=0}^{2n} \sum_{w\in \Lk(v_{i}) \setminus \alpha} [d_{\mathfrak{fl}(w)}(\vartheta_j(x), \vartheta_j(y))]_K
\end{align*}     
\end{proof}

We now have prepared all ingredients in the proof of Proposition~\ref{prop;techinicalQIE}.

\begin{proof}[Proof of Proposition~\ref{prop;techinicalQIE}]
Since $\rho$, $\vartheta_i$ (with $i \in \{1,2,3,4\}$) are $G$-equivariant maps, it follows that $\Phi$ is a $G$-equivariant map.
Since the orbital map of any   isometric action is  Lipschitz (e.g. see \cite[Lemma~I.8.18]{BH99}), it suffices to give an upper bound on $d(x,y)$.

Let $\epsilon >0$ be the constant given by Lemma~\ref{lem:upperboundd}, so that 
\[
    |x-y|_X \le \epsilon \left (|\rho(x)- \rho(y)|_T + \bigl|x-y\bigr|_X^{\hor} +  V(x,y) \right ) 
\]
Appropriate upper bounds of the vertical distance $V(x, y)$  and the horizontal distance $\bigl|x-y\bigr|_X^{\hor}$ have been already treated in Proposition~\ref{vDistFormulaProp} and Lemma~\ref{lem:upperdh} respectively.
They are
\[
V(x, y) \preceq_K \sum_{j=1,2} \left(\sum_{ v \in \alpha \cap \mathcal V_j} [d_{\mathfrak{fl}(v)} (\vartheta_j(x), \vartheta_j(y))]_{K} \right)  +|\rho(x)-\rho(y)|_T
\] and
$$
 \bigl|x-y\bigr|_X^{\hor}
    \preceq_K  \sum_{i=0}^{2n} \bigl|x_i-y_i\bigr|_{\dot Y_{v_i}}^K+    \sum_{i=0}^{2n} \sum_{w\in \Lk(v_{i}) \setminus \alpha} [d_{\mathfrak{fl}(w)}(\vartheta_j(x), \vartheta_j(y))]_K
$$
where the index  $j$ depends on $i$: $j=1$  if $i$ is odd, otherwise $j=2$.
The above two inequalities yield:
\[
\bigl|x-y\bigr|_X^{\hor} + V(x, y) \preceq_{K} |\rho(x)- \rho(y)|_T + \sum_{i=0}^{2n}  |x_i- y_i|_{\dot Y_{v_i}}^K +  \sum_{i=0}^{2n} \sum_{w\in \Lk(v_{i})} [d_{\mathfrak{fl}(w)}(\vartheta_j(x), \vartheta_j(y))]_K
\]
By (\ref{eqn:distanceformula}), we have
\[
 \sum_{i=0}^{2n} \sum_{w\in \Lk(v_{i})} [d_{\mathfrak{fl}(w)}(\vartheta_j(x), \vartheta_j(y))]_K \preceq_{K} \bigl|\vartheta_{1}(x)-\vartheta_{1}(x)\bigr|_{\mathcal{C}_{K}(\mathbb{F}_{1})} + \bigl|\vartheta_{2}(x)-\vartheta_{2}(x)\bigr|_{\mathcal{C}_{K}(\mathbb{F}_{2})}
\]
It follows that
\[
\bigl|x-y\bigr|_X^{\hor} + V(x,y) \preceq_{K} |\rho(x)-\rho(y)|_T + \sum_{i=0}^{2n}  |x_i- y_i|_{\dot Y_{v_i}}^K + \sum_{i=1}^{2} \bigl|\vartheta_{i}(x)- \vartheta_{i}(x)\bigr|_{\mathcal{C}_{K}(\mathbb{F}_{i})}
\]
Plugging the thick distance formula (\ref{ThickdistEQ}) into the above inequality, we obtain
\begin{align*}
\bigl|x-y\bigr|_X^{\hor} + V(x,y) &\preceq_{K} |\rho(x)- \rho(y)|_T + \bigl|\vartheta_{3}(x)-\vartheta_{3}(y)\bigr|_{\dot{\mathcal X_1}}^K + \bigl|\vartheta_{4}(x)-\vartheta_{4}(y)\bigr|_{\dot{\mathcal X_2}}^K \\~\\ 
&+  \bigl|\vartheta_{1}(x)-\vartheta_{1}(x)\bigr|_{\mathcal{C}_{K}(\mathbb{F}_{1})}  + \bigl|\vartheta_{2}(x)- \vartheta_{2}(x)\bigr|_{\mathcal{C}_{K}(\mathbb{F}_{2})}      
\end{align*}
As  $|x- y|_X \le \epsilon  (  |\rho(x)- \rho(y)| +  \bigl|x-y\bigr|_X^{\hor}  +  V(x,y)  ) $, it is a consequence from the above inequality that 
the  map $\Phi=  \vartheta_1\times  \vartheta_2 \times \vartheta_3\times \vartheta_4\times\rho $ in (\ref{map}) is a $G$--equivariant quasi-isometric embedding from $X$ to $\mathcal{C}_{K}(\mathbb{F}_{1})\times \mathcal{C}_{K}(\mathbb{F}_{2})\times \dot{\mathcal X_1}\times \dot{\mathcal X_2}\times T$. The proof of Proposition is complete.
\end{proof}

\section{Proof of  Theorem~\ref{QTCKAThm}}\label{ProofQTCKASec}
\label{sec:proofCKA}
Let $G \curvearrowright X$ be a CKA action such that for every vertex group the central extension (\ref{centralExtEQ}) has omnipotent hyperbolic quotient group.
 Let $\dot G < G$ be the subgroup of the index at most $2$ preserving $\mathcal V_1$ and $\mathcal V_2$ given by Lemma \ref{lem:index2subgroup}.
Upon passing to further finite index subgroups in Lemma~\ref{FindexSubgrpsLem}, we obtain a finite index subgroup $G'$ of $\dot G$ such that the results in Section~\ref{FiberLineSystemSec} and Section~\ref{CKADistFormulaSec} hold for $G'$. We caution   the reader that at the beginning of Section~\ref{FiberLineSystemSec} we assume that each vertex group of $G$ is a direct product, this assumption may not hold for the original $G$, but holds in the finite index subgroup $G'$ of $G$.

As $G'$ is a subgroup of $\dot G$, it follows from 
 Proposition~\ref{OmnipotentProp} that there exists a $G'$--equivariant quasi-isometric embedding
 \[
 \eta \colon (\dot{\mathcal X_1}\times \dot{\mathcal X_2}\times T, d^K_{\dot{\mathcal X_1}}\times d^K_{\dot{\mathcal X_2}}\times d_T) \to T_1\times T_2\cdots \times T_n\times T
 \]
Applying Proposition~\ref{prop;techinicalQIE} to $G'$, we have a $G'$--equivariant quasi-isometric embedding
\[
\Phi \colon   G' o \to \mathcal C_K(\mathbb F_1)\times \mathcal C_K(\mathbb F_2)\times (\dot{\mathcal X_1}, d^K_{\dot{\mathcal X_1}}) \times (\dot{\mathcal X_2}, d^K_{\dot{\mathcal X_2}}) \times T
\]
It implies that $(id_{\mathcal{C}_{K}(\mathbb{F}_{1})} \times id_{\mathcal{C}_{K}(\mathbb{F}_{2})} \times \eta) \circ \Phi$ is a $G'$--equivariant quasi-isometric embedding from $G'\cdot o$ to the finite product of quasi-trees $\mathcal{C}_{K}(\mathbb{F}_1) \times \mathcal{C}_{K}(\mathbb{F}_1) \times  T_1\times T_2\cdots \times T_n\times T$. Thus $G'$ has property (QT), implying $G$ has property (QT).

\section{Applications: Property (QT) of 3-manifold groups}
\label{sec:application3mld}
In this section, we apply results obtained in previous sections to give a complete characterization of property (QT) of all finitely generated 3-manifold groups (Theorem~\ref{thm:QT3MFD}). Note that property (QT) is a commensurability invariant. Hence, we can always assume that all 3-manifolds are compact, orientable (by taking Scott's compact core and double cover).

Let $M$ be a compact, connected, orientable, irreducible 3-manifold with empty or tori boundary. $M$ is called \emph{geometric} if its interior admits geometric structures in the sense of Thurston, that are $S^3$, $\mathbb{E}^3$, $\mathbb{H}^3$, $S^{2} \times \mathbb{R}$, $\mathbb{H}^{2} \times \mathbb{R}$, $\widetilde{SL(2, \mathbb{R})}$, Nil and Sol. If $M$ is not geometric, then $M$ is called a {\it nongeometric $3$--manifold}. By geometric decomposition  of $3$--manifolds, there is a nonempty minimal union $\mathcal{T} \subset M$ of disjoint essential tori and Klein bottles, unique up to isotopy, such that each component of $M \backslash \mathcal{T}$ is either a Seifert fibered piece or a hyperbolic piece. $M$ is called {\it graph manifold} if all the pieces of $M \backslash \mathcal{T}$ are Seifert pieces, otherwise it is a {\it mixed manifold}.

We remark here that the geometric decomposition is slightly different from the torus decomposition, but they are closely related (if $M$ has no decomposing Klein bottle, then these two decompositions agree with each other). Such a difference can be got rid of by passing to some finite cover of $M$. Since we are only interested in   virtual properties of $3$--manifolds 
in this paper, we can always assume that  these two decompositions agree with each other (on some finite cover of $M$).

\subsection{Property (QT) of geometric 3-manifolds}

\begin{prop}
\label{prop:QTGM}
The fundamental group $\pi_1(M)$ of a geometric $3$--manifold $M$ has property (QT) if and only if $M$ does not support Sol and Nil geometry.
\end{prop}

\begin{proof}
We are going to prove the necessity. Assume that $\pi_1(M)$ has property (QT). By Lemma~\ref{QTUndistortedLem}, $\pi_1(M)$ does not contain any distorted element, while the fundamental group of a 3-manifold with Nil geometry or Sol geometry  contains  quadratically/exponentially distorted elements (for example, see \cite[Proposition~1.2]{NS20}). Hence, $M$ does not support Sol or Nil geometry.

Now, we are going to prove sufficiency. If $M$ supports geometry $\mathbb{E}^3, S^3 , S^{2} \times \mathbb{R}$, then $\pi_1(M)$ is virtually abelian so has property (QT). 
If the geometry of $M$ is $\mathbb{H}^{2} \times \mathbb{R}$ then $M$ is virtually covered by $\Sigma \times S^1$ for some hyperbolic surface $\Sigma$. Note that $\pi_1(\Sigma)$ is a residually finite hyperbolic group so it has property (QT) by \cite[Theorem~1.1]{BBF2}. Hence, $\pi_1(\Sigma) \times \mathbb{Z}$ has property (QT). Since $\pi_1(\Sigma) \times \mathbb{Z}$ is a finite index subgroup of $\pi_1(M)$, it follows that $\pi_1(M)$ has property (QT) by Lemma~\ref{lem:passfiniteindex}.
If $M$ supports geometry $\mathbb{H}^3$, $\pi_1(M)$ is virtually compact special by deep theorems of Agol and Wise (see \cite{Agol} \cite{Wise20}), thus $\pi_1(M)$ has property (QT). 

Finally, we need to show that if $M$ supports $\widetilde{SL(2, \R)}$ geometry then $\pi_1(M)$ has property (QT). To see this,  by passing to a finite cover if necessary, we could assume that $M$ is a nontrivial circle bundle over a closed surface $\Sigma$ with $\chi(\Sigma) <0$. Let $
1 \to K \to \pi_1(M) \to \pi_1(\Sigma) \to 1
$ be the short exact sequence associated with the circle bundle where $K$ is the normal cyclic subgroup of $\pi_1(M)$ generated by a fiber. Let  $\pi \colon \pi_1(M) \to \pi_1(\Sigma)$ be the surjective homomorphism in the above short exact sequence.
Note that the short exact sequence does not split since $M$ is supporting $\widetilde{SL(2,\mathbb{R})}$ geometry. According to the first paragraph in the proof of \cite[Corollary~4.3]{HRSS22}, there exists a generating set $\mathcal{S}$ of $G = \pi_1(M)$ so that $\mathcal{L}: = \Cay(G, \mathcal{S})$ is a quasi-line. Moreover, the diagonal action of $G$ on $\pi_1(\Sigma) \times \mathcal{L}$ is metrically proper and cobounded, and thus its orbital map is a quasi-isometry.
Since $\pi_1(\Sigma)$ is a residually finite hyperbolic group, it follows from \cite{BBF2} that $\pi_{1}(\Sigma)$ has property (QT). Hence there exists a finite product of quasi-trees $\prod_{i=1}^{n} T_i$ such that $\pi_1(\Sigma) \curvearrowright \prod_{i=1}^{n} T_i$ so that its orbital map is a quasi-isometric embedding. It is easy to see  that the orbital map of the diagonal action $G \curvearrowright \prod_{i=1}^{n} T_{i} \times \mathcal{L}$ of $G$ on the product $\prod_{i=1}^{n} T_{i} \times \mathcal{L}$  is a quasi-isometric embedding. Therefore $\pi_1(M)$ has property (QT).
\end{proof}

\subsection{Property (QT) of nongeometric 3-manifolds}
In this section, we are going to prove Theorem~\ref{thm:3-mfd}. Recall that a nongeometric 3-manifold is either a graph manifold or a mixed manifold.

\subsubsection{property (QT) of graph manifolds}
\label{sec:qtgraphmld}
Let $M$ be a graph manifold. Since property (QT) is preserved undertaking finite index subgroups (see Lemma~\ref{lem:passfiniteindex}), we only need to show that a finite cover of $M$ has property (QT).
By passing to a finite cover, we can assume that each Seifert fibered piece in the JSJ decomposition of $M$ is a trivial circle bundle over a hyperbolic surface of genus at least $2$, and the intersection numbers of fibers of adjacent Seifert pieces have absolute value $1$ (see\cite[Lemma~2.1]{KL98}). Also we can assume that the underlying graph of the graph manifold $M$ is bipartite since any non-bipartite graph manifold is double covered by a bipartite one. 

We note that $\pi_1(M)$ is an admissible group in the sense of Definition~\ref{defn:admissible}. However, it is not always true that $\pi_1(M)$ can acts geometrically on a CAT(0) space so property (QT) in this case does not follow immediately from Theorem \ref{QTCKAThm}. Indeed,  if $M$ is a graph manifold with nonempty boundary then it always admits a Riemannian metric of nonpositive curvature (see \cite{B95}). In particular, $\pi(M) \curvearrowright \tilde{M}$ is a CKA action, and thus property (QT) of $\pi_1(M)$ follows from Theorem~\ref{QTCKAThm}. However, many closed graph manifolds are shown to not support any Riemannian metric of nonpositive curvature (see \cite{B95}).

We remark here that the CAT(0) metric on the CKA space $X$ in Section~\ref{FiberLineSystemSec} and Section~\ref{CKADistFormulaSec} is not really essential in the proofs. Below we will make certain modifications on some steps to run the proof of Theorem~\ref{QTCKAThm} for closed graph manifolds.

\subsection*{Metrics on $M$:}
We now are going to describe a convenient metric on $M$ introduced by Kapovich--Leeb \cite{KL98}. For each Seifert component $M_v = F_{v} \times S^1$ of $M$, we choose a hyperbolic metric on the base surface $F_v$ so that all boundary components are totally geodesic of unit length, and then equip each Seifert component $M_v = F_v \times S^1$ with the product metric $d_v$ such that the fibers have length one. Metrics $d_v$ on $M_v$ induce the product metrics on $\tilde{M}_v$ which by abuse of notations is also denoted by $d_v$.

Let $M_v$ and $M_w$ be adjacent Seifert components in the closed graph manifold $M$, and let $T \subset M_v \cap M_w$ be a JSJ torus. Each metric space $(\tilde{T},d_{v})$ and $(\tilde{T}, d_{w})$ is a Euclidean plane. After applying a homotopy to the gluing map, we may assume that at each JSJ torus $T$, the gluing map $\phi$ from the boundary torus $\overleftarrow{T} \subset M_v$ to the boundary torus $\overrightarrow{T} \subset M_w$ is affine in the sense that the identity map $(\tilde{T},d_{v}) \to (\tilde{T}, d_{w})$ is affine.
We now have a product metric on each Seifert component $M_v = F_v \times S^1$. These metrics may not agree on the JSJ tori but the gluing maps are bilipschitz (since they are affine).
The product metrics on the Seifert components induce a length metric on the graph manifold $M$ denoted by $d$ (see \cite[Section~3.1]{BBI01}) for details). Moreover, there exists a positive constant $L$ such that on each Seifert component $M_v = F_v \times S^1$ we have
\[
   \frac{1}{L} \, d_v(x,y)
   \leq d(x,y)
   \leq L \, d_v(x,y)
\]
for all $x$ and $y$ in $M_v$. (See \cite[Lemma~1.8]{P05} for detailed proof of the last claim.)
Metric $d$ on $M$ induces metric on $\tilde{M}$, which is also denoted by $d$ (by abuse of notations). Then for all $x$ and $y$ in $\tilde{M}_v$ we have
\[
   \frac{1}{L} \, d_{v}(x,y)
   \leq d(x,y)
   \leq L \, d_{v}(x,y)
\]

\begin{rem}
Note that the space $(\tilde{M}, d)$ may not be a CAT(0) space but  $\pi_1(M)$ acts geometrically on $(\tilde{M}, d)$ via deck transformations.
\end{rem}

In Section~\ref{subsection:specialpath}, we define special paths on a CAT(0) space $X$. In this section, although $(\tilde{M}, d)$ is no longer a CAT(0) space, we are still able to define special paths in  $(\tilde{M}, d)$. The construction is similar to Section~\ref{subsection:specialpath} with slight changes.

\subsection*{Special paths on $\tilde M$:}
\label{subsec:specialpath}
Lift the JSJ decomposition of the graph manifold $M$ to the universal cover $\tilde M$, and let $T$ be the tree dual to this decomposition of $\tilde M$ (i.e., the Bass-Serre tree of $\pi_1(M)$).
For every pair of adjacent edges $e_1$, $e_2$ in $T$, let $v$ be the common vertex of $e_1$ and $e_2$. Let $\ell$ and $\ell'$ be two boundary lines of $\tilde{F}_v$ corresponding to the edges $e_1$ and $e_2$ respectively. Let $\gamma_{e_1 e_2}$ be the shortest geodesic joining $\ell$ and $\ell'$ in $(\tilde{M}_v, d_v )$. This geodesic determines an Euclidean strip $\mathcal{S}_{e_1 e_2} : = \gamma_{e_1 e_2} \times \R$ in $(\tilde{M}_v , d_v )$.
Let $x$ be a point in $(\tilde{M}_v, d_v)$ and $e$ be an edge with an endpoint $v$. The minimal geodesic from $x$ to the plane $F_e$ also define a strip $\mathcal{S}_{xe} : = \gamma_{xe} \times \R$ in $(\tilde{M}_v, d_v)$ where $\gamma_{xe} \subset \tilde{F}_v$ is the projection to $\tilde{F}_v$ of this minimal geodesic.

Now, let $x$ and $y$ be any two points in the universal cover $\tilde M$ of $M$ such that $x$ and $y$ belong to the interiors of pieces $\tilde{M}_v$ and $\tilde{M}_v'$ respectively.
If $v = v'$ then we define a {\it special path} in $X$ connecting $x$ and $y$ to be the geodesic $[x, y]$ in $(\tilde{M}, d)$. Otherwise, let $e_1 \cdots e_n$ be the geodesic edge path connecting $v$ and $v'$. For notational purpose, we write $e_0 : =x $ and $e_{n+1} : = y$.
Let $p_i \in F_{e_i}$ be the intersection point of the strips $\mathcal{S}_{e_{i-1}e_{i}}$ and $\mathcal{S}_{e_{i}e_{i+1}}$. The {\it special path} connecting $x$ and $y$ is the concatenation of the geodesics
\[
[x,p_1] \cdot [p_1, p_2] \cdots [p_n, y]
\]
We label $p_0 : = x$ and $p_{n+1} : = y$.

\begin{prop}
\label{prop:QTgraphmfd}
If $M$ is a graph manifold, then $\pi_1(M)$ has property (QT).
\end{prop}

\begin{proof}
If $M$ is a non-positively curved graph manifold (for example, when $M$ has nonempty boundary) then the fact $\pi_1(M)$ has property (QT) is followed from Theorem~\ref{QTCKAThm}. The only case that does not follow directly from Theorem~\ref{QTCKAThm} is when $M$ is a closed graph manifold (recall many closed graph manifolds are non-positively curved but many are not).  Since the metric $d$ on $\tilde{M}$ restricted to each piece $\tilde{M}_v$ is $L$--bilipschitz equivalent to $d_v$, so the inequalities in Section~\ref{CKADistFormulaSec} are slightly changed by a uniform multiplicative constant. For example, the statement $a \asymp_{K} b$ (or $a \preceq_K b$) in Section~\ref{CKADistFormulaSec} will be changed to $a \asymp_{K'} b$ (or $a \preceq_K' b$) for some constant $K'$ depending on $K$.  Thus, the proof, in this case,  is performed along lines with the proof of Theorem~\ref{QTCKAThm}.
\end{proof}

\subsubsection{property (QT) of mixed 3-manifolds}
Recall that a non-geometric 3-manifold with empty or tori boundary is either a graph manifold or a mixed 3-manifold. The case of graph manifold has been addressed in Section~\ref{sec:qtgraphmld}. In this section, we address the mixed 3-manifold case. 

\begin{prop}
\label{prop:3-mfdQT}
The fundamental group of a mixed 3-manifold has property (QT).
\end{prop}

The fundamental group of a mixed 3-manifold has a natural relatively hyperbolic structure as follows: Let $M_1,\cdots, M_{k}$ be the maximal graph manifold pieces, isolated Seifert fibered components of the
JSJ-decomposition of $M$, and   $S_1,\cdots,S_{l}$ be the tori in $M$ not contained in any $M_i$.  The fundamental group $G = \pi_1(M)$ is hyperbolic relative to the set of parabolic subgroups
\[
\mathcal P = \{\pi_1(M_{p}): 1\le p\le k\} \cup \{\pi_1(S_{q}): 1\le q\le l\}
\] (see \cite{Bigdely-Wise13}, \cite{Dah}).

The following lemma provides many separable subgroups in $\pi_1(M)$, generalizing \cite[Lemma~3.3]{Sun21}. The proof uses a recent result of the second author and Sun in \cite{NS20} where the authors show that separability and distortion of subgroups in 3-manifold groups are closely related.

\begin{lem}
\label{lem:separable}
Let $M$ be a compact, orientable, irreducible, 3-manifold with empty or tori boundary, with nontrivial torus decomposition and does not support the Sol geometry. If $H$ is a finitely generated, undistorted subgroup of $\pi_1(M)$, then $H$ is separable in $\pi_1(M)$.
\end{lem}


\begin{proof}
Let $M_H$ be the covering space of $M$ corresponding to
$H \le \pi_1(M)$.
Generalizing a notion called ``almost fiber
part'' in \cite{Liu17}, an embedded (possibly
disconnected) subsurface $\Phi(H)$  in   $M_H$ called ``almost
fiber surface''   is introduced in  \cite{Sun20}.  Sun shows in  \cite[Theorem 1.3]{Sun20} that all information about the separability of $H$ can
be obtained by examining the almost fibered surface. 

In \cite{NS20}, the authors introduce a notion called ``modified almost fibered surface'' (denoted $\hat{\Phi}(H)$) that modify slightly the original definition of almost fibered surface in \cite{Sun20} and show that information about the distortion of $H$ in $G$ can be also obtained by examining the ``modified almost fibered surface''.  We refer the reader to \cite{Sun20} for the definition of the almost fiber surface and to \cite{NS20} for the definition of the modified almost fiber surface. The precise definitions are not needed here, so we only state here some facts from  \cite{NS20} that will be used later in the proof.

The  torus decomposition of $M$ induces the torus decomposition of
$\Phi(H)$. 
Let $\Phi(H)$ and $\hat{\Phi}(H)$ be the almost fiber surface and modified almost fiber surface of $H$ respectively.

\begin{enumerate}

    \item Both the almost fiber surface $\Phi(H)$ and the  modified almost fiber surface $\hat{\Phi}(H)$ are (possibly disconnected) subsurfaces  of $M_H$. 
    \item 
    The almost
fiber surface $\Phi(H)$ has some piece that is
homeomorphic to the annulus and parallel to the boundary of $\Phi(H)$. We
delete these annulus pieces from $\Phi(H)$ to get the modified almost fiber
surface, and we denote it by $\hat{\Phi}(H)$.

\end{enumerate}

The surface
      $\Phi(H)$ (resp.     $\hat{\Phi}(H)$)  has a natural graph of spaces structure with  the dual graph denoted by $\Gamma_{\Phi(H)}$ (resp. $\Gamma_{\hat{\Phi}(H)}$).
By \cite[Theorem~1.4]{NS20}, every component $S$ of the modified almost fiber surface $\hat{\Phi}(H)$ must contain only one piece (otherwise, the distortion of $H$ in $\pi_1(M)$ is at least quadratic, this contradicts the fact that $H$ is undistorted in $\pi_1(M)$). This fact combined with (2) implies that the graph $\Gamma_{\Phi(H)}$ is a union of trees. By \cite[Theorem~1.3]{Sun20} (or see also  \cite[Theorem~3.2]{Sun21} for a statement) tells us that whenever $\Gamma_{\Phi(H)}$ does not contain a simple cycle then $H$ is separable. As shown above, we are in this case, hence we conclude that the subgroup $H$ is separable in $\pi_1(M)$.
\end{proof}

We now give the proof of Proposition~\ref{prop:3-mfdQT}.

\begin{proof}[Proof of Proposition~\ref{prop:3-mfdQT}]
Let $M_1, \cdots, M_k$ be the collection of maximal graph manifold components and Seifert fibered pieces in the geometric decomposition of $M$. Let $S_1, \dots, S_{\ell}$ be the tori in the boundary of $M$ that bound a hyperbolic piece, and let $T_1, \dots, T_m$ be the tori in the JSJ decomposition of $M$ that separate two hyperbolic components of the JSJ decomposition. Then $\pi_1(M)$ is hyperbolic relative to 
\[
\mathbb{P} = \{\pi_1(M_p)\}_{p=1}^k \cup \{\pi_1(S_q)\}_{q=1}^\ell \cup \{\pi_1(T_r)\}_{r=1}^m
\]
(see \cite{Bigdely-Wise13}, \cite{Dah}).

We are going to show that $G=\pi_1(M)$   satisfies all conditions in Theorem~\ref{QTRHGThm}.

{\it Claim~1:} $\pi_1(M)$ induces the full profinite topology on each $P \in \mathcal{P}$. 
Indeed, it is well-known that the fundamental groups of all compact 3-manifolds are residually finite, thus  $\pi_1(M)$ is residually finite.
Since each peripheral subgroup $P$ is undistorted in $\pi_1(M)$, it follows from Lemma~\ref{lem:separable} that $P$ is separable in $\pi_1(M)$. 
Again, by Lemma~\ref{lem:separable}, each finite index subgroup of $P$ is also separable in $\pi_1(M)$. By \cite[Lemma~2.8]{R18}, $\pi_1(M)$ induces the full profinite topology on $P$.

{\it Claim~2:} For each peripheral subgroup $P \in \mathbb{P}$, there exists a finite index subgroup $P'$ of $P$  acting isometrically on a finite number of quasi-trees so that the diagonal action of $P'$ on the finite product of these  quasi-trees induces quasi-isometric embeddings on orbital maps.
Indeed, if $P = \pi_1(T_r)$ or $P = \pi_1(S_q)$ for some $r, q$ then $\pi_1(P) = \Z^2$, we denote $P ' : = P$. If $P = \pi_1(M_j)$ for some Seifert component $M_j = F_j \times S^1$ then $P = \pi_1(F_j) \times \Z$. In this case, as $F_j$ is a hyperbolic surface with nonempty boundary, we have $\pi_1(F_j)$ is a free group, hence we choose $P' = P$ as $\pi_1(F_j)$ is a quasi-tree.  The last case we must consider is that $P = \pi_1(M_j)$ where $M_j$ is a maximal graph manifold component. Passing to an appropriate finite cover $M'_j \to M_j$ we could assume that $\pi_1(M'_j)$ acts on a finite number of quasi-trees (but they are not quasi-lines) $T_1, T_2, \cdots, T_n$'s so that the orbital map induced from the diagonal action $\pi_1(M_j) \curvearrowright \prod_{1=1}^{n} T_i$ is a quasi-isometric embedding (see Proposition~\ref{prop:QTgraphmfd}). Claim~2 is confirmed. We then repeat the proof of Theorem~\ref{thm:sepaQT} (the second and third paragraph) to show that $P$ satisfies the hypothesis of Theorem~\ref{QTRHGThm}.


In summary, we have verified the hypotheses in Theorem~\ref{QTRHGThm} for $G=\pi_1(M)$, so mixed 3-manifold groups  have property (QT).
\end{proof}

We now give the proof of Theorem~\ref{thm:3-mfd}.
\begin{proof}[Proof of Theorem~\ref{thm:3-mfd}]
    Let $M$ be a compact orientable irreducible 3-manifold with
empty or tori boundary, with nontrivial torus decomposition, and that does not support the Sol geometry. Such a 3-manifold $M$ is either a graph manifold or a mixed manifold. The graph manifold case and mixed manifold case have been addressed in Proposition~\ref{prop:QTgraphmfd} and Proposition~\ref{prop:3-mfdQT} respectively, and hence the theorem is proved.
\end{proof}

\subsection{Property (QT) of finitely generated 3-manifolds}

\begin{prop}
\label{prop:highergenus}
Let $M$ be a compact, orientable, irreducible, $\partial$--irreducible 3-manifold such that it has a boundary component of genus at least $2$. Then $\pi_1(M)$ has property (QT).
\end{prop}

\begin{proof}
We consider the following two cases:

{\it Case~1:}
$M$ has trivial torus decomposition.
In this case, $M$ supports a geometrically finite hyperbolic structure with infinite volume. We paste hyperbolic 3-manifolds with totally geodesic boundaries to $M$  to get a finite volume hyperbolic 3-manifold $N$.  By the Covering Theorem (see \cite{Can96}) and the Subgroup Tameness Theorem (see \cite{Agol04}, \cite{CG06}), a finitely generated subgroup of the finite volume hyperbolic 3-manifold $N$ is either a virtual fiber surface subgroup or undistorted.
By the construction of $N$, the subgroup $\pi_1(M) \le \pi_1(N)$ could not be a virtual fiber surface subgroup, and thus $\pi_1(M)$ must be undistorted in $\pi_1(N)$. Since $\pi_1(N)$ has property (QT), it follows that $\pi_1(M)$ has property (QT) (see Lemma~\ref{lem:passfiniteindex}).


{\it Case~2:}
We now assume that $M$ has nontrivial torus decomposition. 
By  \cite[Section~6.3]{Sun20}, we paste hyperbolic 3-manifolds with totally geodesic boundaries to $M$ to get a 3-manifold $N$ with empty or tori boundary.  The new manifold $N$ satisfies the following properties.
\begin{enumerate}
    \item $M$ is a submanifold of $N$ with incompressible tori boundary.
    \item The torus decomposition of $M$ also gives the torus decomposition of $N$.
    \item Each piece of $M$ with a boundary component of genus at least $2$ is contained in a hyperbolic piece of $N$.
\end{enumerate}
In particular, it follows from (2) and (3) that $N$ is a mixed 3-manifold. {{The subgroup $\pi_1(M)$ sits nicely in $\pi_1(N)$}}. By the proof of Case~1.2 in the proof of \cite[Theorem~1.3]{NS20}), we have that $\pi_1(M)$ is undistorted in $\pi_1(N)$ ({even more than that, $\pi_1(M)$ is strongly quasiconvex in $\pi_1(N)$ (see \cite{NTY19})}.   
Note that $\pi_1(N)$ has property (QT) by Proposition~\ref{prop:3-mfdQT}. Since $\pi_1(M)$ is undistorted in $\pi_1(N)$ and $\pi_1(N)$ has property (QT), it follows that $\pi_1(M)$ has property (QT).
\end{proof}

We now give the proof of   Theorem \ref{thm:QT3MFD} which gives a complete characterization of property (QT) for  finitely generated 3-manifolds groups.

\begin{proof}[Proof of Theorem \ref{thm:QT3MFD}]
Since $M$ is a compact, orientable $3$--manifold, it decomposes into irreducible, $\partial$--irreducible pieces $M_1, \dots, M_k$ by the sphere-disc decomposition. In particular, $\pi_1(M)$ is the free product $\pi_1(M_1)* \pi_1(M_2)* \cdots *\pi_1(M_k) *F_r$ for some free group $F_r$.  We remark here that $\pi_1(M)$ is hyperbolic relative to the collection $\PP = \{P_1, \cdots, P_k, F_r\}$ where $P_i := \pi_1(M_i)$.

We are going to prove the necessity. Assume that $\pi_1(M)$ has property (QT). Since $\pi_1(M_i)$ is undistorted in $\pi_1(M)$, it follows that $\pi_1(M_i)$ has property (QT) (see Lemma~\ref{lem:passfiniteindex}). By Proposition~\ref{prop:QTGM}, $M_i$ does not support Sol and Nil geometry.

Now, we are going to prove sufficiency.
Assume that there is no piece $M_i$ that supports either Sol or Nil geometry. We would like to show that $\pi_1(M)$ has property (QT). 
In this case, we observe that each peripheral subgroup $P \in \PP$ has property (QT). Indeed, a free group $P = F_r$  of course has property (QT), so let us now assume that $P = \pi_1(M_i)$ for some $1 \le i \le k$. If $M_i$ has a   boundary component of genus at least $2$ then property (QT) of $\pi_1(M_i)$   follows from Proposition~\ref{prop:highergenus}. Otherwise, $M_i$ has empty or tori boundary. Then the property (QT) of $\pi_1(M_i)$  follows  from  Proposition~\ref{prop:QTGM} for geometric manifolds, Proposition~\ref{prop:QTgraphmfd} for graph manifolds, and Proposition~\ref{prop:3-mfdQT} for mixed graph manifolds.  

{{We are going to show that $G=\pi_1(M)$   satisfies all conditions in Theorem~\ref{QTRHGThm}.}} The proof is similar to the proof of Proposition~\ref{prop:3-mfdQT} with minor changes.

{\it Claim~1:} $\pi_1(M)$ induces the full profinite topology on each $P_i \in \mathcal{P}$. 
It is well-known that the fundamental groups of all compact 3-manifolds are residually finite, thus  $\pi_1(M)$ is residually finite and its finite index subgroups are residually finite as well.
Any finite index subgroup $H$ of $P_i=\pi_1(M_i)$ is separable in the free product $G = P_1* P_2* \cdots *P_k *F_r$ by  \cite[Theorem~1.1]{B71}. Hence it follows from \cite[Lemma~2.8 ]{R18} that $G$ induces the full profinite topology on $P_i$.

{\it Claim~2:} For each peripheral subgroup $P \in \mathbb{P}$, there exists a finite index subgroup $P'$ of $P$  acting isometrically on a finite number of quasi-trees, so that the diagonal action of $P'$ on the finite product of these  quasi-trees induces quasi-isometric embeddings on orbital maps.
Indeed, the claim obviously holds for $P = F_{r}$ or $P =\Z^2$. The claim also holds for $P = \pi_1(M_i)$ where $M_i$ is a geometric 3-manifold. The case of graph manifolds is proved in the Claim~2 of the proof of Proposition~\ref{prop:3-mfdQT}. The only case left is when $M_i$ is a mixed 3-manifold or $M_i$ has a boundary component with genus at least $2$. It has been shown in Proposition~\ref{prop:highergenus} that if $M_i$ has a boundary component with genus at least $2$ then it is an undistorted subgroup in a mixed 3-manifold. Therefore it suffices  to consider only the mixed 3-manifold case. Recall that in the proof of Proposition~\ref{prop:3-mfdQT}, we show that there exists a finite index subgroup of $\pi_1(M_i)$ such that it is a relatively hyperbolic group, satisfying conditions in Theorem~\ref{QTRHGThm}, and thus Claim~2 is confirmed.

With Claim~1 and Claim~2, we use the same argument as in the proof of Theorem~\ref{thm:sepaQT} (see the second and third paragraph) to find a finite index normal subgroup $ G' $ of $G$ such that $ G'$ is hyperbolic relative to a  collection of subgroups satisfying the hypotheses in Theorem~\ref{QTRHGThm}, and thus $G'$ has property (QT). Therefore, $\pi_1(M)$ has property (QT) since $ G'$ is a finite index subgroup of $\pi_1(M)$ and $G'$ does have property (QT).
\end{proof}

\bibliographystyle{alpha}
\bibliography{CKadmissible}

\begin{thebibliography}{CCMT15}

\bibitem[ABO19]{ABO}
Carolyn Abbott, Sahana~H. Balasubramanya, and Denis Osin.
\newblock Hyperbolic structures on groups.
\newblock {\em Algebr. Geom. Topol.}, 19(4):1747--1835, 2019.

\bibitem[{Ago}04]{Agol04}
Ian {Agol}.
\newblock {Tameness of hyperbolic 3-manifolds}.
\newblock {\em arXiv Mathematics e-prints}, page math/0405568, May 2004.

\bibitem[Ago13]{Agol}
I.~Agol.
\newblock The virtual {H}aken conjecture.
\newblock {\em Doc. Math.}, 18:1045--1087, 2013.
\newblock With an appendix by Agol, Daniel Groves, and Jason Manning.

\bibitem[Baj07]{B07}
Jitendra Bajpai.
\newblock Omnipotence of surface groups.
\newblock Master's thesis, McGill University, 2007.

\bibitem[BBF15]{BBF}
Mladen Bestvina, Ken Bromberg, and Koji Fujiwara.
\newblock Constructing group actions on quasi-trees and applications to mapping
  class groups.
\newblock {\em Publ. Math. Inst. Hautes \'{E}tudes Sci.}, 122:1--64, 2015.

\bibitem[BBF19]{BBF2}
Mladen {Bestvina}, Kenneth {Bromberg}, and Koji {Fujiwara}.
\newblock {Proper actions on finite products of quasi-trees}.
\newblock {\em arXiv e-prints}, page arXiv:1905.10813, May 2019.

\bibitem[BBFS19]{BBFS}
Mladen Bestvina, Ken Bromberg, Koji Fujiwara, and Alessandro Sisto.
\newblock Acylindrical actions on projection complexes.
\newblock {\em Enseign. Math.}, 65(1-2):1--32, 2019.

\bibitem[BBI01]{BBI01}
D.~Burago, Y.~Burago, and S.~Ivanov.
\newblock {\em A course in metric geometry}, volume~33 of {\em Graduate Studies
  in Mathematics}.
\newblock American Mathematical Society, Providence, RI, 2001.

\bibitem[BH99]{BH99}
Martin~R. Bridson and Andr\'{e} Haefliger.
\newblock {\em Metric spaces of non-positive curvature}, volume 319 of {\em
  Grundlehren der Mathematischen Wissenschaften [Fundamental Principles of
  Mathematical Sciences]}.
\newblock Springer-Verlag, Berlin, 1999.

\bibitem[Bow08]{Bow08}
Brian~H. Bowditch.
\newblock Tight geodesics in the curve complex.
\newblock {\em Invent. Math.}, 171(2):281--300, 2008.

\bibitem[Bow12]{Bow12}
B.~H. Bowditch.
\newblock Relatively hyperbolic groups.
\newblock {\em Internat. J. Algebra Comput.}, 22(3):1250016, 66, 2012.

\bibitem[Bur71]{B71}
R.~G. Burns.
\newblock On finitely generated subgroups of free products.
\newblock {\em J. Austral. Math. Soc.}, 12:358--364, 1971.

\bibitem[But20]{Button}
J.~O. Button.
\newblock Groups acting purely loxodromically on products of hyperbolic graphs,
  2020.

\bibitem[BW13]{Bigdely-Wise13}
Hadi Bigdely and Daniel~T. Wise.
\newblock Quasiconvexity and relatively hyperbolic groups that split.
\newblock {\em Michigan Math. J.}, 62(2):387--406, 2013.

\bibitem[Can96]{Can96}
Richard~D. Canary.
\newblock A covering theorem for hyperbolic {$3$}-manifolds and its
  applications.
\newblock {\em Topology}, 35(3):751--778, 1996.

\bibitem[CCMT15]{CCMT}
Pierre-Emmanuel Caprace, Yves Cornulier, Nicolas Monod, and Romain Tessera.
\newblock Amenable hyperbolic groups.
\newblock {\em J. Eur. Math. Soc. (JEMS)}, 17(11):2903--2947, 2015.

\bibitem[CG06]{CG06}
Danny Calegari and David Gabai.
\newblock Shrinkwrapping and the taming of hyperbolic 3-manifolds.
\newblock {\em J. Amer. Math. Soc.}, 19(2):385--446, 2006.

\bibitem[CK00]{CK00}
Christopher~B. Croke and Bruce Kleiner.
\newblock Spaces with nonpositive curvature and their ideal boundaries.
\newblock {\em Topology}, 39(3):549--556, 2000.

\bibitem[CK02]{CK02}
C.~B. Croke and B.~Kleiner.
\newblock The geodesic flow of a nonpositively curved graph manifold.
\newblock {\em Geom. Funct. Anal.}, 12(3):479--545, 2002.

\bibitem[Con00]{C00}
Gregory~R. Conner.
\newblock Discreteness properties of translation numbers in solvable groups.
\newblock {\em J. Group Theory}, 3(1):77--94, 2000.

\bibitem[Dah03]{Dah}
Fran\c{c}ois Dahmani.
\newblock Combination of convergence groups.
\newblock {\em Geom. Topol.}, 7:933--963, 2003.

\bibitem[dC08]{C08}
Yves de~Cornulier.
\newblock Dimension of asymptotic cones of {L}ie groups.
\newblock {\em J. Topol.}, 1(2):342--361, 2008.

\bibitem[DJ99]{DJ99}
A.~Dranishnikov and T.~Januszkiewicz.
\newblock Every {C}oxeter group acts amenably on a compact space.
\newblock In {\em Proceedings of the 1999 {T}opology and {D}ynamics
  {C}onference ({S}alt {L}ake {C}ity, {UT})}, volume~24, pages 135--141, 1999.

\bibitem[DK18]{DK18}
Cornelia Dru\c{t}u and Michael Kapovich.
\newblock {\em Geometric group theory}, volume~63 of {\em American Mathematical
  Society Colloquium Publications}.
\newblock American Mathematical Society, Providence, RI, 2018.
\newblock With an appendix by Bogdan Nica.

\bibitem[DS05]{DS05}
Cornelia Dru\c{t}u and Mark Sapir.
\newblock Tree-graded spaces and asymptotic cones of groups.
\newblock {\em Topology}, 44(5):959--1058, 2005.
\newblock With an appendix by Denis Osin and Mark Sapir.

\bibitem[FL08]{FL08}
Thomas Foertsch and Alexander Lytchak.
\newblock The de {R}ham decomposition theorem for metric spaces.
\newblock {\em Geom. Funct. Anal.}, 18(1):120--143, 2008.

\bibitem[GP16]{GP16}
Victor Gerasimov and Leonid Potyagailo.
\newblock Quasiconvexity in relatively hyperbolic groups.
\newblock {\em J. Reine Angew. Math.}, 710:95--135, 2016.

\bibitem[Ham01]{Hamilton}
Emily Hamilton.
\newblock Abelian subgroup separability of {H}aken 3-manifolds and closed
  hyperbolic {$n$}-orbifolds.
\newblock {\em Proc. London Math. Soc. (3)}, 83(3):626--646, 2001.

\bibitem[HO13]{HO}
Michael Hull and Denis Osin.
\newblock Induced quasicocycles on groups with hyperbolically embedded
  subgroups.
\newblock {\em Algebr. Geom. Topol.}, 13(5):2635--2665, 2013.

\bibitem[HP15]{HP15}
Mark~F. Hagen and Piotr Przytycki.
\newblock Cocompactly cubulated graph manifolds.
\newblock {\em Israel J. Math.}, 207(1):377--394, 2015.

\bibitem[HRSS22]{HRSS22}
Mark Hagen, Jacob Russell, Alessandro Sisto, and Davide Spriano.
\newblock Equivariant hierarchically hyperbolic structures for 3-manifold
  groups via quasimorphisms, 2022.

\bibitem[HW08]{HW08}
Fr\'{e}d\'{e}ric Haglund and Daniel~T. Wise.
\newblock Special cube complexes.
\newblock {\em Geom. Funct. Anal.}, 17(5):1551--1620, 2008.

\bibitem[KL96]{KL96}
Michael Kapovich and Bernhard Leeb.
\newblock Actions of discrete groups on nonpositively curved spaces.
\newblock {\em Math. Ann.}, 306(2):341--352, 1996.

\bibitem[KL98]{KL98}
M.~Kapovich and B.~Leeb.
\newblock {$3$}-manifold groups and nonpositive curvature.
\newblock {\em Geom. Funct. Anal.}, 8(5):841--852, 1998.

\bibitem[Lee95]{B95}
Bernhard Leeb.
\newblock {$3$}-manifolds with(out) metrics of nonpositive curvature.
\newblock {\em Invent. Math.}, 122(2):277--289, 1995.

\bibitem[Liu17]{Liu17}
Yi~Liu.
\newblock A characterization of virtually embedded subsurfaces in 3-manifolds.
\newblock {\em Trans. Amer. Math. Soc.}, 369(2):1237--1264, 2017.

\bibitem[Man05]{M05}
Jason~Fox Manning.
\newblock Geometry of pseudocharacters.
\newblock {\em Geom. Topol.}, 9:1147--1185, 2005.

\bibitem[Man06]{M06}
J.~F. Manning.
\newblock Quasi-actions on trees and property ({QFA}).
\newblock {\em J. London Math. Soc. (2)}, 73(1):84--108, 2006.
\newblock With an appendix by N. Monod and B. R\'{e}my.

\bibitem[MM00]{MM00}
H.~A. Masur and Y.~N. Minsky.
\newblock Geometry of the complex of curves. {II}. {H}ierarchical structure.
\newblock {\em Geom. Funct. Anal.}, 10(4):902--974, 2000.

\bibitem[MS13]{MS13}
John~M. Mackay and Alessandro Sisto.
\newblock Embedding relatively hyperbolic groups in products of trees.
\newblock {\em Algebr. Geom. Topol.}, 13(4):2261--2282, 2013.

\bibitem[NS20]{NS20}
Hoang~Thanh Nguyen and Hongbin Sun.
\newblock Subgroup distortion of 3-manifold groups.
\newblock {\em Trans. Amer. Math. Soc.}, 373(9):6683--6711, 2020.

\bibitem[NTY21]{NTY19}
Hoang~Thanh Nguyen, Hung~Cong Tran, and Wenyuan Yang.
\newblock Quasiconvexity in 3-manifold groups.
\newblock {\em Math. Ann.}, 381(1-2):405--437, 2021.

\bibitem[NY]{NY20}
Hoang~Thanh {Nguyen} and Wenyuan {Yang}.
\newblock Croke-kleiner admissible groups: Property (qt) and quasiconvexity.
\newblock To appear in Michigan Math. J.

\bibitem[Osi16]{Osin}
D.~Osin.
\newblock Acylindrically hyperbolic groups.
\newblock {\em Trans. Amer. Math. Soc.}, 368(2):851--888, 2016.

\bibitem[Pau05]{P05}
G.~Paulik.
\newblock {\em Gluing spaces and analysis}.
\newblock Bonner Mathematische Schriften [Bonn Mathematical Publications], 372.
  Universit\"at Bonn, Mathematisches Institut, Bonn, 2005.
\newblock Dissertation, Rheinische Friedrich-Wilhelms-Universit{\"a}t Bonn,
  Bonn, 2005.

\bibitem[PW18]{PW18}
Piotr Przytycki and Daniel~T. Wise.
\newblock Mixed 3-manifolds are virtually special.
\newblock {\em J. Amer. Math. Soc.}, 31(2):319--347, 2018.

\bibitem[Rei18]{R18}
Alan~W. Reid.
\newblock Profinite rigidity.
\newblock In {\em Proceedings of the {I}nternational {C}ongress of
  {M}athematicians---{R}io de {J}aneiro 2018. {V}ol. {II}. {I}nvited lectures},
  pages 1193--1216. World Sci. Publ., Hackensack, NJ, 2018.

\bibitem[Sis13]{S13}
Alessandro Sisto.
\newblock Projections and relative hyperbolicity.
\newblock {\em Enseign. Math. (2)}, 59(1-2):165--181, 2013.

\bibitem[Sun20]{Sun20}
Hongbin Sun.
\newblock A characterization on separable subgroups of 3-manifold groups.
\newblock {\em J. Topol.}, 13(1):187--236, 2020.

\bibitem[{Sun}21]{Sun21}
Hongbin {Sun}.
\newblock {All finitely generated 3-manifold groups are Grothendieck rigid}.
\newblock {\em arXiv e-prints}, page arXiv:2103.00547, February 2021.

\bibitem[Tid18]{Tid18}
Joseph Tidmore.
\newblock Cocompact cubulations of mixed 3-manifolds.
\newblock {\em Groups Geom. Dyn.}, 12(4):1429--1460, 2018.

\bibitem[Wis00]{Wise00}
Daniel~T. Wise.
\newblock Subgroup separability of graphs of free groups with cyclic edge
  groups.
\newblock {\em Q. J. Math.}, 51(1):107--129, 2000.

\bibitem[Wis20]{Wise20}
Daniel~T. Wise.
\newblock {\em The Structure of Groups with a Quasiconvex Hierarchy}, volume
  AMS-209.
\newblock Annals of Mathematics Studies, 2020.

\bibitem[Yan19]{Yang}
Wen-yuan Yang.
\newblock Statistically convex-cocompact actions of groups with contracting
  elements.
\newblock {\em Int. Math. Res. Not. IMRN}, (23):7259--7323, 2019.

\end{thebibliography}
\end{document}